\documentclass[a4paper]{amsart}

\usepackage[backrefs,lite]{amsrefs}
\usepackage{url}

\makeatletter
\DeclareFontEncoding{LS1}{}{}
\DeclareFontSubstitution{LS1}{stix}{m}{n}
\DeclareMathAlphabet{\skr}{LS1}{stixscr}{m}{n}
\makeatother

\usepackage{amssymb}              
\usepackage{wasysym}              
\usepackage{lmodern,bm}              
\usepackage{stmaryrd}
\usepackage[all]{xy}
\usepackage{tikz}
\usepackage{graphicx}
\usetikzlibrary{calc} 

\usepackage{amsmath,amssymb,amsthm}
\usepackage{tikz, pgfplots, pgfplotstable}
\usetikzlibrary{calc}
\usetikzlibrary{decorations.markings}
\usetikzlibrary{positioning,arrows}

\CompileMatrices

\newtheorem{theorem}{Theorem}[section]

\newtheorem{lemma}[theorem]{Lemma}
\newtheorem{proposition}[theorem]{Proposition}
\newtheorem{corollary}[theorem]{Corollary}

\theoremstyle{definition}

\newtheorem{definition}[theorem]{Definition}
\newtheorem{construction}[theorem]{Construction}

\newtheorem{example}[theorem]{Example}
\newtheorem{remark}[theorem]{Remark}

\newtheorem*{thanx}{Acknowledgement}

\numberwithin{equation}{theorem}

\def\Cl{{\rm Cl}}
\def\cl{{\rm cl}}

\def\CC{{\mathbb C}}
\def\KK{{\mathbb K}}
\def\TT{{\mathbb T}}
\def\ZZ{{\mathbb Z}}

\def\QQ{{\mathbb Q}}
\def\PP{{\mathbb P}}

\def\Aut{{\rm Aut}}

\def\Spec{{\rm Spec}}

\def\cone{{\rm cone}}

\def\neuneinsnull{{\tiny
$
\setlength{\arraycolsep}{1.8pt}
\begin{array}{ccc}
\begin{array}{l}
\text{(9-1-0)}
\\[3pt]
\mathcal{K}^2=9
\end{array}
&
\left[
\begin{array}{ccc}
x_0^2 & x_1^2 & x_2^2
\end{array}
\right]
&
\begin{array}{l}
S(9)
\end{array}                  
\end{array}
$
}}

\def\achteinsnull{{\tiny
$
\setlength{\arraycolsep}{1.8pt}
\begin{array}{ccc}
\begin{array}{l}
\text{(8-1-0)}
\\[3pt]
\mathcal{K}^2=8
\end{array}
&
\left[
\begin{array}{ccc}
x_0^2 & x_1^2 & 2x_2^2
\end{array}
\right]
&
\begin{array}{l}
S(8)
\end{array}                  
\end{array}
$
}}

\def\sechseinsnull{{\tiny
$
\setlength{\arraycolsep}{1.8pt}
\begin{array}{ccc}
\begin{array}{l}
\text{(6-1-0)}
\\[3pt]
\mathcal{K}^2=6
\end{array}
&
\left[
\begin{array}{ccc}
x_0^2 & 2x_1^2 & 3x_2^2
\end{array}
\right]
&
\begin{array}{l}
S(6)
\end{array}                  
\end{array}
$
}}

\def\fuenfeinsnull{{\tiny
$
\setlength{\arraycolsep}{1.8pt}
\begin{array}{ccc}
\begin{array}{l}
\text{(5-1-0)}
\\[3pt]
\mathcal{K}^2=5
\end{array}
&
\left[
\begin{array}{ccc}
x_0^2 & x_1^2 & 5x_2^2
\end{array}
\right]
&
\begin{array}{l}
S(5)
\end{array}                  
\end{array}
$
}}

\def\vierzweieins{{\tiny
$
\setlength{\arraycolsep}{1.8pt}
\begin{array}{ccc}
\begin{array}{l}
\text{(4-2-1)}
\\[3pt]
\mathcal{K}^2=4
\end{array}
&
\left[
\begin{array}{ccc}
x_0^2 & x_1^2 & 2x_2^2
\\[3pt]
\bar 0 & \bar 1 & \bar 1  
\end{array}
\right]
&
\begin{array}{l}
S(8)
\\[3pt]
\ZZ/2\ZZ
\end{array}                  
\end{array}
$
}}

\def\dreidreizwei{{\tiny
$
\setlength{\arraycolsep}{1.8pt}
\begin{array}{ccc}
\begin{array}{l}
\text{(3-3-2)}
\\[3pt]
\mathcal{K}^2=3
\end{array}
&
\left[
\begin{array}{ccc}
x_0^2 & x_1^2 & x_2^2
\\[3pt]
\bar 0 & \bar 1 & \bar 2  
\end{array}
\right]
&
\begin{array}{l}
S(9)
\\[3pt]
\ZZ/3\ZZ
\end{array}                  
\end{array}
$
}}

\def\dreizweieins{{\tiny
$
\setlength{\arraycolsep}{1.8pt}
\begin{array}{ccc}
\begin{array}{l}
\text{(3-2-1)}
\\[3pt]
\mathcal{K}^2=3
\end{array}
&
\left[
\begin{array}{ccc}
x_0^2 & 2x_1^2 & 3x_2^2
\\[3pt]
\bar 0 & \bar 1 & \bar 1  
\end{array}
\right]
&
\begin{array}{l}
S(6)
\\[3pt]
\ZZ/2\ZZ
\end{array}                  
\end{array}
$
}}

\def\zweiviereins{{\tiny
$
\setlength{\arraycolsep}{1.8pt}
\begin{array}{ccc}
\begin{array}{l}
\text{(2-4-1)}
\\[3pt]
\mathcal{K}^2=2
\end{array}
&
\left[
\begin{array}{ccc}
x_0^2 & x_1^2 & 2x_2^2
\\[3pt]
\bar 0 & \bar 1 & \bar 1  
\end{array}
\right]
&
\begin{array}{l}
S(8)
\\[3pt]
\ZZ/4\ZZ
\end{array}                  
\end{array}
$
}}

\def\zweivierdrei{{\tiny
$
\setlength{\arraycolsep}{1.8pt}
\begin{array}{ccc}
\begin{array}{l}
\text{(2-4-3)}
\\[3pt]
\mathcal{K}^2=2
\end{array}
&
\left[
\begin{array}{ccc}
x_0^2 & x_1^2 & 2x_2^2
\\[3pt]
\bar 0 & \bar 1 & \bar 3  
\end{array}
\right]
&
\begin{array}{l}
S(8)
\\[3pt]
\ZZ/4\ZZ
\end{array}                  
\end{array}
$
}}

\def\zweidreieins{{\tiny
$
\setlength{\arraycolsep}{1.8pt}
\begin{array}{ccc}
\begin{array}{l}
\text{(2-3-1)}
\\[3pt]
\mathcal{K}^2=2
\end{array}
&
\left[
\begin{array}{ccc}
x_0^2 & 2x_1^2 & 3x_2^2
\\[3pt]
\bar 0 & \bar 1 & \bar 1  
\end{array}
\right]
&
\begin{array}{l}
S(6)
\\[3pt]
\ZZ/3\ZZ
\end{array}                  
\end{array}
$
}}

\def\zweidreizwei{{\tiny
$
\setlength{\arraycolsep}{1.8pt}
\begin{array}{ccc}
\begin{array}{l}
\text{(2-3-2)}
\\[3pt]
\mathcal{K}^2=2
\end{array}
&
\left[
\begin{array}{ccc}
x_0^2 & 2x_1^2 & 3x_2^2
\\[3pt]
\bar 0 & \bar 1 & \bar 2
\end{array}
\right]
&
\begin{array}{l}
S(6)
\\[3pt]
\ZZ/3\ZZ
\end{array}                  
\end{array}
$
}}

\def\einsneunzwei{{\tiny
$
\setlength{\arraycolsep}{1.8pt}
\begin{array}{ccc}
\begin{array}{l}
\text{(1-9-2)}
\\[3pt]
\mathcal{K}^2=1
\end{array}
&
\left[
\begin{array}{ccc}
x_0^2 & x_1^2 & x_2^2
\\[3pt]
\bar 0 & \bar 1 & \bar 2
\end{array}
\right]
&
\begin{array}{l}
S(9)
\\[3pt]
\ZZ/9\ZZ
\end{array}                  
\end{array}
$
}}

\def\einsneunfuenf{{\tiny
$
\setlength{\arraycolsep}{1.8pt}
\begin{array}{ccc}
\begin{array}{l}
\text{(1-9-5)}
\\[3pt]
\mathcal{K}^2=1
\end{array}
&
\left[
\begin{array}{ccc}
x_0^2 & x_1^2 & x_2^2
\\[3pt]
\bar 0 & \bar 1 & \bar 5
\end{array}
\right]
&
\begin{array}{l}
S(9)
\\[3pt]
\ZZ/9\ZZ
\end{array}                  
\end{array}
$
}}

\def\einsneunacht{{\tiny
$
\setlength{\arraycolsep}{1.8pt}
\begin{array}{ccc}
\begin{array}{l}
\text{(1-9-8)}
\\[3pt]
\mathcal{K}^2=1
\end{array}
&
\left[
\begin{array}{ccc}
x_0^2 & x_1^2 & x_2^2
\\[3pt]
\bar 0 & \bar 1 & \bar 8
\end{array}
\right]
&
\begin{array}{l}
S(9)
\\[3pt]
\ZZ/9\ZZ
\end{array}                  
\end{array}
$
}}

\def\einsachtdrei{{\tiny
$
\setlength{\arraycolsep}{1.8pt}
\begin{array}{ccc}
\begin{array}{l}
\text{(1-8-3)}
\\[3pt]
\mathcal{K}^2=1
\end{array}
&
\left[
\begin{array}{ccc}
x_0^2 & x_1^2 & 2x_2^2
\\[3pt]
\bar 0 & \bar 1 & \bar 3
\end{array}
\right]
&
\begin{array}{l}
S(8)
\\[3pt]
\ZZ/8\ZZ
\end{array}                  
\end{array}
$
}}

\def\einsachtfuenf{{\tiny
$
\setlength{\arraycolsep}{1.8pt}
\begin{array}{ccc}
\begin{array}{l}
\text{(1-8-5)}
\\[3pt]
\mathcal{K}^2=1
\end{array}
&
\left[
\begin{array}{ccc}
x_0^2 & x_1^2 & 2x_2^2
\\[3pt]
\bar 0 & \bar 1 & \bar 5
\end{array}
\right]
&
\begin{array}{l}
S(8)
\\[3pt]
\ZZ/8\ZZ
\end{array}                  
\end{array}
$
}}

\def\einsachteins{{\tiny
$
\setlength{\arraycolsep}{1.8pt}
\begin{array}{ccc}
\begin{array}{l}
\text{(1-8-1)}
\\[3pt]
\mathcal{K}^2=1
\end{array}
&
\left[
\begin{array}{ccc}
x_0^2 & x_1^2 & 2x_2^2
\\[3pt]
\bar 0 & \bar 1 & \bar 1
\end{array}
\right]
&
\begin{array}{l}
S(8)
\\[3pt]
\ZZ/8\ZZ
\end{array}                  
\end{array}
$
}}

\def\einsachtsieben{{\tiny
$
\setlength{\arraycolsep}{1.8pt}
\begin{array}{ccc}
\begin{array}{l}
\text{(1-8-7)}
\\[3pt]
\mathcal{K}^2=1
\end{array}
&
\left[
\begin{array}{ccc}
x_0^2 & x_1^2 & 2x_2^2
\\[3pt]
\bar 0 & \bar 1 & \bar 7
\end{array}
\right]
&
\begin{array}{l}
S(8)
\\[3pt]
\ZZ/8\ZZ
\end{array}                  
\end{array}
$
}}

\def\einssechseins{{\tiny
$
\setlength{\arraycolsep}{1.8pt}
\begin{array}{ccc}
\begin{array}{l}
\text{(1-6-1)}
\\[3pt]
\mathcal{K}^2=1
\end{array}
&
\left[
\begin{array}{ccc}
x_0^2 & 2x_1^2 & 3x_2^2
\\[3pt]
\bar 0 & \bar 1 & \bar 1
\end{array}
\right]
&
\begin{array}{l}
S(6)
\\[3pt]
\ZZ/6\ZZ
\end{array}                  
\end{array}
$
}}

\def\einssechsfuenf{{\tiny
$
\setlength{\arraycolsep}{1.8pt}
\begin{array}{ccc}
\begin{array}{l}
\text{(1-6-5)}
\\[3pt]
\mathcal{K}^2=1
\end{array}
&
\left[
\begin{array}{ccc}
x_0^2 & 2x_1^2 & 3x_2^2
\\[3pt]
\bar 0 & \bar 1 & \bar 5
\end{array}
\right]
&
\begin{array}{l}
S(6)
\\[3pt]
\ZZ/6\ZZ
\end{array}                  
\end{array}
$
}}

\def\einsfuenfeins{{\tiny
$
\setlength{\arraycolsep}{1.8pt}
\begin{array}{ccc}
\begin{array}{l}
\text{(1-5-1)}
\\[3pt]
\mathcal{K}^2=1
\end{array}
&
\left[
\begin{array}{ccc}
x_0^2 & x_1^2 & 5x_2^2
\\[3pt]
\bar 0 & \bar 1 & \bar 1
\end{array}
\right]
&
\begin{array}{l}
S(5)
\\[3pt]
\ZZ/5\ZZ
\end{array}                  
\end{array}
$
}}

\def\einsfuenfzwei{{\tiny
$
\setlength{\arraycolsep}{1.8pt}
\begin{array}{ccc}
\begin{array}{l}
\text{(1-5-2)}
\\[3pt]
\mathcal{K}^2=1
\end{array}
&
\left[
\begin{array}{ccc}
x_0^2 & x_1^2 & 5x_2^2
\\[3pt]
\bar 0 & \bar 1 & \bar 2
\end{array}
\right]
&
\begin{array}{l}
S(5)
\\[3pt]
\ZZ/5\ZZ
\end{array}                  
\end{array}
$
}}

\def\einsfuenfdrei{{\tiny
$
\setlength{\arraycolsep}{1.8pt}
\begin{array}{ccc}
\begin{array}{l}
\text{(1-5-3)}
\\[3pt]
\mathcal{K}^2=1
\end{array}
&
\left[
\begin{array}{ccc}
x_0^2 & x_1^2 & 5x_2^2
\\[3pt]
\bar 0 & \bar 1 & \bar 3
\end{array}
\right]
&
\begin{array}{l}
S(5)
\\[3pt]
\ZZ/5\ZZ
\end{array}                  
\end{array}
$
}}

\def\einsfuenfvier{{\tiny
$
\setlength{\arraycolsep}{1.8pt}
\begin{array}{ccc}
\begin{array}{l}
\text{(1-5-4)}
\\[3pt]
\mathcal{K}^2=1
\end{array}
&
\left[
\begin{array}{ccc}
x_0^2 & x_1^2 & 5x_2^2
\\[3pt]
\bar 0 & \bar 1 & \bar 4
\end{array}
\right]
&
\begin{array}{l}
S(5)
\\[3pt]
\ZZ/5\ZZ
\end{array}                  
\end{array}
$
}}

\title[$\KK^*$-surfaces of Picard number one and integral degree]{$\KK^*$-surfaces of Picard number one \\ and integral degree}

\author[J\"urgen Hausen, Katharina Kir\'{a}ly]{J\"urgen Hausen, Katharina Kir\'{a}ly}

\address{Mathematisches Institut, Universit\"at T\"ubingen,
Auf der Morgenstelle 10, 72076 T\"ubingen, Germany}
\email{juergen.hausen@uni-tuebingen.de}

\address{Mathematisches Institut, Universit\"at T\"ubingen,
Auf der Morgenstelle 10, 72076 T\"ubingen, Germany}
\email{kaki@math.uni-tuebingen.de}


\pgfplotsset{compat=1.18}

\begin{document}

\begin{abstract}
We give an explicit description of all quasismooth, rational,
projective surfaces of Picard number one that admit
a non-trivial torus action and have an integral canonical self
intersection number.
\end{abstract}

\maketitle


\section{Introduction}

We provide an explicit description of all quasismooth,
rational, projective surfaces of Picard number one
and integral degree $\mathcal{K}^2$ that admit a
non-trivial, effective action of an algebraic torus $\TT$;
we work over an algebraically closed field $\KK$ of
characteristic zero.
Our description involves the solutions of the
\emph{squared Markov type equation},
that means the triples $u = (u_0,u_1,u_2)$ of positive integers
satisfying
\[ (u_0+u_1+u_2)^2 \ = \ au_0u_1u_2, \]
where $a$ is a positive integer. The solution set
$S(a) \subseteq \ZZ_{>0}^3$ of these equations
is non-empty only for $a = 1,2,3,4,5,6,8,9$ and
the solutions $u \in S(a)$ admit an explicit
stepwise construction; see 
Section~\ref{sec:squaredmarkov} for an elementary
treatment.

If the torus $\TT$ is of dimension two, then we are
concerned with a projective toric surface $Z$
of Picard number one, also called a
\emph{fake weighted projective plane}.
Any such $Z$ has divisor class group 
$\ZZ \oplus \ZZ / \mu \ZZ$ and we can encode $Z$
in terms of its \emph{degree matrix $Q$}, listing
the classes of the three coordinate divisors as
columns:
\[
Q
\ = \
[q_0,q_1,q_2]
\ = \
\left[
\begin{array}{ccc}
u_0 & u_1 & u_2
\\
\bar \eta_0 & \bar \eta_1 & \bar \eta_2
\end{array}
\right],
\qquad
q_i = (u_i,\bar \eta_i) \in \ZZ \oplus \ZZ / \mu \ZZ.
\]
If the degree $a = \mathcal{K}^2$ of $Z$ is an integer, then we obtain
$u \in S(\mu a)$ for the first row of~$Q$.
That means in particular that, up to reordering, the triple
$u$ is of the shape
\[
u
\ = \  
(x_0^2,\xi_1 x_1^2, \xi_2 x_2^2),
\qquad
x_0,x_1,x_2 \in \ZZ,
\qquad
(\xi_1,\xi_2)
\ = \
\begin{cases}
(1,1), & \mu a = 9,
\\
(1,2), & \mu a = 8,
\\
(2,3), & \mu a = 6,
\\
(1,5), & \mu a = 5.
\end{cases}
\]
Our first result explicitly describes all fake weighted
projective planes of integral degree $\mathcal{K}^2$ by
means of 24 infinite series of degree matrices;
note that all members of a series  with ID ($a$-$\mu$-$\eta$)
have the degree $\mathcal{K}^2 = a$ and the order $\mu$ of the
torsion part of their divisor class group in common.

\bigskip

\goodbreak

\begin{theorem}
\label{thm:introthm1}
Let $Z$ be a fake weighted projective plane of integral 
degree.
Then $Z \cong Z(Q)$ for a degree matrix $Q$ from
the following data table:

\medskip

\begin{center}
\begin{tabular}{c|c|c}
&& \\
\hline  
&& \\[-5pt]
\neuneinsnull
&
\achteinsnull
&
\sechseinsnull
\\ 
&& \\[-5pt] 
\hline  
&& \\[-5pt] 
\fuenfeinsnull
&
\vierzweieins
&
\dreidreizwei
\\ 
&& \\[-5pt] 
\hline  
&& \\[-5pt] 
\dreizweieins
&
\zweiviereins
&
\zweivierdrei
\\ 
&& \\[-5pt] 
\hline  
&& \\[-5pt] 
\zweidreieins
&
\zweidreizwei
&
\einsneunzwei
\\
&& \\[-5pt] 
\hline  
&& \\[-5pt]
\einsneunfuenf
&
\einsneunacht
&
\einsachteins
\\
&& \\[-5pt] 
\hline  
&& \\[-5pt]
\einsachtdrei
&
\einsachtfuenf
&
\einsachtsieben
\\
&& \\[-5pt] 
\hline  
&& \\[-5pt]
\einssechseins
&
\einssechsfuenf
&
\einsfuenfeins
\\
&& \\[-5pt] 
\hline  
&& \\[-5pt]
\einsfuenfzwei
&
\einsfuenfdrei
&
\einsfuenfvier
\\
&& \\[-5pt] 
\hline  
&& \\[-5pt]
\end{tabular}
\end{center}

\medskip

\noindent
Moreover, each of the listed degree matrices defines a fake
weighted projective plane of the degree given in its data package.
Finally, any two distinct degree matrices from the table define
non-isomorphic fake weighted projective planes, except they stem
from (1-9-$\ast$) or (1-8-$\ast$) and are both taken from one
of the sets:

\medskip

\begin{center}
{\tiny
$
\setlength{\arraycolsep}{1.8pt}
\left\{   
\left[
\begin{array}{ccc}
1 & 1 & 1
\\[3pt]
\bar 0 & \bar 1 & \bar 2
\end{array}
\right],
\
\left[
\begin{array}{ccc}
1 & 1 & 1
\\[3pt]
\bar 0 & \bar 1 & \bar 5
\end{array}
\right],
\
\left[
\begin{array}{ccc}
1 & 1 & 1
\\[3pt]
\bar 0 & \bar 1 & \bar 8
\end{array}
\right]
\right\}
$,
\hspace{.5cm}
$
\setlength{\arraycolsep}{1.8pt}
\left\{   
\left[
\begin{array}{ccc}
1 & 1 & 4
\\[3pt]
\bar 0 & \bar 1 & \bar 5
\end{array}
\right],
\
\left[
\begin{array}{ccc}
1 & 1 & 4
\\[3pt]
\bar 0 & \bar 1 & \bar 8
\end{array}
\right]
\right\}
$,
\hspace{.5cm}
$
\setlength{\arraycolsep}{1.8pt}
\left\{   
\left[
\begin{array}{ccc}
1 & 1 & 2
\\[3pt]
\bar 0 & \bar 1 & \bar 3
\end{array}
\right],
\
\left[
\begin{array}{ccc}
1 & 1 & 2
\\[3pt]
\bar 0 & \bar 1 & \bar 7
\end{array}
\right]
\right\}
$.
}
\end{center}

\medskip

\end{theorem}

The proof of Theorem~\ref{thm:introthm1} is provided
in Section~\ref{sec:classifyfwpp} and is split according
to the possible degrees $\mathcal{K}^2$; see
Propositions~\ref{prop:classfwppdeg5689},
\ref{prop:classfwppdeg4}, \ref{prop:classfwppdeg3},
\ref{prop:classfwppdeg2} and \ref{prop:classfwppdeg1}.
The arguments essentially use the unique encoding
of fake weighted projective planes via
\emph{adjusted degree matrices} as provided by
Proposition~\ref{rem:adjdegmat} and, as well,
the various divisibility properties of the entries
of the solution triples of the squared Markov type equations
observed in Sections~\ref{sec:squaredmarkov}
and~\ref{sec:classifyfwpp}.

In Section~\ref{sec:locgi}, we take a closer look at the
possible singularities of our fake weighted projective planes;
by normality, these are at most the three toric
fixed points.
Recall that a \emph{(cyclic) $T$-singularity} is a
quotient singularity $\KK^2/C(dk^2)$ by the subgroup
$C(dk^2) \subseteq \KK^*$ of the $dk^2$-th roots
of unity, acting via
\[
\zeta \cdot (z_1,z_2) \ = \ (\zeta z_1, \zeta^{dpk-1} z_2),
\]  
where $p$ and $k$ are coprime.
Propositions~\ref{prop:gisingdeg9865} to~\ref{prop:gisingdeg1}
provide the constellations of non-trivial local Gorenstein
indices and possible $T$-singularities for all fake weighted
projective planes of integral degree.
As an immediate consequence, one obtains the following.

\goodbreak

\begin{theorem}
\label{thm:introthm2}
Let $Z$ be a fake weighted projective plane of integral
degree.
If $Z$ is isomorphic to a member of one of the series
\begin{center}
(2-3-1), \ (1-8-1), \ (1-8-5), \ (1-6-1), \ (1-5-1), \ (1-5-2), \ (1-5-3),
\end{center}  
then $Z$ has three singularities and precisely one of them
is a $T$-singularity.
Otherwise, the surface $Z$ has at most $T$-singularities.
\end{theorem}

The class of the fake weighted projective planes of integral
degree contains in particular all fake weighted projective
planes having at most $T$-singularities; see for
instance~\cite[Prop.~2.6]{HaPro}.
Thus, the classification results of Section~\ref{sec:locgi}
complement the classification~\cite[Thm.~4.1]{HaPro}
of $T$-singular projective toric surfaces of Picard number
one in the sense that we add three more series, namely
(1-9-2), (1-9-5) and (1-8-3), to Table~2 given there;
see also Examples~\ref{ex:iso-19-1425} and~\ref{ex:iso-18-192}.

We turn to the $\KK^*$-surfaces of Picard number one and
integral degree.
We show that these naturally correspond
to certain pairs of fake weighted projective planes from
Theorem~\ref{thm:introthm1}.
Every quasismooth, rational, projective $\KK^*$-surface
$X$ of Picard number one admits certain toric
degenerations $Z_1 \reflectbox{$\leadsto$} X \hbox{$\leadsto$} Z_2$,
where $Z_1$ and $Z_2$ are fake weighted projective planes
of the same degree as $X$; see Construction~\ref{constr:degs}
and Proposition~\ref{prop:degprops}.
We will call two fake weighted projective planes
\emph{adjacent} if they arise as such degenerations
from a common $\KK^*$-surface.
Theorem~\ref{thm:adjpartner} ensures that every fake
weighted projective plane $Z_1$ of integral degree
admits an adjacent partner $Z_2$. 
A more concise formulation of adjacency in terms of
degree matrices and supporting notions
are given in Definitions~\ref{def:Qmatadj}
and~\ref{def:ordnontor}.
This allows us to assign to any pair $(Q_1,Q_2)$
of adjacent degree matrices a $\KK^*$-surface $X(Q_1,Q_2)$
and leads to the following.

\begin{theorem}
\label{thm:introthm3}
Let $X$ be a non-toric, quasismooth, rational, projective
$\KK^*$-surface of Picard number one with $\mathcal{K}_X^2 \in \ZZ$. 
Then $X \cong X(Q_1,Q_2)$ with a non-toric, ordered
pair of adjacent degree matrices.
Moreover, distinct ordered pairs $(Q_1,Q_2)$ of adjacent
degree matrices yield non-isomorphic $\KK^*$-surfaces.
\end{theorem}

\tableofcontents


\section{Squared Markov type equations}
\label{sec:squaredmarkov}

Given any number $a \in \ZZ_{>0}$, the associated
\emph{squared Markov type equation} in the variables
$w_0,w_1,w_2$ is
\[
(w_0+w_1+w_2)^2 \ = \ aw_0w_1w_2.
\]
By a \emph{solution} we mean a triple
$u = (u_0,u_1,u_2) \in \ZZ_{>0}^3$
satisfying this equation and we denote by
$S(a) \subseteq \ZZ^3_{>0}$ the set of
these solutions.

The set $S(a)$ is in bijection with
the solution set of the corresponding usual
Markov type equation, studied exhaustively
by Hurwitz~\cite{Hu}; see Remark~\ref{rem:usualmarkov}.
We expect everything of this section to be
known; see~e.g.~\cite[Thm.~11]{GyMa} for the case $a=9$.
However, for the moment we can't serve with appropriate
references and hence allow ourselves to provide an
elementary self-contained treatment according to our needs.

\begin{lemma}
\label{lem:mutation}
Let $a \in \ZZ_{>0}$. Then we obtain an involution
on the set $S(a)$ of solutions of the squared Markov type
equation by 
\[
\lambda \colon S(a) \ \to \ S(a),
\quad
u \ \mapsto \
\left(u_0, u_1, \tfrac{(u_0+u_1)^2}{u_2}\right)
\ = \
\left(u_0, u_1, a u_0 u_1  - 2 u_0 -2 u_1 - u_2 \right).
\]
\end{lemma}

\begin{proof}
Let $u  \in S(a)$.
We show that the two representations of $\lambda(u)$ coincide.
This merely means to compare the third components:
\[
a u_0 u_1  - 2 u_0 -2 u_1 - u_2 
=
\tfrac{(u_0 + u_1+u_2)^2}{u_2} - 2 u_0 -2 u_1 - u_2 
=
\tfrac{(u_0 + u_1)^2}{u_2}.
\]
In particular, $\lambda$ maps solutions to positive integer tuples. 
We prove that $u' := \lambda(u)$ is a solution.
With $\hat{u} := u_0+u_1$, we obtain
\[
a = 
\tfrac{(u_0+ u_1 + u_2)^2}{u_0 u_1 u_2} 
= 
\tfrac{(\hat{u} + u_2)^2}{u_0 u_1 u_2}
=
\tfrac{\left(\hat{u} + \frac{\hat{u}^2}{u_2}\right)^2}{u_0 u_1 \frac{\hat{u}^2}{u_2}}
=
\tfrac{(u_0' + u_1' + u_2')^2}{u_0' u_1' u_2'}. 
\]
\end{proof}

A \emph{one-step mutation} of a triple $u  \in S(a)$ is
a permutation of its entries followed
by the operation~$\lambda$.
A \emph{mutation} of $u$ is a composition
of one-step mutations.

\begin{theorem}
\label{thm:intmarkov}
Fix $a \in \ZZ_{>0}$ such that the equation
$(w_0+w_1+w_2)^2 = aw_0w_1w_2$ admits a solution.
Then we have
\begin{equation*}
a \ \in \ \{1,2,3,4,5,6,8,9\}.
\end{equation*} 
For these $a$, the solutions of the above equation
are precisely the mutations of the following triples
\begin{equation*}
\begin{array}{lllll}
\displaystyle\left(\tfrac{9}{a},\tfrac{9}{a},\tfrac{9}{a}\right),
&
a=1,3,9, 
&&
\displaystyle\left(\tfrac{8}{a},\tfrac{8}{a},\tfrac{16}{a}\right),
&
a=1,2,4,8, 
\\[15pt]
\displaystyle\left(\tfrac{6}{a},\tfrac{12}{a},\tfrac{18}{a}\right),
&
a=1,2,3,6,
&&                                                      
\displaystyle\left(\tfrac{5}{a},\tfrac{20}{a},\tfrac{25}{a}\right),
&
a=1,5. 
\end{array}
\end{equation*} 
\end{theorem}

By the \emph{norm} of a triple~$u \in \ZZ^3_{> 0}$, we mean the
number~$\nu(u) := u_0 + u_1 + u_2$.
Obviously the norm is invariant under coordinate permutations. 
We call a triple $u \in S(a)$ \emph{initial} if
$ u_0 \leq u_1 \leq u_2$ and $u_2 \leq u_0 + u_1$.

\begin{lemma}
\label{lem:norm}
Consider a solution $u \in S(a)$. Set~$u' := \lambda(u)$. 
Then we have the following equivalences:
\begin{align*}
\nu(u ) = \nu(u') \ & \Longleftrightarrow \ u_2  = u_2' \ \Longleftrightarrow \ u_2 = u_0 + u_1,
\\%
\nu(u ) < \nu(u') \ & \Longleftrightarrow \ u_2  < u_2' \ \Longleftrightarrow \ u_2 < u_0 + u_1.
\end{align*}
In particular, an ascendingly ordered $u\in S(a)$ is initial if and only
if $\nu(u) \leq \nu(\tilde{u})$ for any one-step mutation $\tilde{u}$ of $u$.
\end{lemma}

\begin{proof}
By the definition of $\lambda$, the vectors $u$ and $u'$
differ only in the last coordinate.
This gives the first equivalence in each row.
The remaining two equivalences follow from $u_2u_2'= (u_0+u_1)^2$.
\end{proof}

\begin{lemma}
\label{lem:discr}
Consider a solution $u \in S(a)$. 
\begin{enumerate}
\item
We have $\tfrac{1}{u_0} + \tfrac{1}{u_1} \leq \tfrac{a}{4}$.
\item
If $u$ is initial, then $\tfrac{a}{4} \leq \tfrac{3}{u_2} + \tfrac{1}{u_0}$.
\end{enumerate}
\end{lemma}

\begin{proof}
For~(i), write the squared Markov type equation as $f(w_2) = 0$ with
a polynomial $f \in \KK[w_0,w_1][w_2]$.
Then $f$ is of degree two in $w_2$ with discriminant
\[
\Delta(f)
\  = \
a^2 w_0^2 w_1^2 - 4 a w_0^2 w_1 - 4 a w_0 w_1^2
\ = \
4aw_0^2w_1^2\left(\tfrac{a}{4}-\tfrac{1}{w_0} - \tfrac{1}{w_1}\right).
\]
Since $u \in S(a)$ has positive integer entries, $\Delta(f)$
evaluates non-negatively at $u$. The assertion follows.
We turn to~(ii). The definition of an initial triple gives
\[
u_2 \leq u_0 + u_1, \quad u_0 + u_1 + u_2 \leq 2(u_0 +u_1), \quad u_0 \leq u_1 \leq u_2.
\]
We conclude
\[
\tfrac{a}{4}
\ =  \
\tfrac{1}{4} \tfrac{(u_0+u_1+u_2)^2}{u_0u_1u_2}
\ \leq \
\tfrac{(u_0+u_1)^2}{u_0u_1u_2}
\ = \
\tfrac{u_0}{u_1u_2} + \tfrac{2}{u_2} + \tfrac{u_1}{u_0u_2}
\ \leq \
\tfrac{3}{u_2} + \tfrac{1}{u_0}.
\]
\end{proof}

\begin{proof}[Proof of Theorem~\ref{thm:intmarkov}]
By Lemma~\ref{lem:norm} every solution is a mutation of an initial one.
So, it suffices to classify the initial solutions.
For $a \ge 5$, any initial triple $u \in S(a)$ satisfies
the inequality
\[
u_2 \ \leq \ \frac{12}{a-4},
\]
use Lemma~\ref{lem:discr}~(ii) and $1/u_0 \leq 1$.
This  effectively bounds the largest entry
$u_2$ of the initial triples $u \in S(a)$ for $a \ge 5$.
Concretely, we arrive at

\begin{center}
\begin{tabular}{c|c|c|c|c|c|c}
$a$ & $\ge 9$ & $8$ & $7$ & $6$ & $5$
\\ \hline
$u_2$ & $\le 2$ &  $\le 3$ &  $\le 4$ &  $\le 5$ &  $\le 6$ & 
\end{tabular}
\end{center}

\noindent
In each case, we check the (finitely many) ascendingly
ordered $u \in \ZZ_{>0}^3$ satisfying the respective
bound on $u_2$ for initial triples.
For $a \ge 10$ and $a=7$, there are no initial triples
and in each of the other cases, there is exactly one:

\begin{center}
\begin{tabular}{c|c|c|c|c|c}
$a$ & $9$ & $8$ & $6$ & $5$
\\ \hline
$u$ & $(1,1,1)$ &  $(1,1,2)$ &  $(1,2,3)$ &  $(1,4,5)$  
\end{tabular}
\end{center}

\noindent
\emph{Case $a = 4$}.
Here, Lemma~\ref{lem:discr}~(i) rules out the case $u_0 = 1$. 
Thus, $u_0 \ge 2$ and we obtain $u_2 \leq 6$, because otherwise
Lemma~\ref{lem:discr}~(ii) would give
\begin{equation*}
1 = \tfrac{a}{4} \leq \tfrac{3}{u_2} + \tfrac{1}{u_0} \leq \tfrac{3}{7}+ \tfrac{1}{2}.
\end{equation*}
Going through the ascendingly ordered $u \in \ZZ_{>0}^3$ with $u_0 \ge 2$ and $u_2 \leq 6$,
we end up with $(2,2,4)$ as the only initial triple in $S(6)$.

\medskip
\noindent
\emph{Case $a = 3$}. 
Again Lemma~\ref{lem:discr}~(i) rules out the case $u_0 = 1$.
Thus, $u_0 \ge 2$, forcing $u_2 \leq 12$, because otherwise Lemma~\ref{lem:discr}~(ii)
would give
\begin{equation*}
\tfrac{3}{4} = \tfrac{a}{4} \leq \tfrac{3}{u_2} + \tfrac{1}{u_0} \leq \tfrac{3}{13} + \tfrac{1}{2}.
\end{equation*}
The search over all ascendingly ordered $u \in \ZZ_{>0}^3$ with $u_0 \ge 2$ and $u_2 \leq 12$
produces exactly two initial triples in $S(3)$, namely $(2,4,6)$ and $(3,3,3)$.

\medskip
\noindent \emph{Case $a = 2$}. 
Here, Lemma~\ref{lem:discr}~(i) rules out $u_0 = 1$ and $u_0 = 2$.
Thus, $u_0 \ge 3$ and $u_2 \leq 18$, because otherwise Lemma~\ref{lem:discr}~(ii)
would imply
\begin{equation*}
\tfrac{1}{2} \ = \ \tfrac{a}{4} \leq \tfrac{3}{u_2} + \tfrac{1}{u_0} \ \leq \ \tfrac{3}{19} + \tfrac{1}{3}.
\end{equation*}
Inside the set of all ascendingly ordered $u \in \ZZ_{>0}^3$ with $u_0 \ge 3$ and $u_2 \leq 18$,
we find $(4,4,8)$ and $(3,6,9)$ as the only two initial triples belonging to $S(2)$.

\medskip
\noindent \emph{Case $a = 1$}. 
Using Lemma~\ref{lem:discr}~(i), we can exclude $u_0 = 1,2,3,4$. Thus, $u_0 \geq 5$ and we derive $u_2 \leq 60$,
as otherwise Lemma~\ref{lem:discr}~(ii) would claim
\begin{align*}
\tfrac{1}{4} = \tfrac{a}{4} \leq \tfrac{3}{u_2} + \tfrac{1}{u_0} \leq \tfrac{3}{61} + \tfrac{1}{5}.
\end{align*}
The search in the ascendingly ordered $u \in \ZZ_{>0}^3$ with $5 \leq u_0$ and $u_2 \leq 60$
yields exactly four initial triples in $S(1)$, namely
$(9, 9, 9)$, $(8, 8, 16)$, $(6, 12, 18)$ and $(5, 20, 25)$. 
\end{proof}

\begin{remark}
\label{rem:markgraf}
Given $a$, let $T(a) \subseteq S(a)$ denote the subset of all
ascendingly ordered solution triples.
For $a \in \{9,8,6,5\}$, we can regard $T(a)$ as the
vertex set of a tree, where we join two triples $u,u' \in T(a)$
by an edge if they are distinct and arise from each other by
a one-step mutation.
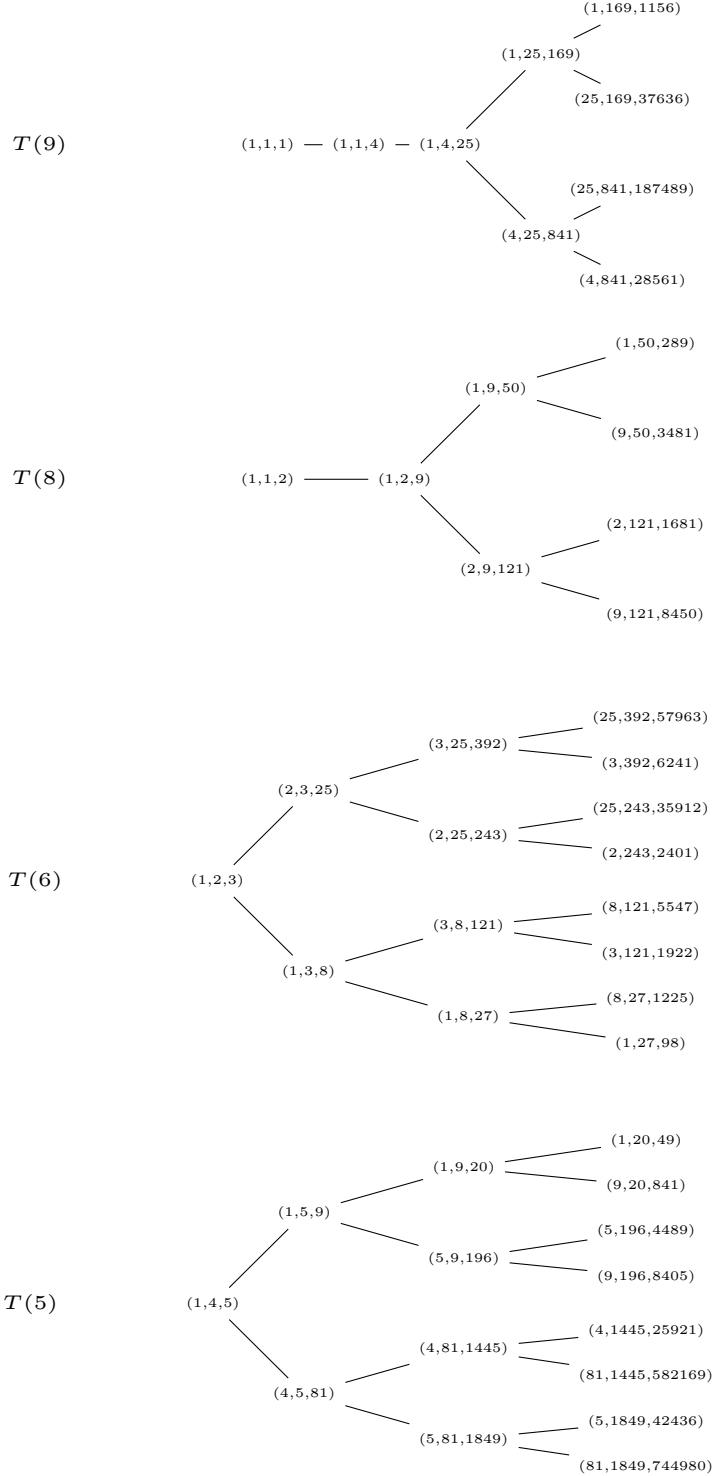
\[
\begin{array}{l}
\begin{tikzpicture}[scale=0.6]
\sffamily
\node[scale=3/2] (a)       at (-5,0) {$\scriptscriptstyle  T(9)$};
\node[] (111) at (0,0) {$\scriptscriptstyle (1,1,1)$};
\node[] (112) at (2,0) {$\scriptscriptstyle (1,1,4)$};
\node[] (125) at (4,0) {$\scriptscriptstyle (1,4,25)$};
\node[] (1513) at (6,2) {$\scriptscriptstyle (1,25,169)$};
\node[] (2529) at (6,-2) {$\scriptscriptstyle (4,25,841)$};
\node[] (113134) at (8,3) {$\scriptscriptstyle (1,169,1156)$};
\node[] (513194) at (8,1) {$\scriptscriptstyle (25,169,37636)$};
\node[] (529533) at (8,-1) {$\scriptscriptstyle (25,841,187489)$};
\node[] (229169) at (8,-3) {$\scriptscriptstyle (4,841,28561)$};
\draw[] (111) edge (112);
\draw[] (112) edge (125);
\draw[] (125) edge (1513);
\draw[] (125) edge (2529);
\draw[] (1513) edge (113134);
\draw[] (1513) edge (513194);
\draw[] (2529) edge (529533);
\draw[] (2529) edge (229169);
\end{tikzpicture}
\\[2ex]
\begin{tikzpicture}[scale=0.6]
\sffamily
\node[scale=3/2] (a)       at (-6,0) {$\scriptscriptstyle  T(8)$};
\node[] (224)        at (-1,0) {$\scriptscriptstyle (1,1,2)$};
\node[] (2418)       at (2,0) {$\scriptscriptstyle (1,2,9)$};
\node[] (218100)     at (4,2) {$\scriptscriptstyle (1,9,50)$};
\node[] (418242)     at (4,-2){$\scriptscriptstyle (2,9,121)$};
\node[] (2100578)    at (7.5,3) {$\scriptscriptstyle (1,50,289)$};
\node[] (181006962)  at (7.5,1) {$\scriptscriptstyle (9,50,3481)$};
\node[] (42423362)   at (7.5,-1){$\scriptscriptstyle (2,121,1681)$};
\node[] (1824216900) at (7.5,-3){$\scriptscriptstyle (9,121,8450)$};
\draw[] (224)  edge (2418);
\draw[] (2418) edge (218100);
\draw[] (2418) edge (418242);
\draw[] (218100) edge (2100578);
\draw[] (218100) edge (181006962);
\draw[] (418242) edge (42423362);
\draw[] (418242) edge (1824216900);
\end{tikzpicture}
\\[2ex]
\end{array}
\]
\[
\begin{array}{l}
\begin{tikzpicture}[scale=0.6]
\sffamily
\node[scale=3/2] (a)       at (-5,0) {$\scriptscriptstyle  T(6)$};
\node[] (2418)       at (-1,0) {$\scriptscriptstyle (1,2,3)$};
\node[] (218100)     at (1,2) {$\scriptscriptstyle (2,3,25)$};
\node[] (418242)     at (1,-2){$\scriptscriptstyle (1,3,8)$};
\node[] (2100578)    at (4.5,3) {$\scriptscriptstyle (3, 25, 392)$};
\node[] (181006962)  at (4.5,1) {$\scriptscriptstyle (2, 25, 243)$};
\node[] (4811445)   at (4.5,-1){$\scriptscriptstyle (3, 8, 121)$};
\node[] (5811849) at (4.5,-3){$\scriptscriptstyle (1, 8, 27)$};
\node[] (5184942436) at (8.5,-3.6){$\scriptscriptstyle (1, 27, 98)$};
\node[] (811849744980) at (8.5,-2.6){$\scriptscriptstyle (8, 27, 1225)$};
\node[] (811445582169) at (8.5,-1.6){$\scriptscriptstyle (3, 121, 1922)$};
\node[] (4144525921)   at (8.5,-0.6){$\scriptscriptstyle (8, 121, 5547)$};
\node[] (51964489)   at (8.5, 0.6){$\scriptscriptstyle (2, 243, 2401)$};
\node[] (91968405)   at (8.5, 1.6){$\scriptscriptstyle (25, 243, 35912)$};
\node[] (12049)   at (8.5, 2.6){$\scriptscriptstyle (3, 392, 6241)$};
\node[] (920841)   at (8.5, 3.6){$\scriptscriptstyle (25, 392, 57963)$};
\draw[] (2418) edge (218100);
\draw[] (2418) edge (418242);
\draw[] (218100) edge (2100578);
\draw[] (218100) edge (181006962);
\draw[] (418242) edge (4811445);
\draw[] (418242) edge (5811849);
\draw[] (5811849) edge (5184942436);
\draw[] (5811849) edge (811849744980);
\draw[] (4811445) edge (811445582169);
\draw[] (4811445) edge (4144525921);
\draw[] (181006962) edge (51964489);
\draw[] (181006962) edge (91968405);
\draw[] (2100578) edge  (12049);
\draw[] (2100578) edge  (920841);
\end{tikzpicture}
\\[2ex]
\end{array}
\]
\[
\begin{array}{l}
\begin{tikzpicture}[scale=0.6]
\sffamily
\node[scale=3/2] (a)       at (-5,0) {$\scriptscriptstyle  T(5)$};
\node[] (2418)       at (-1,0) {$\scriptscriptstyle (1,4,5)$};
\node[] (218100)     at (1,2) {$\scriptscriptstyle (1,5,9)$};
\node[] (418242)     at (1,-2){$\scriptscriptstyle (4,5,81)$};
\node[] (2100578)    at (4.5,3) {$\scriptscriptstyle (1,9,20)$};
\node[] (181006962)  at (4.5,1) {$\scriptscriptstyle (5,9,196)$};
\node[] (4811445)   at (4.5,-1){$\scriptscriptstyle (4,81,1445)$};
\node[] (5811849) at (4.5,-3){$\scriptscriptstyle (5,81,1849)$};
\node[] (5184942436) at (8.5,-3.6){$\scriptscriptstyle (81,1849, 744980)$};
\node[] (811849744980) at (8.5,-2.6){$\scriptscriptstyle (5,1849, 42436)$};
\node[] (811445582169) at (8.5,-1.6){$\scriptscriptstyle (81, 1445, 582169)$};
\node[] (4144525921)   at (8.5,-0.6){$\scriptscriptstyle (4, 1445, 25921)$};
\node[] (51964489)   at (8.5, 0.6){$\scriptscriptstyle (9, 196, 8405)$};
\node[] (91968405)   at (8.5, 1.6){$\scriptscriptstyle (5, 196, 4489)$};
\node[] (12049)   at (8.5, 2.6){$\scriptscriptstyle (9, 20, 841)$};
\node[] (920841)   at (8.5, 3.6){$\scriptscriptstyle (1, 20, 49)$};
\draw[] (2418) edge (218100);
\draw[] (2418) edge (418242);
\draw[] (218100) edge (2100578);
\draw[] (218100) edge (181006962);
\draw[] (418242) edge (4811445);
\draw[] (418242) edge (5811849);
\draw[] (5811849) edge (5184942436);
\draw[] (5811849) edge (811849744980);
\draw[] (4811445) edge (811445582169);
\draw[] (4811445) edge (4144525921);
\draw[] (181006962) edge (51964489);
\draw[] (181006962) edge (91968405);
\draw[] (2100578) edge  (12049);
\draw[] (2100578) edge  (920841);
\end{tikzpicture}
\end{array}
\]
\end{remark}

\goodbreak

\begin{proposition}\label{prop:scalesolutions}
Take $a \in {1,2,3,4,5,6,8,9}$, let $a=ba'$
a factorization into positive integers
and consider the set of scaled triples
\[
bS(a)  := \{(bu_0,bu_1,bu_2); \ u \in S(a)\}.
\]
Then we have $bS(a) \subseteq S(a')$.
Moreover, we can express $S(a')$ for $a' = 4,3,2,1$
in terms of the $S(a)$ for $a = 9,8,6,5$ as
\[
S(4) = 2S(8), \qquad
S(3) = 3S(9) \cup 2S(6), \qquad
S(2) = 4S(8) \cup 3S(6),
\]
\[
S(1) = 9S(9) \cup 8S(8)  \cup 6S(6)  \cup 5S(5).
\]
In particular, $T(4)$ is a tree, each of $T(3)$ and $T(2)$
is a union of two disjoint trees and $T(1)$ is a union
of four disjoint trees.
\end{proposition}

\begin{lemma}\label{lem:gcdsolutions}
Consider a solution $u \in S(a)$ of the squared Markov type equation
and a mutation $\tilde u$ of $u$.
Then $\gcd(u_0,u_1,u_2)=\gcd(\tilde u_0,\tilde u_1,\tilde u_2)$.
\end{lemma}

\begin{proof}
It suffices to treat the case of an
ascendingly ordered $u$ and $\tilde u = \lambda(u)$.
Then $\tilde u_i = u_i$ for $i=0,1$ and
$\tilde u_2 = au_0u_1-2u_0-2u_1-u_2$.
The assertion follows.
\end{proof}

\begin{proof}[Proof of Proposition~\ref{prop:scalesolutions}]
In order to check $bS(a) \subseteq S(a')$, let $u \in S(a)$.
Then the scaled solution $bu$ satisifies    
\[
\frac{(bu_0+bu_1+bu_2)^2}{bu_0bu_1bu_2}
=
\frac{1}{b} \frac{(u_0+u_1+u_2)^2}{u_0u_1u_2}
=
\frac{a}{b}
=
a'.
\]
Thus, $bu \in S(a')$. This verifies in particular all
inclusions ``$\supseteq$'' of the displayed equations
in the proposition.
We show exemplarily $S(4) = 2S(8)$.
Let $u' \in S(4)$. Then Theorem~\ref{thm:intmarkov} tells
us that $u'$ is a mutation of $(2,2,4)$.
Thus, Lemma~\ref{lem:gcdsolutions} ensures $u'=2u$ with
$u \in \ZZ_{>0}^3$ and the above calculation shows
$u \in S(8)$.
\end{proof}

\begin{theorem}
\label{thm:loesungen_Markov_sind_quadratzahlen}
Let $a \in \{1,2,3,4,5,6,8,9\}$ and $u \in S(a)$. 
Then there exist integers~$x_0,x_1,x_2 \in \ZZ_{>0}$
such that, up to permuting the entries, $u$ is of the
form
\begin{equation*}
\begin{array}{lllll}
\displaystyle
\tfrac{9}{a} \left( x_0^2, x_1^2, x_2^2 \right),
&
a=1,3,9, 
&&
\displaystyle
\tfrac{8}{a} \left( x_0^2, x_1^2, 2x_2^2 \right),
&
a=1,2,4,8, 
\\[15pt]
\displaystyle
\tfrac{6}{a} \left( x_0^2, 2x_1^2, 3x_2^2 \right),
&
a=1,2,3,6,
&&                                                      
\displaystyle
\tfrac{5}{a} \left( x_0^2, x_1^2, 5x_2^2 \right),
&
a=1,5 .
\end{array}
\end{equation*} 
Moreover, for $a = 5,6,8,9$ the entries of $u$ are pairwise coprime and for any $a$ the
above numbers $x_0,x_1,x_2$ satisfy the  equation
\begin{equation*}
\xi_0 x_0^2 + \xi_1 x_1^2 + \xi_2 x_2^2 = \sqrt{a \xi_0 \xi_1 \xi_2} x_0x_1x_2, 
\quad 
(\xi_0, \xi_1, \xi_2) 
= 
\begin{cases}
\left(\tfrac{9}{a},\tfrac{9}{a},\tfrac{9}{a}\right), & a= 1,3,9, 
\\[5pt]
\left(\tfrac{8}{a},\tfrac{8}{a},\tfrac{16}{a}\right), & a= 1,2,4,8, 
\\[5pt]
\left(\tfrac{6}{a},\tfrac{12}{a},\tfrac{18}{a}\right), & a= 1,2,3,6, 
\\[5pt]
\left(\tfrac{5}{a},\tfrac{5}{a},\tfrac{25}{a}\right), & a= 1,5.
\end{cases}
\end{equation*}
\end{theorem}

\begin{lemma}
\label{lem:Markov_Teiler_n_teilt_alle_u_i}
Let $a \in \ZZ_{>0}$ and $u \in S(a)$.
Then $\gcd(u_0,u_1) = \gcd(u_0,u_1,u_2)$ holds.
\end{lemma}

\begin{proof}
First recall that $u \in S(a)$ means $(u_0 + u_1 + u_2)^2 = a u_0 u_1 u_2$.
We rewrite this equation as  
\[
u_2^2 + (2u_0+2u_1-au_0u_1)u_2 + (u_0+u_1)^2 \ = \ 0.
\]
Now, set $q := \gcd(u_0,u_1)$. The task is to show that
$q$ divides $u_2$. Dividing the above equation by $q^2$ yields
\[
\left(\frac{u_2}{q}\right)^2 + \left(2\frac{u_0+u_1}{q}-a\frac{u_0u_1}{q} \right) \frac{u_2}{q} + \left(\frac{u_0+u_1}{q} \right)^2 \ = \ 0.
\]
Since $q$ divides $u_0+u_1$ and $u_0u_1$, this is an equation
of integral dependence for the fraction $u_2/q$.
As a factorial ring $\ZZ$ is normal and we can conclude 
$u_2/q \in \ZZ$.
\end{proof}

\begin{proof}[Proof of Theorem~\ref{thm:loesungen_Markov_sind_quadratzahlen}]
Consider any solution $u \in S(a)$ of the squared Markov type equation,
where $a=9,8,6,5$.
Theorem~\ref{thm:intmarkov} and Lemma~\ref{lem:Markov_Teiler_n_teilt_alle_u_i}
tell us $\gcd(u_0,u_1,u_2)=1$.
Let $\xi_i$ be the minimal divisors of $u_i$ such that  
$a\xi_0\xi_1\xi_2$ is a square number and set
$x_i := u_i/\xi_i$.
Then we can write 
\[
(u_0+u_1+u_2)^2 \ = \ au_0u_1u_2 \ = \ a\xi_0\xi_1\xi_2x_0x_1x_2,
\]
where we may assume $\xi_0 \le \xi_1 \le \xi_2$.
By Lemma~\ref{lem:gcdsolutions}, the $u_i$ are pairwise
coprime.
Thus, the $x_i$ are pairwise coprime as well and,
consequently, must be square numbers.
We claim that, according to the value of $a$,  the
vector $\xi = (\xi_0,\xi_1,\xi_2)$ is given as follows:
\begin{center}
\begin{tabular}{c|c|c|c|c|c}
$a$ & $9$ & $8$ & $6$ & $5$
\\ \hline
$\xi$ & $(1,1,1)$ &  $(1,1,2)$ &  $(1,2,3)$ &  $(1,1,5)$  
\end{tabular}
\end{center}
This is obvious for $a=9,8,5$. For $a=6$,
Lemmas~\ref{lem:Markov_Teiler_n_teilt_alle_u_i}
and~\ref{lem:gcdsolutions} yield that precisely
one of the entries of $u$ is divisible by $2$
and precisely one by $3$.
For the initial triple $(1,2,3)$, no entry is divisible by $6$.
Applying stepwise Lemmas~\ref{lem:Markov_Teiler_n_teilt_alle_u_i}
and~\ref{lem:gcdsolutions}, we see that 
the latter holds for any mutation of
$(1,2,3)$. This excludes $\xi = (1,1,6)$.

Thus, the assertion is verified for $a=9,8,6,5$.
For $a=4,3,2,1$, we stress again Theorem~\ref{thm:intmarkov}
and Lema~\ref{lem:gcdsolutions}
and see that in this case the solutions all are obtained by
scaling solutions of the cases $a=9,8,6,5$ with suitable integers.
The assertion follows.
\end{proof}

\begin{remark}
\label{rem:usualmarkov}
Theorem~\ref{thm:loesungen_Markov_sind_quadratzahlen} shows
in particular that the solutions of the \emph{squared} Markov
type equations can be obtained from the solutions
of the usual Markov type equations considered in~\cite{Hu},
see also~\cite[Sec.~3]{KaNo}.
\end{remark}


\section{Background on fake weighted projective spaces}

A \emph{fake weighted projective space} is a $\QQ$-factorial,
projective toric variety of Picard number one.
The defining data of these varieties are  
\emph{projective generator matrices} that means integral
$n \times (n+1)$ matrices
\[
P \ = \ [v_0, \ldots, v_n]
\]
with pairwise distinct primitive columns $v_0, \ldots, v_n \in \ZZ^n$
generating $\QQ^n$ as a convex cone. 
With $P$ we associate the unique fan $\Sigma(P)$ in $\ZZ^n$ having
the maximal cones
\[
\sigma_i \ := \ \cone(v_j; \ j \ne i), \quad i = 0, \ldots, n.
\]
The fake weighted projective space $Z(P)$ is the toric
variety defined by $\Sigma(P)$.
Up to isomorphy, we obtain every fake weighted projective space
this way.

\begin{remark}
\label{rem:genmatiso}
Given projective generator matrices $P,P'$,
one has $Z(P) \cong Z(P')$ if and only if $P' = S \cdot P \cdot T$
with a unimodular matrix $S$ and a permutation matrix~$T$.
\end{remark}

We write $\TT^n$ for the $n$-fold direct product
$\KK^* \times \ldots \times \KK^*$.
A \emph{quasitorus} is an algebraic group isomorphic
to a direct product of $\TT^n \times G$ with $G$ finite
abelian.
Every toric variety admits a natural presentation
as a quotient of an affine space by a quasitorus.
For the fake weighted projective spaces, this looks as
follows.

\begin{remark}
Let $P = (p_{ij})$ be a projective $n \times (n+1)$ generator
matrix. Then~$P$ defines a homomorphism of tori
\[
\pi \colon \TT^{n+1} \ \to \ \TT^n,
\qquad
(t_0,\ldots, t_n)
\ \mapsto \ 
(t_0^{p_{10}} \cdots t_n^{p_{1n}}, \ldots , t_0^{p_{n0}} \cdots t_n^{p_{nn}}).
\]
The kernel $H := \ker(\pi) \subseteq \TT^{n+1}$
is given as $H = \KK^* \times G$ and has 
$K = \ZZ \oplus \Gamma$ as its character group.
Moreover, $H$ acts naturally on $\KK^{n+1}$ and we have
\[
Z(P) \ = \ (\KK^{n+1} \setminus \{0\}) / H.
\]
This gives rise to homogeneous coordinates:
for $z \in \KK^{n+1} \setminus \{0\})$,
write $[z] \in Z(P)$ for the associated point.
For instance, the toric fixed points in $Z(P)$
are
\[
z(0) := [1,0, \ldots, 0],
\quad
z(1) := [0,1,0, \ldots, 0],
\quad
\ldots,
\quad
z(n) := [0, \ldots, 0,1].
\]
\end{remark}

\begin{construction}
\label{constr:dualseq}
Let $P$ be a projective $n \times (n+1)$ generator matrix.
Then $P$ and its transpose $P^*$ fit into a pair of mutually
dual exact sequences
\[
\xymatrix@R=.5cm{
0
\ar[r]
&
{\ZZ}
\ar[r]
&
{\ZZ^{n+1}}
\ar[r]^{P}
&
{\ZZ^n}
&
\\
0
\ar@{<-}[r]
&
K
\ar@{<-}[r]_{Q}
&
{\ZZ^{n+1}}
\ar@{<-}[r]_{\quad P^*}
&
{\ZZ^n}
\ar@{<-}[r]
&
0.
}
\]
Fixing a splitting $K = \ZZ \oplus \Gamma$ into a free
part and the torsion part, we can represent the projection
$\ZZ^{n+1} \to K$ by a \emph{degree matrix $Q$} corresponding
to $P$, that means
\[
Q
\ = \ 
[q_0, \ldots, q_n]
\ = \ 
\left[
\begin{array}{ccc}
u_0 & \ldots & u_n
\\
\eta_0 & \ldots  & \eta_n  
\end{array}
\right],
\qquad
u_i \in \ZZ_{>0},
\quad
\eta_i \in \Gamma
\]
such that any $n$ of the columns $q_i = (u_i,\eta_i)$ of $Q$
generate the group $K = \ZZ \oplus \Gamma$.
If $\Gamma = 0$, then the second row of $Q$ is omitted.
\end{construction}

We denote by $\Cl(Z)$ the divisor class group of a normal
variety and by $\Cl(Z,z)$ the \emph{local class group}
of a point $z \in Z$, that means the factor group of all Weil
divisors by those being principal near $z$.
The order of $\Cl(Z,z)$ is denoted by $\cl(z)$.
The \emph{fake weight vector} of $Z(P)$ is 
\[
w \ = \ w(P) \ = \ (w_0,\ldots,w_n),
\qquad
w_i \ := \ \vert \det(v_j; \ j \ne i) \vert.
\]

\begin{proposition}
\label{prop:clandcr}
Consider $Z = Z(P)$, its fake weight vector~$w$
and $K = \ZZ \oplus \Gamma$ as in
Construction~\ref{constr:dualseq}.
The divisor class group and any degree matrix
of $Z$ satisfy
\[
\Cl(Z) \ \cong \ K,
\qquad
Q
\ = \
\left[
\begin{array}{ccc}
u_0 & \ldots & u_n
\\
\eta_0 & \ldots & \eta_n
\end{array}
\right],
\quad
\begin{array}{l}
\mu := \gcd(w_0, \ldots ,w_n),
\\[2pt]
u := \mu^{-1} \cdot w.
\end{array}
\]
The torsion part $\Gamma$ of $\Cl(Z)$ is of
order $\mu$.
The local class groups of the toric fixed points
and their orders are given by 
\[
\Cl(Z,z(i)) \ = \ K / \ZZ \cdot q_i,
\qquad
\cl(z(i)) \ = \ [K : \ZZ \cdot q_i] \ = \ w_i,  
\]
where $q_i = (u_i,\eta_i)$ are the columns of the degree matrix $Q$.
Finally, the Cox ring of $Z$ together with its $\Cl(Z)$-grading
is given by
\[
\mathcal{R}(Z)
\ = \
\KK[T_0,\ldots, T_n],
\qquad\qquad
\deg(T_i) \ = \ q_i \ \in \ K.
\]
\end{proposition}

\begin{proof}
For the shape of $Q$, note that $P$ annihilates
the fake weight vector~$w$ and thus~$u$ generates the
kernel of $P$.
For the remaining statements see
\cite[Lemma~2.1.4.1 and Prop.~2.4.2.3]{ArDeHaLa}.
\end{proof}

\begin{corollary}
\label{cor:fwpstors}
For any $n$-dimensional fake weighted projective
space, the torsion part of $\Cl(Z)$ is generated
by at most $n-1$ elements.
\end{corollary}

\begin{remark}
\label{rem:wps}
Consider a fake weighted projetive plane $Z=Z(P)$.
If the associated fake weight vector $w$ is primitive,
then $Q = [w_0,\ldots,w_n]$ and $Z = \PP(w_0,\ldots,w_n)$
is an ordinary weighted projective space.
\end{remark}

\begin{proposition}
\label{prop:fwppKsquare}
Consider a fake weighted projective plane $Z = Z(P)$.
An anticanonical divisor on $Z$ is given by
\[
-\mathcal{K}_Z \ = \ D_0+ D_1 + D_2,
\qquad
D_i :=  V(T_i) \subseteq Z.
\]
Moreover, the canonical self intersection number
can be expressed in terms of the
fake weight vector $w = (w_0,w_1,w_2)$ as
\[
\mathcal{K}_Z^2
\ = \
\frac{(w_0+w_1+w_2)^2}{w_0w_1w_2}.
\]
\end{proposition}

\begin{proof}
The first statement is standard toric geometry.
For the second one, we may assume that the
generator matrix $P$ of $Z$ is of the form
\[
P
\ = \
\left[
\begin{array}{lll}
l_0 & l_1 & l_2 
\\
d_0 & d_1 & d_2
\end{array} 
\right],
\qquad
l_i \ne 0, \ i = 0,1,2.
\] 
Then we can use for instance~\cite[Rem.~3.3]{HaHaSp}
for the computation of the canonical self intersection
number.
\end{proof}

The following construction shows how to gain fake weighted
projective planes $Z$ directly from degree matrices and
it allows us to switch between the presentations of a
given $Z$ in terms of a degree matrix or a corresponding
degree matrix.

\begin{construction}
\label{constr:degmat2Z}
Let $K = \ZZ \oplus \Gamma$ with $\Gamma$
finite abelian and let $Q = [q_0, \ldots, q_n]$
with $q_i \in \ZZ_{>0} \oplus \Gamma$
a \emph{degree matrix in $K$}, i.e., 
any $n$ of the $q_i$ generate $K$ as a group.
The quasitorus $H = \Spec \, \KK[K]$
has $K$ as its character group and acts
on $\KK^{n+1}$ via 
\[
h \cdot z \ = \ (\chi^{q_0}(h) z_0, \ldots, \chi^{q_n}(h) z_n).
\]
The orbit space $Z(Q) := \KK^{n+1} \setminus \{0\}$
is a fake weighted projective space.
More precisely, starting with $Q$, we can reproduce
the pair of mutually dual exact sequences from
Construction~\ref{constr:dualseq},
\[
\xymatrix@R=.5cm{
0
\ar[r]
&
{\ZZ}
\ar[r]
&
{\ZZ^{n+1}}
\ar[r]^{P}
&
{\ZZ^n}
&
\\
0
\ar@{<-}[r]
&
K
\ar@{<-}[r]_{Q}
&
{\ZZ^{n+1}}
\ar@{<-}[r]_{\quad P^*}
&
{\ZZ^n}
\ar@{<-}[r]
&
0,
}
\]
by choosing the columns of $P^*$ as the members of a
$\ZZ$-basis for the kernel of~$Q$, regarded as
the map $\ZZ^{n+1} \to K$, $e_i \to q_i$.
Then the transpose $P$ is a projective generator
matrix with $Q$ as corresponding degree matrix 
and we have
\[
Z(Q) \ \cong \ Z(P).
\]  
\end{construction}

\begin{example}
Consider a $2 \times 3$ degree matrix $Q$ in $K$.
Then, thanks to Corollary~\ref{cor:fwpstors},
we have $K = \ZZ \oplus \ZZ/\mu\ZZ$, that means
\[
Q
\ = \
\left[
\begin{array}{ccc}
u_0 & u_1 & u_2
\\[3pt]
\bar \eta_0 & \bar \eta_1 & \bar \eta_2
\end{array}
\right],
\qquad
u_0,u_1,u_2 \in \ZZ_{>0},
\quad
\bar \eta_0,\bar \eta_1,\bar \eta_2 \in \ZZ/ \mu \ZZ.
\]
Moreover, the quasitorus $H = \Spec \, \KK[K]$ is explicitly given
as $H = \KK^* \times C(\mu)$, where $C(\mu) \subseteq \KK^*$
is the group of $\mu$-th roots of unity and $H$ acts on $\KK^3$ via
\[
(t,\zeta) \cdot (z_0,z_1,z_2) \ = \ (t^{u_0}\zeta^{\eta_0}z_0, \, t^{u_1}\zeta^{\eta_1}z_1, \, t^{u_2}\zeta^{\eta_2}z_2).
\]
\end{example}

\begin{proposition}
\label{prop:isochar}
Let $Q,Q'$ be degree matrices in $K = \ZZ \oplus \Gamma$
and $Z,Z'$ the associated weighted projective spaces.
Then the following statements are equivalent:
\begin{enumerate}
\item
We have an isomorphism of the varieties $Z \cong Z'$.
\item
We have an isomorphism of graded algebras $\mathcal{R}(Z) \cong \mathcal{R}(Z')$.
\item
$Q' = [\imath(q_0), \ldots, \imath(q_n)] \cdot B$ with $\imath \in \Aut(K)$
and a permutation matrix $B$.
\end{enumerate}
\end{proposition}

\begin{proof}
The equivalence of the first two statements holds more
generally for any two $\QQ$-factorial Mori dream spaces of
Picard number one; see~\cite[Rem.~3.3.4.2]{ArDeHaLa}.
The equivalence of the second and the third
and statement is clear by Proposition~\ref{prop:clandcr}
and the definition of an isomorphism of graded algebras.
\end{proof}

Proposition~\ref{prop:isochar} becomes an efficient
classification tool in the case of a cyclic torsion part
when we combine it with the following.

\begin{lemma}
\label{lem:autzopluszn}
Consider the abelian group $G := \ZZ \oplus \ZZ/ \mu \ZZ$.
Then any choice of $\epsilon = \pm 1$, $\bar a \in \ZZ/ \mu \ZZ$ and
$\bar c \in (\ZZ/ \mu \ZZ)^*$ defines an automorphism
\[
\varphi_{\epsilon, \bar a, \bar c} \colon G \ \to \ G, 
\qquad
(k, \bar m) \ \mapsto \ (\epsilon k, \, \bar a \bar k + \bar c \bar m).
\]
Conversely, every automorphism $\varphi \colon G \to G$
is of this form.
In particular, $\Aut(G)$ is of order $2 \mu \Phi(\mu)$,
where $\Phi$ denotes Euler's totient function.
\end{lemma}

\begin{proof}
One directly checks that $\varphi_{\epsilon, \bar a, \bar c}$ 
is an automorphism of $G$.
Let $\varphi \colon G \to G$ be any automorphism.
We claim
\[
\varphi(1,\bar 0) \ = \ (\epsilon, \bar a),
\quad
\epsilon \ = \ \pm 1,
\quad \bar a \in \ZZ/ \mu \ZZ,
\qquad
\varphi(0,\bar 1) \ = \ (0, \bar c),
\quad  \bar c \in (\ZZ/ \mu \ZZ)^*.
\]
Indeed, $(0,\bar 1)$ generates the torsion part
$\ZZ/ \mu \ZZ$, which in turn is fixed by $\varphi$,
hence generated by $\varphi(0,\bar 1)$.
This gives the second identity.
So far, we have
\[
\ZZ \oplus \ZZ / \mu \ZZ
\ = \ 
\ZZ \cdot \varphi(1,\bar 0) \oplus \ZZ \cdot \varphi(0,\bar 1)
\ = \
\ZZ \cdot (\epsilon, \bar a) \oplus \ZZ \cdot (0, \bar c)
\]
with some $\epsilon \in \ZZ$ and $\bar a \in \ZZ/ \mu \ZZ$.
This forces $\epsilon = \pm 1$, verifying the claim.
Now, comparing the values on $(1,\bar 0)$ and $(0,\bar 1)$
shows $\varphi = \varphi_{\epsilon, \bar a, \bar c}$.
\end{proof}

\begin{remark}
\label{rem:admopq}
Consider $K = \ZZ \oplus \ZZ / \mu \ZZ$ 
and a degree matrix $Q = [q_0, \ldots, q_n]$
in $K$, where $q_i = (u_i, \bar \eta_i)$.
Then only the automorphims of the
form $\varphi_{1, \bar a, \bar c}$
from Lemma~\ref{lem:autzopluszn}
respect positivity in the first
line of $Q$.
Moreover, observe
\[
\left[
\begin{array}{ccc}
\varphi_{1, \bar a, \bar c}  (q_0)
&
\varphi_{1, \bar a, \bar c}  (q_1)
&
\varphi_{1, \bar a, \bar c}  (q_2)
\end{array}
\right]
\ = \ 
\left[
\begin{array}{cc}
1 & 0
\\      
a & c
\end{array}
\right]
\cdot
\left[
\begin{array}{ccc}
u_0 & u_1 & u_2
\\      
\bar \eta_0 & \bar \eta_1 & \bar \eta_2
\end{array}
\right].
\]
That means that we can realize in practice the
application of an automorphism $\varphi_{1, \bar a, \bar c}$
by matrix multiplication or, equivalently, by
stepwise adding multiples of the upper to the
lower row and scaling the lower row by units.
\end{remark}

\begin{remark}
Let $Q = [q_0,\ldots,q_n]$ be a degree matrix in $K = \ZZ \oplus \Gamma$.
Then the fake weight vector of $Z(Q)$ can be written as
\[
w \ = \ w(Q) \ = \ (w_0, \ldots, w_n), \qquad w_i \ = \ [K:\ZZ \cdot q_i].
\]
We say that $Q$ is \emph{associated} with $w = w(Q)$. Note that
$w(Q) = w(Q')$ may happen even if $Z(Q)$ and $Z(Q')$ are not
isomorphic to each other.
\end{remark}

\begin{definition}
\label{def:adjusted}
Let $a \in \{1,2,3,4,5,6,8,9\}$, $w \in S(a)$, 
$\mu := \gcd(w_0,w_1,w_2)$ and $u := \mu^{-1} \cdot w$.
We call $w$ \emph{adjusted} if one
of the following holds
\begin{itemize}
\item
$a=9$ and $w_0 \le w_1 \le w_2$,
\item
$a=8$, $w_0 \le w_1$ and $2 \mid w_2$,  
\item
$a=6$, $2 \mid w_1$ and $3 \mid w_2$,  
\item
$a=5$, $w_0 \le w_1$ and $5 \mid w_2$,  
\item
$a\le 4$ and $u$ is adjusted.
\end{itemize}
Moreover, we call a degree matrix $Q$ associated
with $w \in S(a)$ \emph{adjusted} if
$w$ is adjusted and for $a \le 4$, the matrix $Q$
is of the form
\[
Q
\ = \
\left[
\begin{array}{ccc}
u_0 & u_1 & u_2
\\[3pt]
\bar 0 & \bar 1 & \bar \eta  
\end{array}
\right],
\quad
1 \le \eta < \mu,
\qquad
\bar \eta
=
\begin{cases}
\bar 2,                  & \text{if } a=9, \ u = (1,1,1),
\\
\bar 2, \bar 5,          & \text{if } a=9, \ u = (1,1,4),
\\
\bar 1, \bar 3, \bar 5,  & \text{if } a=8, \ u = (1,1,2).
\end{cases}
\]
\end{definition}

\begin{proposition}
\label{rem:adjdegmat}
Let $a \in \{1,2,3,4,5,6,8,9\}$,
$w \in S(a)$ adjusted and
$Q,Q'$ adjusted degree matrices 
associated with $w$.
Then we have
\[
Z(Q) \cong Z(Q') \ \iff \ Q=Q'.
\]  
\end{proposition}

\begin{proof}
The assertion follows from Proposition~\ref{prop:isochar}.
\end{proof}

\begin{example}
For $a= 1,2,3,4,5,6,8,9$, we list divisor class group,
corresponding generator and adjusted degree matrices
for the fake weighted projective planes having an
initial triple as fake weight vector:
  

{\tiny
\[
\begin{array}{lc}
a=9
&
\ZZ 
\\[2pt]
\left[
\setlength{\arraycolsep}{3pt}
\begin{array}{rrr}
1 & 1 & -2
\\      
0 & 1 & -1
\end{array}
\right]
& 
\left[
\setlength{\arraycolsep}{3pt}
\begin{array}{rrr}
1 & 1 & 1
\end{array}
\right]
\end{array}  
\qquad
\begin{array}{lc}
a=8
&
\ZZ
\\[2pt]
\left[
\setlength{\arraycolsep}{3pt}
\begin{array}{rrr}
1 & 1 & -1
\\      
0 & -2 & 1
\end{array}
\right]
& 
\left[
\setlength{\arraycolsep}{3pt}
\begin{array}{rrr}
1 & 1 & 2
\end{array}
\right]
\end{array}  
\qquad
\begin{array}{lc}
a=6
&
\ZZ
\\[2pt]
\left[
\setlength{\arraycolsep}{3pt}
\begin{array}{rrr}
1 & 1 & -1
\\      
0 & -3 & 2
\end{array}
\right]
& 
\left[
\setlength{\arraycolsep}{3pt}
\begin{array}{rrr}
1 & 2 & 3
\end{array}
\right]
\end{array}  
\]
}
  

{\tiny
\[
\begin{array}{lc}
a=5
&
\ZZ
\\[2pt]
\left[
\setlength{\arraycolsep}{3pt}
\begin{array}{rrr}
1 & 1 & -1
\\      
0 & -5 & 4
\end{array}
\right]
& 
\left[
\setlength{\arraycolsep}{3pt}
\begin{array}{rrr}
1 & 4 & 5
\end{array}
\right]
\end{array}  
\qquad
\begin{array}{lc}
a=4
&
\ZZ \oplus \ZZ/2\ZZ
\\[2pt]
\left[
\setlength{\arraycolsep}{3pt}
\begin{array}{rrr}
1 & 1 & -1
\\      
0 & -4 & 2
\end{array}
\right]
& 
\left[
\setlength{\arraycolsep}{3pt}
\begin{array}{rrr}
1 & 1 & 2
\\      
\bar 0 & \bar 1 & \bar 1
\end{array}
\right]
\end{array}  
\qquad
\begin{array}{lc}
a=3
&
\ZZ \oplus \ZZ/3\ZZ
\\[2pt]
\left[
\setlength{\arraycolsep}{3pt}
\begin{array}{rrr}
1 & 1 & -1
\\      
0 & -3 & 3
\end{array}
\right]
& 
\left[
\setlength{\arraycolsep}{3pt}
\begin{array}{rrr}
1 & 1 & 1
\\      
\bar 0 & \bar 1 & \bar 2
\end{array}
\right]
\end{array}  
\]
}
 

{\tiny
\[
\begin{array}{lc}
a=3
&
\ZZ \oplus \ZZ/2\ZZ
\\[2pt]
\left[
\setlength{\arraycolsep}{3pt}
\begin{array}{rrr}
1 & 1 & -1
\\      
0 & -6 & 4
\end{array}
\right]
& 
\left[
\setlength{\arraycolsep}{3pt}
\begin{array}{rrr}
1 & 2 & 3
\\      
\bar 0 & \bar 1 & \bar 1
\end{array}
\right]
\end{array}  
\qquad
\begin{array}{lc}
a=2
&
\ZZ \oplus \ZZ/4\ZZ
\\[2pt]
\left[
\setlength{\arraycolsep}{3pt}
\begin{array}{rrr}
1 & 1 & -1
\\      
0 & -8 & 4
\end{array}
\right]
& 
\left[
\setlength{\arraycolsep}{3pt}
\begin{array}{rrr}
1 & 1 & 2
\\      
\bar 0 & \bar 1 & \bar 1
\end{array}
\right]
\end{array}  
\qquad
\begin{array}{lc}
a=2
&
\ZZ \oplus \ZZ/4\ZZ
\\[2pt]
\left[
\setlength{\arraycolsep}{3pt}
\begin{array}{rrr}
2 & 2 & -3
\\      
1 & -3 & 1
\end{array}
\right]
& 
\left[
\setlength{\arraycolsep}{3pt}
\begin{array}{rrr}
1 & 1 & 2
\\      
\bar 0 & \bar 1 & \bar 3
\end{array}
\right]
\end{array}  
\]
}
 

{\tiny
\[
\begin{array}{lc}
a=2
&
\ZZ \oplus \ZZ/3\ZZ
\\[2pt]
\left[
\setlength{\arraycolsep}{3pt}
\begin{array}{rrr}
1 & 1 & -1
\\      
0 & -9 & 6
\end{array}
\right]
& 
\left[
\setlength{\arraycolsep}{3pt}
\begin{array}{rrr}
1 & 2 & 3
\\      
\bar 0 & \bar 1 & \bar 1
\end{array}
\right]
\end{array}  
\qquad
\begin{array}{lc}
a=2
&
\ZZ \oplus \ZZ/3\ZZ
\\[2pt]
\left[
\setlength{\arraycolsep}{3pt}
\begin{array}{rrr}
3 & 3 & -3
\\      
1 & -2 & 1
\end{array}
\right]
& 
\left[
\setlength{\arraycolsep}{3pt}
\begin{array}{rrr}
1 & 2 & 3
\\      
\bar 0 & \bar 1 & \bar 2
\end{array}
\right]
\end{array}  
\qquad
\begin{array}{lc}
a=1
&
\ZZ \oplus \ZZ/9\ZZ
\\[2pt]
\left[
\setlength{\arraycolsep}{3pt}
\begin{array}{rrr}
3 & 3 & -6
\\      
1 & -2 & 1
\end{array}
\right]
& 
\left[
\setlength{\arraycolsep}{3pt}
\begin{array}{rrr}
1 & 1 & 1
\\      
\bar 0 & \bar 1 & \bar 2
\end{array}
\right]
\end{array}  
\]
}
 

{\tiny
\[
\begin{array}{lc}
a=1
&
\ZZ \oplus \ZZ/8\ZZ
\\[2pt]
\left[
\setlength{\arraycolsep}{3pt}
\begin{array}{rrr}
1 & 1 & -1
\\      
0 & -16 & 8
\end{array}
\right]
& 
\left[
\setlength{\arraycolsep}{3pt}
\begin{array}{rrr}
1 & 1 & 2
\\      
\bar 0 & \bar 1 & \bar 1
\end{array}
\right]
\end{array}  
\qquad
\begin{array}{lc}
a=1
&
\ZZ \oplus \ZZ/8\ZZ
\\[2pt]
\left[
\setlength{\arraycolsep}{3pt}
\begin{array}{rrr}
4 & 4 & -4
\\      
1 & -3 & 1
\end{array}
\right]
& 
\left[
\setlength{\arraycolsep}{3pt}
\begin{array}{rrr}
1 & 1 & 2
\\      
\bar 0 & \bar 1 & \bar 3
\end{array}
\right]
\end{array}  
\qquad
\begin{array}{lc}
a=1
&
\ZZ \oplus \ZZ/8\ZZ
\\[2pt]
\left[
\setlength{\arraycolsep}{3pt}
\begin{array}{rrr}
2 & 2 & -2
\\      
0 & -7 & 3
\end{array}
\right]
& 
\left[
\setlength{\arraycolsep}{3pt}
\begin{array}{rrr}
1 & 1 & 2
\\      
\bar 0 & \bar 1 & \bar 5
\end{array}
\right]
\end{array}  
\]
}


{\tiny
\[
\begin{array}{lc}
a=1
&
\ZZ \oplus \ZZ/6\ZZ
\\[2pt]
\left[
\setlength{\arraycolsep}{2.96pt}
\begin{array}{rrr}
1 & 1 & -1
\\      
0 & -18 & 12
\end{array}
\right]
& 
\left[
\setlength{\arraycolsep}{2.96pt}
\begin{array}{rrr}
1 & 2 & 3
\\      
\bar 0 & \bar 1 & \bar 1
\end{array}
\right]
\end{array}  
\qquad
\begin{array}{lc}
a=1
&
\ZZ \oplus \ZZ/6\ZZ
\\[2pt]
\left[
\setlength{\arraycolsep}{2.96pt}
\begin{array}{rrr}
3 & 3 & -3
\\      
2 & -4 & 2
\end{array}
\right]
& 
\left[
\setlength{\arraycolsep}{2.96pt}
\begin{array}{rrr}
1 & 2 & 3
\\      
\bar 0 & \bar 1 & \bar 5
\end{array}
\right]
\end{array}  
\qquad
\begin{array}{lc}
a=1
&
\ZZ \oplus \ZZ/5\ZZ
\\[2pt]
\left[
\setlength{\arraycolsep}{2.96pt}
\begin{array}{rrr}
1 & 1 & -1
\\      
0 & -25 & 20
\end{array}
\right]
& 
\left[
\setlength{\arraycolsep}{2.96pt}
\begin{array}{rrr}
1 & 4 & 5
\\      
\bar 0 & \bar 1 & \bar 1
\end{array}
\right]
\end{array}  
\]
}


{\tiny
\[
\begin{array}{lc}
a=1
&
\ZZ \oplus \ZZ/5\ZZ
\\[2pt]
\left[
\setlength{\arraycolsep}{3pt}
\begin{array}{rrr}
5 & 5 & -5
\\      
3 & -2 & 1
\end{array}
\right]
& 
\left[
\setlength{\arraycolsep}{3pt}
\begin{array}{rrr}
1 & 4 & 5
\\      
\bar 0 & \bar 1 & \bar 2
\end{array}
\right]
\end{array}  
\qquad
\begin{array}{lc}
a=1
&
\ZZ \oplus \ZZ/5\ZZ
\\[2pt]
\left[
\setlength{\arraycolsep}{3pt}
\begin{array}{rrr}
5 & 5 & -5
\\      
1 & -4 & 3
\end{array}
\right]
& 
\left[
\setlength{\arraycolsep}{3pt}
\begin{array}{rrr}
1 & 4 & 5
\\      
\bar 0 & \bar 1 & \bar 3
\end{array}
\right]
\end{array}  
\qquad
\begin{array}{lc}
a=1
&
\ZZ \oplus \ZZ/5\ZZ
\\[2pt]
\left[
\setlength{\arraycolsep}{3pt}
\begin{array}{rrr}
5 & 5 & -5
\\      
2 & -3 & 2
\end{array}
\right]
& 
\left[
\setlength{\arraycolsep}{3pt}
\begin{array}{rrr}
1 & 4 & 5
\\      
\bar 0 & \bar 1 & \bar 4
\end{array}
\right]
\end{array}  
\]
}

\end{example}


\section{Proof of Theorem~\ref{thm:introthm1}}
\label{sec:classifyfwpp}

We classify the fake weighted projective planes
$Z$ of integral degree $\mathcal{K}_Z^2$.
Proposition~\ref{prop:fwppKsquare} shows that the
potential fake weight vectors of such $Z$ are the
solution triples of the squared Markov type
equations.
It turns out that in fact every solution triple
occurs as a fake weight vector.
We determine and verify all associated adjusted
degree matrices and list them in a redundance free
manner.

\begin{proposition}
Let $Z$ be a projective toric surface of Picard number
one of degree $a := \mathcal{K}_Z^2 \in \ZZ$.
Then $a \in \{1,2,3,4,5,6,8,9\}$. 
\end{proposition}

\begin{proof}
We may assume $Z=Z(P)$. Then   
Proposition~\ref{prop:fwppKsquare}
yields that the fake weight vector of $Z$ is a
solution triple of a squared Markov type equation.
Thus, Theorem~\ref{thm:intmarkov} yields the
assertion.
\end{proof}

\begin{proposition}
\label{prop:classfwppdeg5689}
Let $Z$ be a projective toric surface of Picard number one
of degree $a \in \{5,6,8,9\}$.
Then $Z$ is isomorphic to a weighted projective space
$\PP(w_0,w_1,w_2)$ for exactly one adjusted $w \in S(a)$.
\end{proposition}

\begin{proof}
We may assume $Z=Z(P)$. By Proposition~\ref{prop:fwppKsquare},
the fake weight vector $w$ of $Z$ satisfies $w \in S(a)$.
By Theorem~\ref{thm:loesungen_Markov_sind_quadratzahlen},
the entries $w_0,w_1,w_2$ of $w$ are pairwise coprime.
Thus, $Z \cong \PP(w_0,w_1,w_2)$ is an ordinary weighted
projective plane; see Remark~\ref{rem:wps}.
Taking adjusted solution triples, makes the presentation
redundance free.
\end{proof}

For the fake weighted projective planes $Z$ of degree
$a \in \{1,2,3,4\}$, we will encounter torsion in the
divisor class group. We will list the resulting $Z$
in terms of adjusted degree matrices in the sense
of Definition~\ref{def:adjusted}.
Whereas the procedure of finding and verifying
those follows a common pattern for all $q=1,2,3,4$,
the detailed arguing relies on the specific divisibility
properties of the solution triples from $S(a)$ in each
case.

\begin{lemma}
\label{lem:Q2localcl}
Let $Q = [q_0,q_1,q_2]$ be a degree matrix 
in $\ZZ \oplus \ZZ / \mu \ZZ$, write
$q_i = (u_i, \bar \eta_i)$ and assume
$\gcd(u_0,\mu)= 1$.
Then we find an automorphism of
$\ZZ \oplus \ZZ / \mu \ZZ$ which,
applied to columns, turns $Q$ into
\[
Q'
\ = \ 
\left[
\begin{array}{ccc}
u_0 & u_1 & u_2
\\  
\bar 0 & \bar 1 & \bar \eta
\end{array}
\right],
\qquad
\bar \eta \in (\ZZ / \mu \ZZ)^*.
\]
We have $Z(Q) \cong Z(Q')$ and $Q,Q'$ share the
same corresponding generator matrices.
Moreover, the local class group 
of $z(i) \in Z(Q)$ is cyclic of order $\mu u_i$
and the fake weight vector is given by
$w(Q) = \mu \cdot (u_0,u_1,u_2)$.
\end{lemma}

\begin{proof}
Proposition~\ref{prop:isochar} and
Lemma~\ref{lem:autzopluszn} tell us
that suffices to apply suitable operations from
Remark~\ref{rem:admopq}.
Thanks to $\gcd(u_0,\mu)=1$, we have $\bar u_0 \in (\ZZ/\mu \ZZ)^*$.
Hence $\bar \eta_0 = \bar a \bar u_0$ for some $a \in \ZZ$.
Subtracting $\bar a \cdot (\bar u_0,\bar u_1,\bar u_2)$
from the second row, we achieve $\bar \eta_0 = \bar 0$.
As each of the pairs $q_0,q_1$ and $q_0,q_2$ generates
$\ZZ \oplus \ZZ / \mu \ZZ$ as a group, we can conclude
$\bar \eta_1,\bar \eta_2 \in (\ZZ/\mu \ZZ)^*$.
Scaling the second row of $Q$ with the multiplicative
inverse on $\bar \eta_1$ finally turns $Q$ into $Q'$.
 
By Construction~\ref{constr:degmat2Z} and Proposition~\ref{prop:clandcr},
the local class groups $K_i := \Cl(Z,z_i)$ are given
as $K_i = K / \ZZ \cdot q_i'$, where
$K = \ZZ \oplus \ZZ / \mu \ZZ$
and $q_i'$ is the $i$-th column of $Q'$.
Clearly $K_i$ is generated by the classes of $(1,\bar 0)$
and $(0, \bar 1)$.
We treat $K_0$.
Let $\zeta, \kappa \in \ZZ$ with $\zeta \mu + \kappa u_0 =1$.
Then we observe
\[
(1,\bar 0) \ = \  \zeta \mu \cdot (1, \bar 1) + \kappa \cdot (u_0 , \bar 0),
\qquad
(0,\bar 1) \ = \ (1,\bar 1) - (1,\bar 0),
\]
and conclude that the class of $(1,\bar 1)$ generates $K_0$.
Moreover, due to the coprimeness of $u_0$ and $\mu$, the class
$(1,\bar 1)$  in $K_0$ of order $\mu u_0$.
Now consider $K_i$ for $i=1,2$.
Set $\eta_1 := 1$ and $\eta_2 := \eta$.
Choosing $\zeta_i \in \ZZ$ with $\bar \zeta_i \bar \eta_i=- \bar 1$
in $\ZZ / \mu \ZZ$. Then
\[
(0,\bar 1) \ = \  \zeta_i u_i \cdot (1, \bar 0) -  \zeta_i \cdot (u_i , \bar \eta_i),
\]
and thus the class of $(1,\bar 0)$ generates $K_i$,
where $i=1,2$.
Consequently, each of the local class groups $\Cl(Z,z(i)) = K_i$
is cyclic of order $\mu u_i$.
The statement on the fake weight vector 
is clear by Proposition~\ref{prop:clandcr}.
\end{proof}

\begin{proposition}
\label{prop:classfwppdeg4}
Let $Z$ be a fake weighted projective plane of degree~4.
Then $Z \cong Z(Q)$ with $w = w(Q)$ adjusted
and there are $x_0,x_1,x_2 \in \ZZ_{>0}$ such that
\[
(x_0^2,x_1^2,2x_2^2) \ \in \ S(8),
\qquad
w \ = 2 \cdot (x_0^2,x_1^2,2x_2^2).
\]
Moreover, the divisor class group is 
$\Cl(Z) = \ZZ \oplus \ZZ/2\ZZ$
and the degree matrix~$Q$ can be choosen as
\[
\qquad
Q(w)
\ = \
\left[
\begin{array}{ccc}
x_0^2 & x_1^2 & 2x_2^2
\\[3pt]
\bar 0 & \bar 1 & \bar 1  
\end{array}
\right].
\]
Any such matrix is a degree matrix of a fake
weighted projective plane of degree~4 and distinct
matrices represent non-isomorphic fake weighted
projective planes.
\end{proposition}

\begin{proof}
We may assume $Z = Z(P)$ with a generator matrix $P$.
Let $Q$ be a corresponding degree matrix.
Suitably renumbering the columns of both, we achieve
that $w = w(P)$ is adjusted.
By Proposition~\ref{prop:fwppKsquare}, we have
$w \in S(4)$.
Proposition~\ref{prop:scalesolutions} says 
$w = 2 \cdot u$ with $u \in S(8)$.
From Theorem~\ref{thm:loesungen_Markov_sind_quadratzahlen}
we infer $u = (x_0^2,x_1^2,2x_2^2)$ and pairwise coprimeness
of the entries of $u$.
Proposition~\ref{prop:clandcr} shows $\Cl(Z) = \ZZ \oplus \ZZ/2\ZZ$
and establishes the first row of $Q$.
Since $x_0$ is coprime to $2x_2^2$, it must be odd.
Thus, Lemma~\ref{lem:Q2localcl} allows us to assume
$Q = Q(w)$ as claimed.

Now consider any $Q(w)$ as in the assertion,
that means a matrix with first row 
$u = (x_0^2,x_1^2,2x_2^2)$ from $S(8)$
and second row $(\bar 0, \bar 1, \bar 1)$
with entries in $\ZZ/2\ZZ$.
We verify that $Q(w)$ is a degree matrix of a
fake weighted projective plane with the
claimed properties.
Given two columns $q_i,q_j$ of $Q$ we have to show
that these generate $K = \ZZ \oplus \ZZ/2\ZZ$ as a group.
Since $u_i,u_j$ are coprime, we find $(1, \bar 0)$
or $(1, \bar 1)$ in the span of $q_i,q_j$.
In the first case, we are done.
In the second case, observe that we find also an
element $(b, \bar 0)$ with $b$ odd in the span of
$q_i,q_j$.
Subtracting the $(b-1)$-fold of $(1,\bar 1)$ gives
$(1, \bar 0)$ and we are done as well.
Thus, $Q(w)$ is a degree matrix in $K$.
Set $Z = Z(Q(w))$.
Lemma~\ref{lem:Q2localcl} ensures that $Z$ has fake
weight vector $2 \cdot u \in S(4)$
and we can apply Proposition~\ref{prop:fwppKsquare}
to see that $Z$ is of degree 4.
Finally, Proposition~\ref{rem:adjdegmat}
ensures that distinct matrices $Q(w)$ define
non-isomorphic fake weighted projective
planes.
\end{proof}

\begin{proposition}
\label{prop:classfwppdeg3}
Let $Z$ be a fake weighted projective plane of degree 3.
Then $Z \cong Z(Q)$ with $w=w(Q)$ adjusted
and there are $x_0,x_1,x_2 \in \ZZ_{>0}$ such that
we are in one of the following situations:
\begin{enumerate}
\item  
we have $(x_0^2,x_1^2,x_2^2) \in S(9)$,
$w = 3 \cdot (x_0^2,x_1^2,x_2^2)$,
the divisor class group is
$\Cl(Z) = \ZZ \oplus \ZZ/3\ZZ$,
and the degree matrix $Q$ can be chosen as
\[
\qquad\qquad
Q(w)
\ = \
\left[
\begin{array}{ccc}
x_0^2 & x_1^2 & x_2^2
\\[3pt]
\bar 0 & \bar 1 & \bar 2
\end{array}
\right],
\]
\item
we have $(x_0^2,2x_1^2,3x_2^2) \in S(6)$,
$w = 2 \cdot (x_0^2,2x_1^2,3x_2^2)$,
the divisor class group is
$\Cl(Z) = \ZZ \oplus \ZZ/2\ZZ$
and the degree matrix $Q$ can be chosen
as
\[
\qquad\qquad
Q(w)
\ = \
\left[
\begin{array}{ccc}
x_0^2 & 2x_1^2 & 3x_2^2
\\[3pt]
\bar 0 & \bar 1 & \bar 1  
\end{array}
\right].
\]
\end{enumerate}
Any matrix as above is a degree matrix of a fake weighted
projective plane of degree~3 and distinct
matrices represent non-isomorphic fake weighted projective
planes.
\end{proposition}

\begin{lemma}
\label{lem:a9wimod3}
Let $u \in S(9)$.
Then $\bar u_i = \bar 1 \in \ZZ/3 \ZZ$ holds for $i=0,1,2$.
In particular, $3 \nmid u_i$ for $i=0,1,2$.
\end{lemma}

\begin{proof} 
Clearly the assertion holds for $u=(1,1,1)$.
By Theorem~\ref{thm:intmarkov}, we only
have to show that, if $u \in S(9)$ satisfies
the assertion, then $u' = \lambda(u)$ does so.
Recall
\[
u_0' \ = \ u_0,
\quad
u_1' \ = \ u_1,
\quad
u_2' \ = \ 9u_0u_1 - 2u_0 - 2u_1 - u_2
\]
from the construction of $\lambda$ given in
Lemma~\ref{lem:mutation}.
For $u_0'$ and $u_1'$ is nothing to show
and, passing to the classes in $\ZZ / 3 \ZZ$,
we directly see $\bar u_2' = - \bar 5 = \bar 1$.
\end{proof}

\begin{proof}[Proof of Proposition~\ref{prop:classfwppdeg3}]
We may assume $Z = Z(P)$ and $w = w(P)$ being adjusted.
By Proposition~\ref{prop:fwppKsquare}, we have
$w \in S(3)$.
According to Proposition~\ref{prop:scalesolutions},
we either have $w = 3 \cdot u$ with $u \in S(9)$
or $w = 2 \cdot u$ with $u \in S(6)$.
We go through each of these two cases
and establish the claimed shape for a corresponding
degree matrix $Q$.

\medskip
\noindent
\emph{Case (i)}: $w = 3 \cdot u$ with $u \in S(9)$.
Theorem~\ref{thm:loesungen_Markov_sind_quadratzahlen}
provides us with $u = (x_0^2,x_1^2,x_2^2)$ and
Proposition~\ref{prop:clandcr} yields that the torsion
part of $\Cl(Z)$ is of order three.
Moreover, Lemma~\ref{lem:a9wimod3} says in particular
$\gcd(x_0^2,3) = 1$ and thus Lemma~\ref{lem:Q2localcl}
shows that $Q$ can be choosen to be of the shape
\[
\qquad\qquad
Q
\ = \
\left[
\begin{array}{ccc}
x_0^2 & x_1^2 & x_2^2
\\[3pt]
\bar 0 & \bar 1 & \bar \eta
\end{array}
\right],
\qquad
\eta \in (\ZZ/ 3 \ZZ)^*.
\]
We show $\bar \eta \ne  \bar 1$.
Otherwise, due to Lemma~\ref{lem:a9wimod3},
adding $\bar 2 \cdot (\bar x_0^2, \bar x_1^2, \bar x_2^2)$
turns the second row of $Q$
into $(\bar 2, \bar 0, \bar 0)$,
which contradicts to Remark~\ref{rem:admopq},
guaranteeing that this operation respects
the properties of a degree matrix.

We check that each of the matrices $Q(w)$ from
Case~(i) is a degree matrix in
$K := \ZZ \oplus \ZZ / 3 \ZZ$
defining a fake weighted projective plane $Z$ as claimed.
Given two columns $q_i,q_j$ of $Q(w)$, we use
Lemma~\ref{lem:a9wimod3} to observe
\[
u_jq_i - u_iq_j \ = \ (0, \bar \eta_i + \bar 2 \bar \eta_j) \ \in \ \{(0,\bar 1), \, (0,\bar 2)\},  
\]
where $q_k = (u_k, \bar \eta_k)$.
Since $u_i$ and $u_j$ are coprime,
$(1, \bar \zeta)$ lies in the span of $q_i,q_j$
for some $\zeta \in \ZZ / 3 \ZZ$.
We conclude that $q_i,q_j$ generate $K$ as a group.
Lemma~\ref{lem:Q2localcl} ensures
that $Z$ has fake weight vector
$3 \cdot (x_0^2,x_1^2,x_2^2) \in S(3)$. Thus,
Proposition~\ref{prop:fwppKsquare}
yields that $Z$ is of degree 3.

\medskip
\noindent
\emph{Case (ii)}: $w = 2 \cdot u$ with $u \in S(6)$.
By Theorem~\ref{thm:loesungen_Markov_sind_quadratzahlen}
we know $u = (x_0^2,2x_1^2,3x_2^2)$ and that $u$
has pairwise coprime entries.
In particular, $x_0^2$ is odd.
Proposition~\ref{prop:clandcr} gives
$\Cl(Z) = \ZZ \oplus \ZZ/ 2\ZZ$
and by Lemma~\ref{lem:Q2localcl}, the degree
matrix can be choosen as
\[
\qquad\qquad
Q
\ = \
\left[
\begin{array}{ccc}
x_0^2 & 2 x_1^2 & 3 x_2^2
\\[3pt]
\bar 0 & \bar 1 & \bar 1
\end{array}
\right].
\] 

We check that each of the matrices $Q(w)$ from Case~(ii)
is a degree matrix in $K := \ZZ \oplus \ZZ / 2 \ZZ$
defining a fake weighted projective plane $Z$ as claimed.
As $u = (x_0^2,2 x_1^2,3 x_2^2)$ has pairwise coprime
entries, $x_0$ and $x_2$ are odd.
Thus, for any two columns $q_i,q_j$ of $Q(w)$, we obtain
\[
u_j q_i - u_i q_j \ = \ (0,\bar 1).
\]
Moreover, since $u_i$ and $u_j$ are coprime,
$(1, \bar \zeta)$ lies in the span of $q_i,q_j$
for $\zeta \in \ZZ/2\ZZ$.
We conclude that~$q_i,q_j$ generate $K$ as a group.
Lemma~\ref{lem:Q2localcl} ensures
that $Z$ has fake weight vector $2 \cdot u \in S(3)$ 
and Proposition~\ref{prop:fwppKsquare}
guarantees that~$Z$ is of degree 3.

Finally, in each of the Cases (i) and (ii),
Proposition~\ref{rem:adjdegmat} tells us that
distinct matrices $Q(w)$ define
non-isomorphic fake weighted projective planes.
\end{proof}

\begin{proposition}
\label{prop:classfwppdeg2}
Let $Z$ be a fake weighted projective plane of degree $2$.
Then $Z \cong Z(Q)$ with $w=w(Q)$ adjusted
and there are $x_0,x_1,x_2 \in \ZZ_{>0}$ such that
we are in one of the following situations:
\begin{enumerate}
\item  
we have $(x_0^2,x_1^2,x_2^2) \in S(8)$,
$w = 4 \cdot (x_0^2,x_1^2, 2 x_2^2)$,
the divisor class group is
$\Cl(Z) = \ZZ \oplus \ZZ/4\ZZ$
and the degree matrix $Q$ can be chosen as
\[
\qquad\qquad
Q(w, \bar 1)
\ = \
\left[
\begin{array}{ccc}
x_0^2 & x_1^2 & 2x_2^2
\\[3pt]
\bar 0 & \bar 1 & \bar 1  
\end{array}
\right],
\qquad
Q(w, \bar 3)
\ = \
\left[
\begin{array}{ccc}
x_0^2 & x_1^2 & 2x_2^2
\\[3pt]
\bar 0 & \bar 1 & \bar 3
\end{array}
\right],
\]
\item
we have $(x_0^2, 2x_1^2, 3x_2^2) \in S(6)$,
$w = 3 \cdot (x_0^2, 2x_1^2, 3x_2^2)$,
the divisor class group is
$\Cl(Z) = \ZZ \oplus \ZZ/3\ZZ$
and the degree matrix $Q$ can be chosen
as
\[
\qquad\qquad
Q(w, \bar 1)
\ = \
\left[
\begin{array}{ccc}
x_0^2 & 2x_1^2 & 3x_2^2
\\[3pt]
\bar 0 & \bar 1 & \bar 1  
\end{array}
\right],
\qquad
Q(w, \bar 2)
\ = \
\left[
\begin{array}{ccc}
x_0^2 & 2x_1^2 & 3x_2^2
\\[3pt]
\bar 0 & \bar 1 & \bar 2
\end{array}
\right].
\]
\end{enumerate}
Any matrix as above is a degree matrix of a fake weighted
projective plane of degree~$2$ and distinct
matrices represent non-isomorphic fake weighted projective
planes.
\end{proposition}

\begin{lemma}
\label{lem:a_gl_8_w_i_mod_4} 
Let $u \in S(8)$ be adjusted. Then
$(\bar{u}_0, \bar{u}_1, \bar{u}_2) = (\bar{1}, \bar{1}, \bar{2})$
holds in $\ZZ/ 4\ZZ$.
\end{lemma}

\begin{proof} 
The claim holds for~$u = (1,1,2)$.
By Theorem~\ref{thm:intmarkov} it suffices 
to show that if $u \in S(8)$ fulfills
the claim after adjusting, then
$u' := \lambda(u)$ does so. 
Lemma~\ref{lem:mutation} says
\begin{equation*}
u_0' \ = \ u_0,
\quad
u_1' \ = \ u_1,
\quad
u_2' \ = \ 8 u_0u_1 - 2u_0 - 2u_1 - u_2.
\end{equation*}
If $\bar u_2 = \bar 2$ holds, then
$\bar u_0 = \bar u_1 = \bar 1$ and we compute
$\bar u_2' = - \bar u_2 = \bar 2$.
If $\bar u_2 = \bar 1$ holds, then 
one of $\bar u_0, \bar u_1$ equals $\bar 1$
and the other equals $\bar 2$.
We obtain $\bar u_2' = - \bar 3 = \bar 1$.
\end{proof}

\begin{lemma}
\label{lem:a_gl_6_w_i_mod_3}  
Let $u \in S(6)$ be adjusted. Then
$(\bar u_0, \bar u_1, \bar u_2)  = (\bar 1, \bar 2, \bar 0)$
holds in $\ZZ/ 3\ZZ$.
\end{lemma}

\begin{proof} 
The assertion is true for~$u = (1,2,3)$.
By Theorem~\ref{thm:intmarkov}, we only have
to show that, if $u \in S(6)$ fulfills the claim
after adjusting, then $u' = \lambda(u)$ does so. 
According to Lemma~\ref{lem:mutation}, we have
\[
u_0' \ = \ u_0,
\quad
u_1' \ = \ u_1,
\quad
u_2' \ = \ 6 u_0u_1 - 2u_0 - 2u_1 - u_2.
\]
If $\bar u_2 = \bar 0$, then we may assume
$\bar u_0 = \bar 1$ and $\bar u_1 = \bar 2$
and compute $\bar u_2' = \bar 0$.
If $\bar u_2 = \bar 1$, then we may assume
$\bar u_0 = \bar 0$ and $\bar u_1 = \bar 2$
and compute $\bar u_2' = -\bar 5 = \bar 1$.
If $\bar u_2 = \bar 2$, then we may assume
$\bar u_0 = \bar 0$ and $\bar u_1 = \bar 1$
and compute $\bar u_2' = -\bar 4 = \bar 2$.
\end{proof}

\begin{proof}[Proof of Proposition~\ref{prop:classfwppdeg2}] 
We may assume $Z = Z(P)$ and that $w = w(P)$ is adjusted.
Let $Q$ be a degree matrix corresponding to $P$.
Proposition~\ref{prop:fwppKsquare} yields $w \in S(2)$.
According to Proposition~\ref{prop:scalesolutions},
we either have~$w = 4 \cdot u$ with~$u \in S(8)$
or~$w = 3 \cdot u$ with~$u \in S(6)$.

\medskip 
\noindent
\emph{Case (i)}: $w = 4 \cdot u$ with $u \in S(8)$.
Then $u = (x_0^2,x_1^2, 2 x_2^2)$ and $u$ has pairwise
coprime entries; see 
Theorem~\ref{thm:loesungen_Markov_sind_quadratzahlen}.
Moreover, from Proposition~\ref{prop:clandcr}
and Corollary~\ref{cor:fwpstors} we infer that
the toriosn part of $\Cl(Z)$ is cyclic of order four.
Coprimeness of $x_0^2$ and $2 x_2^2$ forces
$x_0^2$ to be odd and thus Lemma~\ref{lem:Q2localcl}
allows us to assume
\[
\qquad\qquad
Q
\ = \
\left[
\begin{array}{ccc}
x_0^2 & x_1^2 & 2 x_2^2
\\[3pt]
\bar 0 & \bar 1 & \bar \eta
\end{array}
\right],
\qquad
\bar \eta \in (\ZZ/ 4 \ZZ)^* = \{\bar 1, \bar 3\}.
\]

We verify that each of the matrices $Q(w, \bar \eta)$
from Case~(i) is a degree 
matrix in $K := \ZZ \oplus \ZZ / 4 \ZZ$, defining
a fake weighted projective plane $Z$ as claimed.
Lemma~\ref{lem:a_gl_8_w_i_mod_4} delivers
$(\bar u_0, \bar u_1, \bar u_2) = (\bar 1, \bar 1, \bar 2)$
for $u = (x_0^2, x_1^2, 2 x_2^2)$.
Thus, given two columns $q_i,q_j$ of $Q(w, \bar \eta)$, 
we can compute
\begin{equation*}
u_j q_i - u_i q_j \ \in \ \{ (0, \bar{1}), (0, \bar{3}) \}.
\end{equation*}
Moreover, since $u_i$ and $u_j$ are coprime,
$(1, \bar \zeta)$ lies in the span of $q_i,q_j$
for some $\zeta \in \ZZ/4\ZZ$.
We conclude that~$q_i,q_j$ generate $K$ as a group.
Lemma~\ref{lem:Q2localcl} establishes the
fake weight vector $w = 4 \cdot u \in S(2)$
and Proposition~\ref{prop:fwppKsquare} says 
that~$Z$ is of degree~$2$.

\medskip 
\noindent
\emph{Case (ii)}: $w = 3 \cdot u$ with $u \in S(6)$. 
Then $u = (x_0^2, 2 x_1^2, 3 x_2^2)$ and $u$ has
pairwise coprime entries as provided by 
Theorem~\ref{thm:loesungen_Markov_sind_quadratzahlen}.
Proposition~\ref{prop:clandcr} shows that $\Cl(Z)$
has torsion part of order three.
Moreover, $\gcd(x_0^2,3) = 1$ brings Lemma~\ref{lem:Q2localcl}
into the game and we can assume 
\[
\qquad\qquad
Q
\ = \
\left[
\begin{array}{ccc}
x_0^2 & 2 x_1^2 & 3 x_2^2
\\[3pt]
\bar 0 & \bar 1 & \bar \eta
\end{array}
\right],
\qquad
\bar \eta \in (\ZZ/ 3 \ZZ)^*.
\]

We check that each of the matrices $Q(w, \bar \eta)$ from
Case~(ii) is a degree matrix in $K := \ZZ \oplus \ZZ / 3 \ZZ$
and defines a fake weighted projective plane $Z$ as claimed.
From Lemma~\ref{lem:a_gl_6_w_i_mod_3} we know
\[
(\bar x_0^2, \bar 2 \bar x_1^2, \bar{3} \bar x_2^2)
\ = \
(\bar{1}, \bar{2}, \bar{0})
\ \in \
\ZZ/3\ZZ.
\]
Thus, given two columns $q_i,q_j$ of $Q$, we see that 
$u_j q_i - u_i q_j$ equals either $(0, \bar 1)$ or
$(0, \bar 2)$, where $u = (x_0^2,2x_1^2,3x_2^2)$.
Since $u_i$ and $u_j$ are coprime, we find additionally 
$(1, \bar \zeta)$ in the span of $q_i,q_j$ for some
$\zeta \in \ZZ/3\ZZ$.
Thus, $q_i,q_j$ generate $K$ as a group.
Lemma~\ref{lem:Q2localcl}
tells us that the fake weight vector of~$Z$
is $w = 3 \cdot u \in S(2)$. 
Thus, Proposition~\ref{prop:fwppKsquare} ensures that~$Z$
is of degree~$2$.

As in the preceding proofs, Proposition~\ref{rem:adjdegmat}
guarantees that in all the cases just treated, distinct
matrices define non-isomorphic varieties.
\end{proof}

\begin{proposition}
\label{prop:classfwppdeg1}
Let $Z$ be a fake weighted projective plane of degree $1$.
Then $Z \cong Z(Q)$ with $w=w(Q)$ adjusted
and there are $x_0,x_1,x_2 \in \ZZ_{>0}$ such that
we are in one of the following situations:
\begin{enumerate}
\item  
we have $(x_0^2, x_1^2, x_2^2) \in S(9)$,
$w = 9 \cdot (x_0^2,x_1^2,  x_2^2)$,
the divisor class group is given 
by~$\Cl(Z) = \ZZ \oplus \ZZ/9\ZZ$
and the degree matrix $Q$ can be chosen as
\[
\qquad\qquad
Q(w, \bar \eta)
\ = \
\left[
\begin{array}{ccc}
x_0^2 & x_1^2 & x_2^2
\\[3pt]
\bar 0 & \bar 1 & \bar \eta 
\end{array}
\right],
\qquad
\bar{\eta} = \bar{2}, \bar{5}, \bar{8}.
\]
\item  
we have $(x_0^2, x_1^2, 2 x_2^2) \in S(8)$,
$w = 8 \cdot (x_0^2, x_1^2,  2 x_2^2)$,
the divisor class group is
$\Cl(Z) = \ZZ \oplus \ZZ/8\ZZ$
and the degree matrix $Q$ can be chosen as
\[
\qquad\qquad
Q(w, \bar \eta)
\ = \
\left[
\begin{array}{ccc}
x_0^2 & x_1^2 & 2x_2^2
\\[3pt]
\bar 0 & \bar 1 & \bar \eta  
\end{array}
\right],
\qquad
\bar{\eta} = \bar 1, \bar 3, \bar 5, \bar 7.
\]
\item  
we have $(x_0^2, 2 x_1^2, 3 x_2^2) \in S(6)$,
$w = 6 \cdot (x_0^2, 2 x_1^2,  3 x_2^2)$,
the divisor class group is
$\Cl(Z) = \ZZ \oplus \ZZ/6\ZZ$
and the degree matrix $Q$ can be chosen as
\[
\qquad\qquad
Q(w,  \bar{\eta})
\ = \
\left[
\begin{array}{ccc}
x_0^2 & 2x_1^2 & 3x_2^2
\\[3pt]
\bar 0 & \bar 1 & \bar{\eta} 
\end{array}
\right],
\qquad
\bar{\eta} =  \bar 1, \bar 5.
\]
\item  
we have $(x_0^2, x_1^2, 5 x_2^2) \in S(5)$,
$w = 5 \cdot (x_0^2, x_1^2, 5 x_2^2)$,
the divisor class group is
$\Cl(Z) = \ZZ \oplus \ZZ/5\ZZ$
and the degree matrix $Q$ can be chosen as
\[
\qquad\qquad
Q(w, \bar{\eta})
\ = \
\left[
\begin{array}{ccc}
x_0^2 & x_1^2 & 5x_2^2
\\[3pt]
\bar 0 & \bar 1 & \bar{\eta}
\end{array}
\right],
\qquad
\bar{\eta} =  \bar 1, \bar 2, \bar 3, \bar 4.
\]
\end{enumerate}
Any $Q(w, \bar{\eta})$ as above is a degree matrix of a fake
weighted projective plane of degree~$1$.
Two distinct $Q(w, \bar{\eta})$ define non-isomorphic
fake weighted projective planes, except they stem from (1-9-$\ast$) or
(1-8-$\ast$) and are both taken from one of the~sets

\medskip

\begin{center}
{\tiny
$
\setlength{\arraycolsep}{1.8pt}
\left\{   
\left[
\begin{array}{ccc}
1 & 1 & 1
\\[3pt]
\bar 0 & \bar 1 & \bar 2
\end{array}
\right],
\
\left[
\begin{array}{ccc}
1 & 1 & 1
\\[3pt]
\bar 0 & \bar 1 & \bar 5
\end{array}
\right],
\
\left[
\begin{array}{ccc}
1 & 1 & 1
\\[3pt]
\bar 0 & \bar 1 & \bar 8
\end{array}
\right]
\right\}
$,
\hspace{.5cm}
$
\setlength{\arraycolsep}{1.8pt}
\left\{   
\left[
\begin{array}{ccc}
1 & 1 & 4
\\[3pt]
\bar 0 & \bar 1 & \bar 5
\end{array}
\right],
\
\left[
\begin{array}{ccc}
1 & 1 & 4
\\[3pt]
\bar 0 & \bar 1 & \bar 8
\end{array}
\right]
\right\}
$,
\hspace{.5cm}
$
\setlength{\arraycolsep}{1.8pt}
\left\{   
\left[
\begin{array}{ccc}
1 & 1 & 2
\\[3pt]
\bar 0 & \bar 1 & \bar 1
\end{array}
\right],
\
\left[
\begin{array}{ccc}
1 & 1 & 2
\\[3pt]
\bar 0 & \bar 1 & \bar 7
\end{array}
\right]
\right\}
$.
}

\medskip

\end{center}
\noindent
In these cases, for any pair of degree matrices stemming from
a common set, the associated fake weighted projective planes
are in fact isomorphic to each other.
\end{proposition}

\begin{lemma}
\label{lem:a_gl_8_w_i_mod_8} 
Let $u \in S(8)$ be adjusted. Then
$(\bar{u}_0, \bar{u}_1, \bar{u}_2) = (\bar{1}, \bar{1}, \bar{2})$
in $\ZZ/ 8\ZZ$.
\end{lemma}

\begin{proof} 
The claim holds for~$u = (1,1,2)$.
By Theorem~\ref{thm:intmarkov} it suffices 
to show that if $u \in S(8)$ fulfills
the claim after adjusting, then
$u' := \lambda(u)$ does so. 
Lemma~\ref{lem:mutation} says
\begin{equation*}
u_0' \ = \ u_0,
\quad
u_1' \ = \ u_1,
\quad
u_2' \ = \ 8 u_0u_1 - 2u_0 - 2u_1 - u_2.
\end{equation*}
If $\bar u_2 = \bar 2$ holds, then
$\bar u_0 = \bar u_1 = \bar 1$ and we compute
$\bar u_2' = - \bar 6 = \bar 2$.
If $\bar u_2 = \bar 1$ holds, then 
one of $\bar u_0, \bar u_1$ equals $\bar 1$
and the other equals $\bar 2$.
We obtain $\bar u_2' = - \bar 7 = \bar 1$.
\end{proof}

\begin{lemma}
\label{lem:a_gl_6_w_i_mod_6}  
Let $u \in S(6)$ be adjusted. Then
$(\bar u_0, \bar u_1, \bar u_2)  = (\bar 1, \bar 2, \bar 3)$ 
in $\ZZ/ 6\ZZ$.
\end{lemma}

\begin{proof} 
The assertion is valid for~$u = (1,2,3)$.
By Theorem~\ref{thm:intmarkov}, we only have
to show that, if $u \in S(6)$ fulfills the claim
after adjusting, then $u' = \lambda(u)$ does as well. 
According to Lemma~\ref{lem:mutation}, we have
\[
u_0' \ = \ u_0,
\quad
u_1' \ = \ u_1,
\quad
u_2' \ = \ 6 u_0u_1 - 2u_0 - 2u_1 - u_2.
\]
If $\bar u_2 = \bar 3$, then we may assume
$\bar u_0 = \bar 1$ and $\bar u_1 = \bar 2$
and compute $\bar u_2' = - \bar{9} = \bar 3$.
If $\bar u_2 = \bar 1$, then we may assume
$\bar u_0 = \bar 2$ and $\bar u_1 = \bar 3$
and compute $\bar u_2' = -\bar{11} = \bar 1$.
If $\bar u_2 = \bar 2$, then we may assume
$\bar u_0 = \bar 1$ and $\bar u_1 = \bar 3$
and compute $\bar u_2' = -\bar{10} = \bar 2$.
\end{proof}

\begin{lemma}\label{lem:a_gl_5_w_i_mod_5} 
Let $u \in S(5)$ be adjusted. Then
$(\bar u_0, \bar u_1, \bar u_2)
=
(\bar 1, \bar 4, \bar 0),
\,
(\bar 4, \bar 1, \bar 0),
$ 
in~$\ZZ/ 5\ZZ$. 
\end{lemma}

\begin{proof} 
The assertion is true for~$u = (1,4,5)$.
Due to Theorem~\ref{thm:intmarkov}, it is enough
to show that, if $u \in S(5)$ fulfills the claim
after adjusting, then $u' = \lambda(u)$ does so. 
By Lemma~\ref{lem:mutation}, we have
\[
u_0' \ = \ u_0,
\quad
u_1' \ = \ u_1,
\quad
u_2' \ = \ 5 u_0u_1 - 2u_0 - 2u_1 - u_2.
\]
If $\bar u_2 = \bar 0$, then up to switching
$\bar u_0 = \bar 1$ and $\bar u_1 = \bar 4$
and we obtain $\bar u_2' = - \bar{10} = \bar 0$.
If $\bar u_2 = \bar 1$, then up to switching
$\bar u_0 = \bar 0$ and $\bar u_1 = \bar 4$
and we obtain $\bar u_2' = -\bar 9 = \bar 1$.
If $\bar u_2 = \bar 4$, then up to switching
$\bar u_0 = \bar 0$ and $\bar u_1 = \bar 1$
and we obtain $\bar u_2' = -\bar 6 = \bar 4$.
\end{proof}

\begin{proof} [Proof of Proposition~\ref{prop:classfwppdeg1}] 
We may assume $Z = Z(P)$ and that $w = w(P)$ is an adjusted
fake weight vector. Let $Q$ be a degree matrix
corresponding to $P$.
Proposition~\ref{prop:fwppKsquare} yields~$w \in S(1)$.
By Proposition~\ref{prop:scalesolutions},
one of the following holds 
\[
w = 9 \cdot u, \, u \in S(9), \
w = 8 \cdot u, \, u \in S(8), \
w = 6 \cdot u, \, u \in S(6), \
w = 5 \cdot u, \, u \in S(5).
\]

\medskip 
\noindent
\emph{Case (i):} $w = 9 \cdot u$ with $u \in S(9)$.
Theorem~\ref{thm:loesungen_Markov_sind_quadratzahlen}
provides us with $u = (x_0^2,x_1^2, x_2^2)$, having
pairwise coprime entries.
Due to Proposition~\ref{prop:clandcr} and Corollary~\ref{cor:fwpstors},
the torsion part of $\Cl(Z)$ is cyclic of order nine.
Lemmas~\ref{lem:Q2localcl} and~\ref{lem:a9wimod3}
allow us to assume
\[
\qquad\qquad
Q
\ = \
\left[
\begin{array}{ccc}
x_0^2 & x_1^2 &   x_2^2
\\[3pt]
\bar 0 & \bar 1 & \bar \eta
\end{array}
\right],
\qquad
\bar{\eta} \in (\ZZ/ 9 \ZZ)^* = \{\bar 1, \bar 2, \bar 4, \bar 5, \bar 7, \bar 8\}.
\]
We exclude $\bar \eta = \bar 1, \bar 4 , \bar 7$.
Otherwise, $Q$ would be a degree matrix for one of these values.
Then $(0,\bar 1)$ must lie in the span of the last two columns
of $Q$.
Since $x_1,x_2$ are coprime, this means
$\bar x_2^2  = \bar x_1^2 \bar \eta + \bar 1$ in
$\ZZ / 9 \ZZ$.
Inserting the values of $\bar \eta$ gives
\[
\bar x_2^2 \ = \ \bar x_1^2 + \bar 1, \qquad 
\bar x_2^2 \ = \ \bar 4 \bar x_1^2 + \bar 1, \qquad 
\bar x_2^2 \ = \ \bar 7 \bar x_1^2 + \bar 1,
\]
as identities in $\ZZ/9\ZZ$. Moreover, Lemma~\ref{lem:a9wimod3}
says that $\bar x_1^2, \bar x_2^2$ are taken from
$\{\bar 1, \bar 4, \bar 7\}$ as well.
A direct computation shows that none of the above
identities can be satisfied this way.
Thus, $\bar \eta$ must be one of $\bar 2, \bar 5, \bar 8$.

We show that $Q(w, \bar \eta)$ for $\bar \eta = \bar 2, \bar 5, \bar 8$
is a degree matrix in $K := \ZZ \oplus \ZZ / 9 \ZZ$.
Since $u= (x_0^2, x_1^2, x_2^2)$ has are pairwise coprime
entries, we find an element of
the form $(1, \bar \zeta)$ with $\bar \zeta \in \ZZ/9\ZZ$ in the
span of any two columns of $Q$. Moreover,
$(0, \bar \kappa_{ij})$ lies  the span of the columns $q_i,q_j$ of
$Q(w, \bar \eta)$ for 
\[ 
\bar \kappa_{ij} 
\ := \
\bar x_j^2 \bar \eta_i - \bar x_i^2 \bar \eta_j,
\]
where we write $q_k = (x_k^2,\bar \eta_k)$.
Recall that due to Lemma~\ref{lem:a9wimod3} the classes
$\bar x_i^2, \bar x_j^2$ belong to $\{\bar 1, \bar 4, \bar 7\}$.
Thus, for $i=0$ and $j=1,2$ we see due to $\bar \eta_0 = \bar 0$
that $\bar \kappa_{ij} = \bar x_0^2 \bar \eta_j$ is a unit in $\ZZ/9\ZZ$.
Moreover,
\[
\bar \kappa_{12} = \bar x_2^2 - \bar \eta_2 \bar x_1^2 \in (\ZZ/9\ZZ)^*
\]
is checked via directly by inserting all possible
values $\bar 1, \bar 4, \bar 7$ for the $\bar x_i$ and $\bar 2, \bar 5, \bar 8$
for~$\bar \eta_2$.
Lemma~\ref{lem:Q2localcl}
shows that the fake weight vector of the
fake weighted projective
plane $Z$ associated with $Q(w,\bar \eta)$ is  
$w = 9 \cdot u \in S(1)$ for $u= (x_0^2, x_1^2, x_2^2)$. 
Thus, Proposition~\ref{prop:fwppKsquare} yields
that~$Z$ is of degree~$1$.

\medskip 
\noindent
\emph{Case (ii):} $w = 8 \cdot u$ with $u \in S(8)$.
Then $u = (x_0^2,x_1^2, 2 x_2^2)$
and $u$ has pairwise coprime entries; see
Theorem~\ref{thm:loesungen_Markov_sind_quadratzahlen}.
Proposition~\ref{prop:clandcr} and Corollary~\ref{cor:fwpstors}
ensure that $\Cl(Z)$ has cyclic torsion part of order eight.
As $x_0$ is odd by coprimeness of $x_0, 2x_2$,
Lemma~\ref{lem:Q2localcl} allows us to assume 
\[
\qquad\qquad
Q
\ = \
\left[
\begin{array}{ccc}
x_0^2 & x_1^2 & 2 x_2^2
\\[3pt]
\bar \eta_0 & \bar \eta_1 & \bar \eta_2
\end{array}
\right],
\qquad
\bar 0, \bar 1, \bar \eta \in \ZZ/ 8 \ZZ = \{\bar 1, \bar 3, \bar 5, \bar 7\}.
\]

Let us see why each $Q(w, \bar \eta)$ is a degree
matrix in $K := \ZZ \oplus \ZZ / 8 \ZZ$,
defining a fake weighted projective plane
$Z$ as claimed.
Lemma~\ref{lem:a_gl_8_w_i_mod_8} delivers
$(\bar{u}_0, \bar{u}_1, \bar{u}_2) = (\bar{1}, \bar{1}, \bar{2})$
for $u = (x_0^2,x_1^2,2 x_2^2)$.
Using this, we derive for any 
two columns $q_i,q_j$ of $Q(w, \bar{\eta})$
the following:
\[
u_j q_i - u_i q_j \ \in \ \{ (0, \bar{\zeta}); \ \zeta  \in (\ZZ/ 8 \ZZ)^* \}.
\]
Indeed, for $q_0,q_j$ this obvious and for $q_1,q_2$,
we arrive at $\bar \zeta = \bar 2 - \bar \eta$,
which is a unit for $\eta = \bar 1, \bar 3, \bar 5, \bar 7$.
Moreover, since $u_i$ and $u_j$ are coprime,
$(1, \bar \zeta)$ belongs to the span of $q_i,q_j$.
We conclude that $q_i,q_j$ generate $K$ as a group.
Lemma~\ref{lem:Q2localcl} delivers the
fake weight vector $w = 8 \cdot u \in S(1)$
with $u= (x_0^2, x_1^2, 2 x_2^2)$.
So, Proposition~\ref{prop:fwppKsquare}
yields $\mathcal{K}_Z^2=1$.

\medskip 
\noindent
\emph{Case (iii):} $w = 6 \cdot u$ with $u \in S(6)$.
Then  $u = (x_0^2, 2 x_1^2, 3 x_2^2)$
and $u$ has pairwise coprime entries
by Theorem~\ref{thm:loesungen_Markov_sind_quadratzahlen}.
Proposition~\ref{prop:clandcr} yields that the torsion part of $\Cl(Z)$
is $\ZZ/ 6 \ZZ$.
Moreover, $x_0^2$ is coprime to $6$ and thus Lemma~\ref{lem:Q2localcl}
says that we may assume 
\[
\qquad\qquad
Q
\ = \
\left[
\begin{array}{ccc}
x_0^2 & x_1^2 & 2 x_2^2
\\[3pt]
\bar 0 & \bar 1 & \bar \eta
\end{array}
\right],
\qquad
\bar{\eta} \in (\ZZ/ 6 \ZZ)^* = \{\bar 1, \bar 5\}.
\]

Let us see that each of the matrices $Q(w, \bar \eta)$ 
is a degree matrix in $K := \ZZ \oplus \ZZ / 6 \ZZ$ and
defines the desired fake weighted projective plane $Z$.
For $u= (x_0^2, 2x_1^2, 3x_2^2)$, Lemma~\ref{lem:a_gl_6_w_i_mod_6} delivers
$
(\bar{x}_0^2, \bar{2} \bar{x}_1^2, \bar{3}\bar{x}_2^2)
=
(\bar{1}, \bar{2}, \bar{3})
$. 
Using this, we obtain that any two columns $q_i, q_j$ of $Q(w, \bar \eta)$,
satisfy
\[
u_j q_i - u_i q_j \ \in \ \{ (0, \bar{1}), (0, \bar{5}) \}.
\]
For $q_0,q_j$ this is obvious.
Moreover, one computes
$u_2q_1-u_1q_2 = (0, \bar 3 - \bar 2 \bar \eta)$.
By coprimeness of $u_i$ and $u_j$, 
the span of $q_i,q_j$ contains $(1, \bar \zeta)$ 
for some $\zeta \in \ZZ/6\ZZ$.
Hence,~$q_i,q_j$ generate $K$ as a group.
Lemma~\ref{lem:Q2localcl} provides us with the
fake weight vector $w = 6 \cdot u \in S(1)$. 
Using Proposition~\ref{prop:fwppKsquare}, we get $\mathcal{K}_Z^2 = 1$.

\medskip 
\noindent
\emph{Case (iv):} $w = 5 \cdot u$ with $u \in S(5)$.
Then $u = (x_0^2, x_1^2, 5 x_2^2)$
and $u$ has pairwise coprime entries by 
Theorem~\ref{thm:loesungen_Markov_sind_quadratzahlen}.
According to Proposition~\ref{prop:clandcr},
the torsion part of $\Cl(Z)$ is cyclic of order five
and by Lemma~\ref{lem:Q2localcl}, we can choose
\[
\qquad\qquad
Q
\ = \
\left[
\begin{array}{ccc}
x_0^2 & x_1^2 & 5 x_2^2
\\[3pt]
\bar 0 & \bar 1 & \bar \eta
\end{array}
\right],
\qquad
\bar{\eta} \in (\ZZ/ 5 \ZZ)^*.
\]

We check that each of the  $Q(w, \bar \eta)$ 
is a degree matrix in $K := \ZZ \oplus \ZZ / 5 \ZZ$,
defining a fake weighted projective plane $Z$ as claimed.
We know  
$
(\bar x_0^2, \bar x_1^2, \bar 5 \bar x_2^2)
=
(\bar 1, \bar 4, \bar 0)
$
from Lemma~\ref{lem:a_gl_5_w_i_mod_5}.
Set $u=(x_0^2, x_1^2, 5 x_2^2)$.
Then, given two columns $q_i, q_j$ of $Q(w, \bar{\eta}_2)$, 
we observe
\[
  u_j q_i - u_i q_j \ \in \ \{ (0, \bar \zeta); \ \bar \zeta \in (\ZZ/5\ZZ)^* \}.
\]
This is obvious for $q_0,q_j$ and
$u_2q_1-u_1q_2$ evaluates to $(0, - \bar 4 \bar \eta)$.
Moreover, since $u_i$ and $u_j$ are coprime,
some $(1, \bar \zeta)$ lies in the span of
$q_i,q_j$.
Lemma~\ref{lem:Q2localcl} gives
the fake weight vector $w = 5 \cdot u \in S(1)$.
Proposition~\ref{prop:fwppKsquare} yields that~$Z$ is of 
degree~$1$.

Once more, Proposition~\ref{rem:adjdegmat}
guarantees that in all the cases, apart from the exceptions
listed in the assertion, distinct matrices define non-isomorphic
varieties.
See Examples~\ref{ex:iso-19-111}, \ref{ex:iso-19-114}
and~\ref{ex:iso-18-112} for the verifications of the
isomorphies in the exceptional cases.
\end{proof}


\section{Local Gorenstein indices and $T$-singularities}
\label{sec:locgi}

We investigate the singularities of the fake weighted
projective planes of integral degree, classified
in the preceding section.
Recall that the only possible singular points of a fake
weighted projective plane $Z = Z(P)$ are the toric
fixed points $z(k) \in Z$.
Moreover, $z(k) \in Z$ is singular, unless
the entries of the fake weight vector $w = w(P)$
are coprime and we have $w_k=1$; in this case,
$Z$ must be an ordinary weighted projective plane.

Let us recall the notion of a $T$-singularity;
see also~\cite[Sections~2,~4]{HaPro}.
Let~$k,p$ be coprime positive integers, denote
by $C(dk^2) \subseteq \KK^*$ the group of $dk^2$-th
roots of unity and consider the action 
\[
C(dk^2) \times \KK^2 \ \to \ \KK^2,
\qquad
\zeta \cdot z \ = \ (\zeta z_1, \, \zeta^{dpk-1} z_2).
\]
Then $U := \KK^2/\KK^*$ is an affine toric surface
and the image $u \in U$ of $0 \in \KK^2$ is singular
as soon as $k>1$.
A \emph{(cyclic) $T$-singularity}, also called a
\emph{singularity of type}
\[
\frac{1}{dk^2}(1,dpk-1),
\]  
is a surface singularity isomorphic to $u \in U$
as above.
The \emph{local Gorenstein index}~$\iota(z)$ of 
a point~$z$ in a normal variety~$Z$ is the order of the 
canonical divisor class in the local class group $\Cl(Z,z)$.

\begin{lemma}
\label{lem:tsing} 
Let $U$ be an affine toric surface with fixed point~$u\in U$. 
Then the following statements are equivalent:
\begin{enumerate}
\item
$u \in U$ is of type $\tfrac{1}{d k^2} (1, dpk-1).$
\item
$\iota(u)^2$ divides $\cl(u)$.
\end{enumerate}
If these statements hold, then $d = \cl(u)/\iota(u)^2$ and
$k= \iota(u)$ and there exists $b \in \ZZ$ such that $U$
has generator matrix
\begin{equation*}
P
\ = \
\left[ 
\begin{array}{cc}
\iota(u) & \iota(u) 
\\
d\iota(u)+b & b 
\end{array}
\right].
\end{equation*}
\end{lemma}

\begin{proof}
Let $u \in U$ be of type $\frac{1}{dk^2}(1,dpk-1)$.
Choose $a,b \in \ZZ$ such that $ak - bp = 1$.
Consider the affine toric surface $U'$ given by the
generator matrix
\[
P
\ = \
\left[
\begin{array}{cc}
k & k 
\\
dk + b & b
\end{array}
\right].
\]
The associated homomorphism $\TT^2 \to \TT^2$,
$t \mapsto (t_1^kt_2^k, t_1^{dk+b}t_2^b)$
extends to a toric morphism $\pi \colon \KK^2 \to U'$.
Moreover, we obtain an isomorphism
\[
C(dk^2) \ \to \ \ker(\pi),
\qquad
\zeta \ \mapsto \ (\zeta, \zeta^{dpk-1}).
\]
We conclude $U' \cong \KK^2/C(dk^2)$ with $C(dk^2)$ acting
as needed for type $\frac{1}{dk^2}(1,dpk-1)$.
Thus $U' \cong U$ and, using~\cite[Rem.~3.7]{HaHaSp},
we obtain  
\[
\cl(u) \ = \ \vert \det(P) \vert \ = \ dk^2,
\qquad
\iota(u) \ = \ k.
\]
Assume $\cl(u) = d\iota(u)^2$. The affine toric surface~$U$
is given by a generator matrix~$P$.
With~$k := \iota(u)$, a suitable unimodular transformation
turns~$P$ into
\[
P
\ = \
\left[
\begin{array}{cc}
k & k 
\\
c & b
\end{array}
\right],
\qquad
\gcd(c,k) = \gcd(b,k)=1.
\]
By assumption $\cl(u) = \vert \det(P) \vert$
equals $d\iota(u)^2 = dk^2$.
Thus, we may assume $c = dk+b$.
Take $a,p \in \ZZ$ with
$ak - bp = 1$ and $p \ge 1$.
Then we have an action
\[
C(dk^2) \times \KK^2 \ \to \ \KK^2,
\qquad
\zeta \cdot z \ = \ (\zeta z_1, \, \zeta^{dpk-1} z_2).
\]
With similar arguments as above, we verify that
$U$ is the 
quotient~$\KK^2/C(dk^2)$ for this action
and thus see that $u \in U$ is of type
$\frac{1}{dk^2}(1,dpk-1)$. 
\end{proof}

\begin{example}
\label{ex:tsing}
Take integers $c,c_0,l$ with $c_0 < 0$ and $l > 0$ such that
$l$ and $c$ are coprime and consider the affine toric surface
$U$ with the generator matrix
\[
P
\ = \
\left[ 
\begin{array}{cc}
l & l
\\
c & c + lc_0 
\end{array}
\right].
\]
The toric fixed point $u \in U$ has local class group
order $\cl(u) = -c_0l^2$ and local Gorenstein index
$\iota(u) = l$; see~\cite[Rem.~3.7]{HaHaSp}.
Thus, $U$ is at most $T$-singular.
\end{example}

The aim of the section is to determine the local Gorenstein indices
and the $T$-singularities for all fake weighted projective
planes of integral degree. We begin with the following auxiliary
observations.

\begin{lemma}
\label{lem:x_k_teilt_iota_k}
Let $w \in S(a)$, set
$\mu := \gcd(w_0,w_1,w_2)$ and consider
a degree matrix in $\ZZ \oplus \ZZ / \mu \ZZ$
associated with $w$ of the following shape
(last row omitted for $a \ge 5$):
\begin{equation*}
Q
\ = \
[q_0, q_1, q_2] 
\ = \
\left[
\begin{array}{ccc}
x_0^2 & \xi_1 x_1^2 & \xi_2 x_2^2
\\[3pt]
\bar 0 & \bar 1 & \bar \eta 
\end{array}
\right],
\qquad
(\bar{\eta} \in \ZZ/ \mu \ZZ)^*,
\end{equation*}
where $(x_0^2, \xi_1 x_1^2, \xi_2 x_2^2) \in S( \mu a)$
is as in Theorem~\ref{thm:loesungen_Markov_sind_quadratzahlen}.
Then the anticanonical divisor class of the (fake) weighted
projective plane associated with $Q$ is
\[
w_Z 
\ = \
(\sqrt{ \mu a \xi_1 \xi_2} x_0  x_1 x_2, \,  \bar{1} + \bar{\eta})
\ \in \
\Cl(Z)
\ = \
\ZZ \oplus \ZZ / \mu \ZZ.
\]
Moreover, the local Gorenstein index $\iota_k$ of the toric
fixed point $z(k) \in Z(Q)$ equals the minimal positive multiple
$\nu x_k$ such that $\nu x_k w_Z \in \ZZ \cdot q_k$. 
\end{lemma}

\begin{proof}
Proposition~\ref{prop:fwppKsquare} says that $D_0+D_1+D_2$
is an anticanonical divisor on $Z$ for $D_k = V(T_k)$.
Moreover, the divisor class of $D_k$ equals the column $q_k$.
Thus, using Theorem~\ref{thm:loesungen_Markov_sind_quadratzahlen}
for the last equality, we obtain
\[
w_Z 
\ = \
q_0 + q_1 + q_2
\ = \
( x_0^2 + \xi_1 x_1^2 + \xi_2 x_2^2, \,  \bar{1} + \bar{\eta} )
\ = \
( \sqrt{ \mu a \xi_1 \xi_2} x_0  x_1 x_2, \,  \bar{1} + \bar{\eta} ).
\]
The local Gorenstein index $\iota_k$ is the order 
of $w_Z$ in the local class group of $z(k)$, which
according to Proposition~\ref{prop:clandcr} is given
as
\[
\Cl(Z, z(k)) \ = \ (\ZZ \oplus \ZZ / \mu \ZZ) /  \ZZ \cdot q_k.
\]
Thus, the local Gorenstein index of $z(k)$ is the minimal positive
integer $\iota_k$ satisfying $\iota_k w_Z \in \ZZ \cdot q_k$.
By the above description of $w_Z$, the latter implies
\[
\sqrt{ \mu a \xi_1 \xi_2} x_0  x_1 x_2 \iota_k \ \in \  \xi_k x_k^2 \ZZ.
\]  

\medskip

\noindent
\emph{Case 1}: $\sqrt{ \mu a \xi_1 \xi_2} = 3$.
Then $\xi_1 = \xi_2=1$ and $3x_0x_1x_2\iota_k$ is a multiple of $\xi_k x_k^2$.
Moreover, $3,x_0,x_1,x_2$ are pairwise coprime by
Lemma~\ref{lem:a9wimod3}.
We conclude $x_k \mid \iota_k$.

\medskip

\noindent
\emph{Case 2}: $\sqrt{\mu a \xi_1 \xi_2} = 4$. 
Then $\xi_1 = 1$, $\xi_2=2$ and $4x_0x_1x_2\iota_k$ is a
multiple of~$\xi_k x_k^2$.
Pairwise coprimeness of $x_0,x_1,2x_2$ yields $x_0,x_1 \mid \iota_0$.
Further, $x_2$ is odd due to Lemma~\ref{lem:a_gl_8_w_i_mod_8}.
Thus, $x_2 \mid \iota _2$.

\medskip

\noindent
\emph{Case 3}: $\sqrt{\mu a \xi_1 \xi_2} = 5$.
Then $\xi_1 = 1$, $\xi_2=5$ and $5x_0x_1x_2\iota_k$ is a
multiple of~$\xi_k x_k^2$.  
Pairwise coprimeness of $x_0,x_1,5x_2$ directly gives $x_k \mid \iota_k$.

\medskip

\noindent
\emph{Case 4}: $\sqrt{ a  \mu \xi_1 \xi_2} = 6$.
Then $\xi_1 = 2$, $\xi_2=3$ and $6x_0x_1x_2\iota_k$ is a multiple of~$\xi_k x_k^2$.   
Again, pairwise coprimeness of $x_0, 2x_1, 3x_2$ leads to $x_k \mid \iota_k$.
\end{proof}

\begin{remark}
\label{rem:nwinZq}
In the situation of Lemma~\ref{lem:x_k_teilt_iota_k}, set $\xi_0 := 1$
and $\eta_0 := 0$, $\eta_1 := 1$, $\eta_2 := \eta$. Moreover,
fix $n \in \ZZ_{>0}$ and define
\[
\beta_k(n)
\ := \ 
n \cdot \frac{\sqrt{\mu a \xi_1\xi_2}}{\xi_k} \cdot \frac{x_0x_1x_2}{x_k},
\qquad
k = 0,1,2.
\]
Then membership of the multiple $nx_kw_Z$ in $\ZZ \cdot q_k$ is characterized
as follows in terms of modular identity:
\[
n x_k w_Z  \in  \ZZ \cdot q_k
\ \iff \
n x_k w_Z  =  \beta_k(n) q_k
\ \iff \
\bar n  (\bar 1 + \bar \eta) \bar x_k = \overline{\beta_k(n)} \bar \eta_k  \in  \ZZ/\mu \ZZ.
\]
\end{remark}

\begin{proposition}
\label{prop:gisingdeg9865}
The following table lists the divisor class group,  
the adjusted degree matrix~$Q$,
the anticanonical class $w_Z$, the
constellation of local Gorenstein indices
$\iota_k = \iota(z(k))$
and the constellation of $T$-singularities
for the fake weighted projective planes
$Z = Z(Q)$ of degrees $9,8,6,5$:

\bigskip

\begin{center}

{\small
\begin{tabular}{c|c|c|c|c|c}
{\rm ID}
&
$\Cl(Z)$
&
$Q$
&
$w_Z$
&
$(\iota_0,\iota_1,\iota_2)$
&$(\pm,\pm,\pm)$ 
\\[4pt] \hline &&&& \\[-6pt]
(9-1-0)
&
$\ZZ$
&      
$
\begin{array}{c}                          
\left[
\begin{array}{ccc}
x_0^2 & x_1^2 & x_2^2
\end{array}
\right]
\\[3pt]
\scriptstyle (x_0^2,x_1^2,x_2^2) \ \in \ S(9)
\end{array}                  
$
&                  
$
\left[
\begin{array}{c}
3x_0x_1x_2
\end{array}
\right]
$
&
$
\begin{array}{c}
(x_0,x_1,x_2)
\end{array}
$
&
$(+,+,+)$  
\\[14pt] \hline &&&&  \\[-4pt]
(8-1-0)
&
$\ZZ$
&      
$
\begin{array}{c}                          
\left[
\begin{array}{ccc}
x_0^2 & x_1^2 & 2x_2^2
\end{array}
\right]
\\[3pt]
\scriptstyle (x_0^2,x_1^2,2x_2^2) \ \in \ S(8)
\end{array}                  
$
&                  
$
\left[
\begin{array}{c}
4x_0x_1x_2
\end{array}
\right]
$
&
$
\begin{array}{c}
(x_0,x_1,x_2)
\end{array}
$
&
$(+,+,+)$  
\\[14pt] \hline &&&&  \\[-4pt]
(6-1-0)
&
$\ZZ$
&      
$
\begin{array}{c}                          
\left[
\begin{array}{ccc}
x_0^2 & 2x_1^2 & 3x_2^2
\end{array}
\right]
\\[3pt]
\scriptstyle (x_0^2,2x_1^2,3x_2^2) \ \in \ S(6)
\end{array}                  
$
&                  
$
\left[
\begin{array}{c}
6x_0x_1x_2
\end{array}
\right]
$
&
$
\begin{array}{c}
(x_0,x_1,x_2)
\end{array}
$
&
$(+,+,+)$  
\\[14pt] \hline &&&&  \\[-4pt]
(5-1-0)
&
$\ZZ$
&      
$
\begin{array}{c}                          
\left[
\begin{array}{ccc}
x_0^2 & x_1^2 & 5x_2^2
\end{array}
\right]
\\[3pt]
\scriptstyle (x_0^2,x_1^2,5x_2^2) \ \in \ S(5)
\end{array}                  
$
&                  
$
\left[
\begin{array}{c}
5x_0x_1x_2
\end{array}
\right]
$
&
$
\begin{array}{c}
(x_0,x_1,x_2)
\end{array}
$
&
$(+,+,+)$  
\\[14pt] \hline &&&&  \\[-4pt]
\end{tabular}
}

\bigskip

\end{center}
\end{proposition}

\begin{proof}
This is classically known, compare~\cite{HaPro}.
In our setting, one can proceed as follows.
Observe that $x_kw_Z \in \ZZ w_k$ holds
in all cases. 
Thus, Lemma~\ref{lem:x_k_teilt_iota_k} establishes
the claim on the local Gorenstein indices.
Obviously, $x_k^2 \mid w_k$ holds in all cases.
Thus, Proposition~\ref{lem:tsing} shows that
the $z(k)$ are at most $T$-singular.
\end{proof}

\begin{proposition}
\label{prop:gisingdeg4}
The following table lists the divisor class group,  
the adjusted degree matrix~$Q$,
the anticanonical class $w_Z$, the
constellation of local Gorenstein indices
$\iota_k = \iota(z(k))$
and the constellation of $T$-singularities
for the fake weighted projective planes
$Z = Z(Q)$ of degree $4$:

\bigskip

\begin{center}

{\small
\begin{tabular}{c|c|c|c|c|c}
{\rm ID}
&
$\Cl(Z)$
&
$Q$
&
$w_Z$
&
$(\iota_0,\iota_1,\iota_2)$
&$(\pm,\pm,\pm)$ 
\\[4pt] \hline &&&& \\[-6pt]
(4-2-1)
&
$\ZZ \oplus \ZZ / 2\ZZ$
&      
$
\begin{array}{c}                          
\left[
\begin{array}{ccc}
x_0^2 & x_1^2 & 2x_2^2
\\[3pt]
\bar 0 & \bar 1 & \bar 1  
\end{array}
\right]
\\[3pt]
\scriptstyle (x_0^2,x_1^2,2x_2^2) \ \in \ S(8)
\end{array}                  
$
&                  
$
\left[
\begin{array}{c}
4x_0x_1x_2
\\[3pt]
\bar 0  
\end{array}
\right]
$
&
$
\begin{array}{c}
(x_0,x_1,x_2)
\end{array}
$
&
$(+,+,+)$  
\\[14pt] \hline &&&&  \\[-4pt]
\end{tabular}
}

\bigskip

\end{center}
\end{proposition}

\begin{proof}
The ID, the divisor class group and the degree matrix
are taken from Proposition~\ref{prop:classfwppdeg4}.
Lemma~\ref{lem:x_k_teilt_iota_k} provides us with
the desired representation of $w_Z \in \ZZ \oplus \ZZ/2\ZZ$.
Now look at the from Remark~\ref{rem:nwinZq}
characterizing $nx_kw_Z \in \ZZ \cdot q_k$ by means of
conditions in $\ZZ/2\ZZ$. We have
\[
\beta_0(n) \ = \ 4nx_1x_2,
\qquad
\beta_1(n) \ = \ 4nx_0x_2,
\qquad
\beta_2(n) \ = \ 2nx_0x_1.
\]
Thus, all conditions become $\bar 0 = \bar 0$
and hold trivially.
In particular, $x_kw_Z \in \ZZ \cdot q_k$
and Lemma~\ref{lem:x_k_teilt_iota_k}
says $\iota_k = x_k$ for $k=0,1,2$.
By Lemma~\ref{lem:Q2localcl}, the local class
group orders of $z(0),z(1),z(2)$ are 
$2x_0^2, 2x_1^2, 4x_2^2$.
Thus, $\iota_k^2$ divides $\cl(z(k))$ in all
cases and Lemma~\ref{lem:tsing} yields
that each $z(k)$ is $T$-singular.
\end{proof}

\begin{proposition}
\label{prop:gisingdeg3}
The following table lists 
the divisor class group,  
the adjusted degree matrix~$Q$,
the anticanonical class $w_Z$,
the  constellation of local Gorenstein indices
$\iota_k = \iota(z(k))$
and the constellation of $T$-singularities
for all fake weighted projective planes $Z = Z(Q)$
of degree $3$:

\bigskip

\begin{center}

{\small
\setlength{\tabcolsep}{4.6pt}
\begin{tabular}{c|c|c|c|c|c}
{\rm ID}
&
$\Cl(Z)$
&
$Q$
&
$w_Z$
&
$(\iota_0,\iota_1,\iota_2)$
&$(\pm,\pm,\pm)$ 
\\[4pt] \hline &&&& \\[-6pt]
(3-3-2)
&
$\ZZ \oplus \ZZ / 3\ZZ$
&      
$
\begin{array}{c}                          
\left[
\begin{array}{ccc}
x_0^2 & x_1^2 & x_2^2
\\[3pt]
\bar 0 & \bar 1 & \bar 2  
\end{array}
\right]
\\[3pt]
\scriptstyle (x_0^2,x_1^2,x_2^2) \ \in \ S(9)
\end{array}                  
$
&                  
$
\left[
\begin{array}{c}
3x_0x_1x_2
\\[3pt]
\bar 0  
\end{array}
\right]
$
&
$
\begin{array}{c}
(x_0,x_1,x_2)
\end{array}
$
&
$(+,+,+)$  
\\[14pt] \hline &&&&  \\[-4pt]
(3-2-1)
&
$\ZZ \oplus \ZZ / 2\ZZ$
&      
$
\begin{array}{c}                          
\left[
\begin{array}{ccc}
x_0^2 & 2x_1^2 & 3x_2^2
\\[3pt]
\bar 0 & \bar 1 & \bar 1  
\end{array}
\right]
\\[3pt]
\scriptstyle (x_0^2,2x_1^2,3x_2^2) \ \in \ S(6)
\end{array}                  
$
&                  
$
\left[
\begin{array}{c}
6x_0x_1x_2
\\[3pt]
\bar 0  
\end{array}
\right]
$
&
$
\begin{array}{c}
(x_0,2x_1,x_2)
\end{array}
$
&
$(+,+,+)$  
\\[14pt] \hline &&&&  \\[-4pt]
\end{tabular}
}

\bigskip

\end{center}
\end{proposition}

\begin{proof}
The first three columns of the table are taken from
Proposition~\ref{prop:classfwppdeg4} and
Lemma~\ref{lem:x_k_teilt_iota_k} yields
the anticanonical classes.

For the ID (3-3-2), we find $\beta_k(n) = 3nx_0x_1x_2/x_k$
for $k=1,2,3$ in Remark~\ref{rem:nwinZq} 
and see that all conditions in $\ZZ / 3\ZZ$
characterizing $nx_kw_Z \in \ZZ \cdot q_k$
are trivially satisfied.
Thus, Lemma~\ref{lem:x_k_teilt_iota_k} shows
$\iota_k = x_k$ for $k=0,1,2$.
By Lemma~\ref{lem:Q2localcl}, the 
local class group order of $z(k)$ is $3x_k^2$.
By Lemms~\ref{lem:tsing}, all $z(k)$
are $T$-singular. 

We turn to the ID (3-2-1).
For $k=0,2$, we see $x_kw_Z \in \ZZ \cdot q_k$,
hence, $\iota_0 = x_0$ and $\iota_2=x_2$,
where we apply Lemma~\ref{lem:x_k_teilt_iota_k} and
Remark~\ref{rem:nwinZq} as before.
Also for $k=1$, we obtain $\beta_1(n) = 3nx_0x_2$
in Remark~\ref{rem:nwinZq} and the 
conditions in $\ZZ/2\ZZ$ characterizing
$nx_1w_Z \in \ZZ \cdot q_1$ are
\[
n=1: \ \bar 0 = \bar x_0 \bar x_2,
\qquad
n=2: \ \bar 0 = \bar 0.
\]
By Theorem~\ref{thm:loesungen_Markov_sind_quadratzahlen},
the numbers $x_0, 2x_1, 3x_2$ are pairwise coprime.
Hence, $x_0x_2$ is odd and the equation for
$n=1$ doesn't hold.
We conclude $\iota_1=2x_1$.
Moreover, the local class group orders are
$2x_0^2$, $4x_1^2$, $6x_2^2$ by Lemma~\ref{lem:Q2localcl}.
Again $\iota_k^2$ divides $\cl(z(k))$ for $k = 0,1,2$
and Lemma~\ref{lem:tsing} ensures that
each $z(k)$ is $T$-singular.
\end{proof}

\begin{proposition}
\label{prop:gisingdeg2}
The following table lists 
the divisor class group,  
the adjusted degree matrix~$Q$,
the anticanonical
class $w_Z$, the possible
constellations of local Gorenstein indices
$\iota_k = \iota(z(k))$
and the constellation of $T$-singularities
for all fake weighted projective planes $Z = Z(Q)$
of degree $2$:

\bigskip

\begin{center}

{\small
\setlength{\tabcolsep}{3.8pt}
\begin{tabular}{c|c|c|c|c|c}
{\rm ID}
&
$\Cl(Z)$
&
$Q$
&
$w_Z$
&
$(\iota_0,\iota_1,\iota_2)$
&$(\pm,\pm,\pm)$ 
\\[4pt] \hline &&&& \\[-6pt]
(2-4-1)
&
$\ZZ \oplus \ZZ / 4\ZZ$
&      
$
\begin{array}{c}                          
\left[
\begin{array}{ccc}
x_0^2 & x_1^2 & 2x_2^2
\\[3pt]
\bar 0 & \bar 1 & \bar 1  
\end{array}
\right]
\\[3pt]
\scriptstyle (x_0^2,x_1^2,2x_2^2) \ \in \ S(8)
\end{array}                  
$
&                  
$
\left[
\begin{array}{c}
4x_0x_1x_2
\\[3pt]
\bar 2  
\end{array}
\right]
$
&
$
\begin{array}{c}
(2x_0,2x_1,x_2)
\end{array}
$
&
$(+,+,+)$  
\\[14pt] \hline &&&&  \\[-4pt]
(2-4-3)
&
$\ZZ \oplus \ZZ / 4\ZZ$
&      
$
\begin{array}{c}                          
\left[
\begin{array}{ccc}
x_0^2 & x_1^2 & 2x_2^2
\\[3pt]
\bar 0 & \bar 1 & \bar 3 
\end{array}
\right]
\\[3pt]
\scriptstyle (x_0^2,x_1^2,2x_2^2) \ \in \ S(8)
\end{array}                  
$
&                  
$
\left[
\begin{array}{c}
4x_0x_1x_2
\\[3pt]
\bar 0  
\end{array}
\right]
$
&
$
\begin{array}{c}
(x_0,x_1,2x_2)
\end{array}
$
&
$(+,+,+)$  
\\[14pt] \hline &&&&  \\[-4pt]
(2-3-1)
&
$\ZZ \oplus \ZZ / 3 \ZZ$
&      
$
\begin{array}{c}                          
\left[
\begin{array}{ccc}
x_0^2 & 2x_1^2 & 3x_2^2
\\[3pt]
\bar 0 & \bar 1 & \bar 1
\end{array}
\right]
\\[3pt]
\scriptstyle (x_0^2,2x_1^2,3x_2^2) \ \in \ S(6)
\end{array}                  
$
&                  
$
\left[
\begin{array}{c}
6x_0x_1x_2
\\[3pt]
\bar 2
\end{array}
\right]
$
&
$
\begin{array}{c}
(3x_0,3x_1,3x_2)
\\[3pt]
(3x_0,3x_1,x_2)
\end{array}
$
&
$(-,-,+)$  
\\[14pt] \hline &&&&  \\[-4pt]
(2-3-2)
&
$\ZZ \oplus \ZZ / 3 \ZZ$
&      
$
\begin{array}{c}                          
\left[
\begin{array}{ccc}
x_0^2 & 2x_1^2 & 3x_2^2
\\[3pt]
\bar 0 & \bar 1 & \bar 2
\end{array}
\right]
\\[3pt]
\scriptstyle (x_0^2,2x_1^2,3x_2^2) \ \in \ S(6)
\end{array}                  
$
&                  
$
\left[
\begin{array}{c}
6x_0x_1x_2
\\[3pt]
\bar 0  
\end{array}
\right]
$
&
$
\begin{array}{c}
(x_0,x_1,3x_2)
\end{array}
$
&
$(+,+,+)$  
\\[14pt] \hline &&&&  \\[-4pt]
\end{tabular}
}

\bigskip

\end{center}
\end{proposition}

\begin{proof}
Proposition~\ref{prop:classfwppdeg2} delivers IDs, divisor class
groups and degree matrices, and Lemma~\ref{lem:x_k_teilt_iota_k}
the anticanonical classes.

For the IDs (2-4-$\ast$), Lemma~\ref{lem:a_gl_8_w_i_mod_4} tells us 
$\bar x_k = \bar 1$ or $\bar x_k = \bar 3$ in $\ZZ/4\ZZ$
for $k=0,1,2$,
and Lemma~\ref{lem:Q2localcl} computes the local class group orders
of $z(0), z(1), z(2)$ as $4x_0^2, 4x_1^2, 8x_2^2$.
Now consider the ID (2-4-1). There, Lemma~\ref{lem:x_k_teilt_iota_k} and
Remark~\ref{rem:nwinZq} characterize the membership
$nx_kw_Z \in \ZZ \cdot q_k$ by the following identities in $\ZZ/4\ZZ$:
\[
\begin{array}{lll}
k=0 \colon & n = 1 \colon \ \bar 2 \bar x_0 = \bar 0, &  n = 2 \colon \ \bar 0 = \bar 0, 
\\
k=1 \colon & n = 1 \colon \ \bar 2 \bar x_1 = \bar 0, &  n = 2 \colon \ \bar 0 = \bar 0, 
\\
k=2 \colon & n = 1 \colon \ \bar 2 \bar x_2 =  \bar 2 \bar x_0 \bar x_1. &
\end{array}  
\] 
The first equations for $k=0,1$ never hold, whereas the one
for $k=2$ always holds.
Thus, $(\iota_0,\iota_1,\iota_2)$ equals $(2x_0,2x_1,x_2)$.
By Lemma~\ref{lem:tsing}, each $z(k)$ 
is $T$-singular.
For the ID (2-4-3), the identities in $\ZZ/4\ZZ$ characterizing
$nx_kw_Z \in \ZZ \cdot q_k$ provided by 
Lemma~\ref{lem:x_k_teilt_iota_k}
and Remark~\ref{rem:nwinZq} are
\[
\begin{array}{lll}
k=0 \colon & n = 1 \colon \ \bar 0 = \bar 0, & 
\\
k=1 \colon & n = 1 \colon \ \bar 0 = \bar 0, &  
\\
k=2 \colon & n = 1 \colon \ \bar 0  =  \bar 2 \bar x_0 \bar x_1, &  n = 2 \colon \ \bar 0 = \bar 0.
\end{array}  
\] 
Thus, $(\iota_0,\iota_1,\iota_2)$ equals $(x_0,x_1,2x_2)$.
Consequently, Lemma~\ref{lem:tsing} tells us
that each of $z(0)$, $z(1)$ and $z(2)$ is $T$-singular.

For the IDs (2-3-$\ast$), Lemma~\ref{lem:a_gl_6_w_i_mod_3}
tells us $\bar x_k = \bar 1$ or $\bar x_k = \bar 2$ for $k=0,1$,
and Lemma~\ref{lem:Q2localcl} computes the local class group orders
of $z(0), z(1), z(2)$ as $3x_0^2, 6x_1^2, 9x_2^2$.
For the ID (2-3-1), Lemma~\ref{lem:x_k_teilt_iota_k} and
Remark~\ref{rem:nwinZq} characterize $nx_kw_Z \in \ZZ \cdot q_k$ by
the following identities in $\ZZ/3\ZZ$:
\[
\begin{array}{llll}
k=0 \colon & n = 1 \colon \ \bar 2 \bar x_0 = \bar 0, & n = 2 \colon \ \bar x_0 = \bar 0, & n = 3 \colon \ \bar 0 = \bar 0, 
\\
k=1 \colon & n = 1 \colon \ \bar 2 \bar x_1 = \bar 0, & n = 2 \colon \ \bar x_1 = \bar 0, & n = 3 \colon \ \bar 0 = \bar 0, 
\\
k=2 \colon & n = 1 \colon \ \bar 2 \bar x_2 =  \bar 2 \bar x_0 \bar x_1, &  n = 2 \colon  \bar x_2 =   \bar x_0 \bar x_1, & n = 3 \colon \ \bar 0 = \bar 0.
\end{array}  
\] 
This implies $\iota_0=3x_0$ and $\iota_1 = 3x_1$.
For $k=2$, the first two equations are
equivalent, which excludes $\iota_2=2$ and leaves
us with $\iota_2 = x_2, 3x_2$.
By Lemma~\ref{lem:tsing}, only $z(2)$ is $T$-singular.
For the ID (2-3-2),
the identities in $\ZZ/3\ZZ$
characterizing $nx_kw_Z \in \ZZ \cdot q_k$ are
\[
\begin{array}{llll}
k=0 \colon & n = 1 \colon \ \bar 0 = \bar 0, & & 
\\
k=1 \colon & n = 1 \colon \ \bar 0 = \bar 0, & &
\\
k=2 \colon & n = 1 \colon \ \bar 0 =  \bar x_0 \bar x_1, &  n = 2 \colon  \bar 0 =   \bar 2 \bar x_0 \bar x_1, & n = 3 \colon \ \bar 0 = \bar 0.
\end{array}  
\] 
We conclude $\iota_0 = x_0$, $\iota_1 = x_1$ and $\iota_2 = 3x_2$.
Stressing once more Lemma~\ref{lem:tsing},
we see that all points $z(k)$ are $T$-singularities.
\end{proof}

\begin{proposition}
\label{prop:gisingdeg1}
The following table lists the divisor class group,  
the adjusted degree matrix~$Q$, anticanonical
class $w_Z$, the possible constellations of local
Gorenstein indices $\iota_k = \iota(z(k))$
and the constellation of $T$-singularities
for all fake weighted projective planes $Z=Z(Q)$
of degree $1$:

\bigskip

\begin{center}

{\small
\setlength{\tabcolsep}{3.8pt}
\begin{tabular}{c|c|c|c|c|c}
{\rm ID}
&
$\Cl(Z)$
&
$Q$
&
$w_Z$
&
$(\iota_0,\iota_1,\iota_2)$
&$(\pm,\pm,\pm)$ 
\\[4pt] \hline &&&& \\[-6pt]
(1-9-2)
&
$\ZZ \oplus \ZZ / 9\ZZ$
&      
$
\begin{array}{c}                          
\left[
\begin{array}{ccc}
x_0^2 & x_1^2 & x_2^2
\\[3pt]
\bar 0 & \bar 1 & \bar 2  
\end{array}
\right]
\\[3pt]
\scriptstyle (x_0^2,x_1^2,x_2^2) \ \in \ S(9)
\end{array}                  
$
&                  
$
\left[
\begin{array}{c}
3x_0x_1x_2
\\[3pt]
\bar 3  
\end{array}
\right]
$
&
$
\begin{array}{c}
(3x_0, x_1, 3x_2)
\\[3pt]
(3x_0, 3x_1, x_2)
\end{array}
$
&
$(+,+,+)$  
\\[14pt] \hline &&&&  \\[-4pt]
(1-9-5)
&
$\ZZ \oplus \ZZ / 9\ZZ$
&      
$
\begin{array}{c}                          
\left[
\begin{array}{ccc}
x_0^2 & x_1^2 & x_2^2
\\[3pt]
\bar 0 & \bar 1 & \bar 5
\end{array}
\right]
\\[3pt]
\scriptstyle (x_0^2,x_1^2,x_2^2) \ \in \ S(9)
\end{array}                  
$
&                  
$
\left[
\begin{array}{c}
3x_0x_1x_2
\\[3pt]
\bar 6  
\end{array}
\right]
$
&
$
\begin{array}{c}
(3x_0, x_1, 3 x_2)
\\[3pt]
(3x_0, 3x_1, x_2)
\end{array}
$
&
$(+,+,+)$  
\\[14pt] \hline &&&&  \\[-4pt]
(1-9-8)
&
$\ZZ \oplus \ZZ / 9 \ZZ$
&      
$
\begin{array}{c}                          
\left[
\begin{array}{ccc}
x_0^2 & x_1^2 & x_2^2
\\[3pt]
\bar 0 & \bar 1 & \bar 8
\end{array}
\right]
\\[3pt]
\scriptstyle (x_0^2,x_1^2,x_2^2) \ \in \ S(9)
\end{array}                  
$
&                  
$
\left[
\begin{array}{c}
3x_0x_1x_2
\\[3pt]
\bar 0
\end{array}
\right]
$
&
$
\begin{array}{c}
(x_0,3x_1,3x_2)
\end{array}
$
&
$(+,+,+)$  
\\[14pt] \hline &&&&  \\[-4pt]
(1-8-1)
&
$\ZZ \oplus \ZZ / 8 \ZZ$
&      
$
\begin{array}{c}                          
\left[
\begin{array}{ccc}
x_0^2 & x_1^2 & 2x_2^2
\\[3pt]
\bar 0 & \bar 1 & \bar 1
\end{array}
\right]
\\[3pt]
\scriptstyle (x_0^2,x_1^2,2x_2^2) \ \in \ S(8)
\end{array}                  
$
&                  
$
\left[
\begin{array}{c}
4x_0x_1x_2
\\[3pt]
\bar 2  
\end{array}
\right]
$
&
$
\begin{array}{c}
(4x_0,4x_1,x_2)
\\[3pt]
(4x_0,4x_1,2x_2)
\end{array}
$
&
$(-,-,+)$  
\\[14pt] \hline &&&&  \\[-4pt]
(1-8-3)
&
$\ZZ \oplus \ZZ / 8 \ZZ$
&      
$
\begin{array}{c}                          
\left[
\begin{array}{ccc}
x_0^2 & x_1^2 & 2x_2^2
\\[3pt]
\bar 0 & \bar 1 & \bar 3
\end{array}
\right]
\\[3pt]
\scriptstyle (x_0^2,x_1^2,2x_2^2) \ \in \ S(8)
\end{array}                  
$
&                  
$
\left[
\begin{array}{c}
4x_0x_1x_2
\\[3pt]
\bar 4  
\end{array}
\right]
$
&
$
\begin{array}{c}
(2x_0,x_1,4x_2)
\end{array}
$
&
$(+,+,+)$  
\\[14pt] \hline &&&&  \\[-4pt]
(1-8-5)
&
$\ZZ \oplus \ZZ / 8 \ZZ$
&      
$
\begin{array}{c}                          
\left[
\begin{array}{ccc}
x_0^2 & x_1^2 & 2x_2^2
\\[3pt]
\bar 0 & \bar 1 & \bar 5
\end{array}
\right]
\\[3pt]
\scriptstyle (x_0^2,x_1^2,2x_2^2) \ \in \ S(8)
\end{array}                  
$
&                  
$
\left[
\begin{array}{c}
4x_0x_1x_2
\\[3pt]
\bar 6  
\end{array}
\right]
$
&
$
\begin{array}{c}
(4x_0,4x_1,x_2)
\\[3pt]
(4x_0,4x_1,2x_2)
\end{array}
$
&
$(-,-,+)$  
\\[14pt] \hline &&&&  \\[-4pt]
(1-8-7)
&
$\ZZ \oplus \ZZ / 8 \ZZ$
&      
$
\begin{array}{c}                          
\left[
\begin{array}{ccc}
x_0^2 & x_1^2 & 2x_2^2
\\[3pt]
\bar 0 & \bar 1 & \bar 7
\end{array}
\right]
\\[3pt]
\scriptstyle (x_0^2,x_1^2,2x_2^2) \ \in \ S(8)
\end{array}                  
$
&                  
$
\left[
\begin{array}{c}
4x_0x_1x_2
\\[3pt]
\bar 0
\end{array}
\right]
$
&
$
\begin{array}{c}
(x_0,2x_1,4x_2)
\end{array}
$
&
$(+,+,+)$  
\\[14pt] \hline &&&&  \\[-4pt]
(1-6-1)
&
$\ZZ \oplus \ZZ / 6 \ZZ$
&      
$
\begin{array}{c}                          
\left[
\begin{array}{ccc}
x_0^2 & 2x_1^2 & 3x_2^2
\\[3pt]
\bar 0 & \bar 1 & \bar 1
\end{array}
\right]
\\[3pt]
\scriptstyle (x_0^2,2x_1^2,3x_2^2) \ \in \ S(6)
\end{array}                  
$
&                  
$
\left[
\begin{array}{c}
6x_0x_1x_2
\\[3pt]
\bar 2
\end{array}
\right]
$
&
$
\begin{array}{c}
(3x_0,6x_1,x_2)
\\[3pt]
(3x_0,6x_1,3x_2)
\end{array}
$
&
$(-,-,+)$  
\\[14pt] \hline &&&&  \\[-4pt]
(1-6-5)
&
$\ZZ \oplus \ZZ / 6 \ZZ$
&      
$
\begin{array}{c}                          
\left[
\begin{array}{ccc}
x_0^2 & 2x_1^2 & 3x_2^2
\\[3pt]
\bar 0 & \bar 1 & \bar 5
\end{array}
\right]
\\[3pt]
\scriptstyle (x_0^2,2x_1^2,3x_2^2) \ \in \ S(6)
\end{array}                  
$
&                  
$
\left[
\begin{array}{c}
6x_0x_1x_2
\\[3pt]
\bar 0
\end{array}
\right]
$
&
$
\begin{array}{c}
(x_0,2x_1,3x_2)
\end{array}
$
&
$(+,+,+)$  
\\[14pt] \hline &&&&  \\[-4pt]
(1-5-1)
&
$\ZZ \oplus \ZZ / 5 \ZZ$
&      
$
\begin{array}{c}                          
\left[
\begin{array}{ccc}
x_0^2 & x_1^2 & 5x_2^2
\\[3pt]
\bar 0 & \bar 1 & \bar 1
\end{array}
\right]
\\[3pt]
\scriptstyle (x_0^2,x_1^2,5x_2^2) \ \in \ S(5)
\end{array}                  
$
&                  
$
\left[
\begin{array}{c}
5x_0x_1x_2
\\[3pt]
\bar 2
\end{array}
\right]
$
&
$
\begin{array}{c}
(5x_0,5x_1,x_2)
\\[3pt]
(5x_0,5x_1,5x_2)  
\end{array}
$
&
$(-,-,+)$  
\\[14pt] \hline &&&&  \\[-4pt]
(1-5-2)
&
$\ZZ \oplus \ZZ / 5 \ZZ$
&      
$
\begin{array}{c}                          
\left[
\begin{array}{ccc}
x_0^2 & x_1^2 & 5x_2^2
\\[3pt]
\bar 0 & \bar 1 & \bar 2
\end{array}
\right]
\\[3pt]
\scriptstyle (x_0^2,x_1^2,5x_2^2) \ \in \ S(5)
\end{array}                  
$
&                  
$
\left[
\begin{array}{c}
5x_0x_1x_2
\\[3pt]
\bar 3
\end{array}
\right]
$
&
$
\begin{array}{c}
(5x_0,5x_1,x_2)
\\[3pt]
(5x_0,5x_1,5x_2)  
\end{array}
$
&
$(-,-,+)$  
\\[14pt] \hline &&&&  \\[-4pt]
(1-5-3)
&
$\ZZ \oplus \ZZ / 5 \ZZ$
&      
$
\begin{array}{c}                          
\left[
\begin{array}{ccc}
x_0^2 & x_1^2 & 5x_2^2
\\[3pt]
\bar 0 & \bar 1 & \bar 3
\end{array}
\right]
\\[3pt]
\scriptstyle (x_0^2,x_1^2,5x_2^2) \ \in \ S(5)
\end{array}                  
$
&                  
$
\left[
\begin{array}{c}
5x_0x_1x_2
\\[3pt]
\bar 4
\end{array}
\right]
$
&
$
\begin{array}{c}
(5x_0,5x_1,x_2)
\\[3pt]
(5x_0,5x_1,5x_2)  
\end{array}
$
&
$(-,-,+)$  
\\[14pt] \hline &&&&  \\[-4pt]
(1-5-4)
&
$\ZZ \oplus \ZZ / 5 \ZZ$
&      
$
\begin{array}{c}                          
\left[
\begin{array}{ccc}
x_0^2 & x_1^2 & 5x_2^2
\\[3pt]
\bar 0 & \bar 1 & \bar 4
\end{array}
\right]
\\[3pt]
\scriptstyle (x_0^2,x_1^2,5x_2^2) \ \in \ S(5)
\end{array}                  
$
&                  
$
\left[
\begin{array}{c}
5x_0x_1x_2
\\[3pt]
\bar 0
\end{array}
\right]
$
&
$
\begin{array}{c}
(x_0,x_1,5x_2)
\end{array}
$
&
$(+,+,+)$  
\\[14pt] \hline &&&&  \\[-4pt]
\end{tabular}
}

\bigskip

\end{center}
\end{proposition}

\begin{proof}
In all cases, ID, divisor class group and 
degree matrix are as in Proposition~\ref{prop:classfwppdeg4}
and Lemma~\ref{lem:x_k_teilt_iota_k} delivers 
the anticanonical classes.
The subsequent proof tacitly uses Lemma~\ref{lem:Q2localcl}
for the local class group orders,
Lemma~\ref{lem:x_k_teilt_iota_k} and Remark~\ref{rem:nwinZq}
for the local Gorenstein indices
and Lemma~\ref{lem:tsing} for the $T$-singularities.

\medskip
\noindent
\emph{IDs (1-9-$\ast$)}:
Here $\xi_k = 1$, $\cl(z(k)) = 9x_k^2$ and
Lemma~\ref{lem:a9wimod3} yields $\bar x_k \in (\ZZ/9\ZZ)^*$
for $k=0,1,2$.
For the ID (1-9-2), the conditions characterizing 
$nx_kw_Z \in \ZZ \cdot q_k$ are
\[
\begin{array}{llll}
k=0 \colon & n = 1 \colon \ \bar 3 \bar x_0 = \bar 0, & n = 2 \colon \ \bar 6\bar x_0 = \bar 0, & n = 3 \colon \ \bar 0 = \bar 0, 
\\
k=1 \colon & n = 1 \colon \ \bar 3 \bar x_1 = \bar 3 \bar x_0 \bar x_2 & n = 2 \colon \ \bar 6 \bar x_1 = \bar 6 \bar x_0 \bar x_2, & n = 3 \colon \ \bar 0 = \bar 0, 
\\
k=2 \colon & n = 1 \colon \ \bar 3 \bar x_2 =  \bar 6 \bar x_0 \bar x_1, &  n = 2 \colon \ \bar 6 \bar x_2 =   3 \bar x_0 \bar x_1, & n = 3 \colon \ \bar 0 = \bar 0.
\end{array}  
\] 
Thus, $\iota_0=3x_0$. Checking all $(\bar x_0, \bar x_1, \bar x_3)$,
$\bar x_k \in (\ZZ/9\ZZ)^*$ shows that $(\iota_1,\iota_2)$
is one of $(x_1,3x_2)$, $(3x_1,x_2)$.
Hence, all $z(k)$ are $T$-singular.
For the ID (1-9-5), we get
\[
\begin{array}{llll}
k=0 \colon & n = 1 \colon \ \bar 6 \bar x_0 = \bar 0, & n = 2 \colon \ \bar 3 \bar x_0 = \bar 0, & n = 3 \colon \ \bar 0 = \bar 0, 
\\
k=1 \colon & n = 1 \colon \ \bar 6 \bar x_1 = \bar 3 \bar x_0 \bar x_2 & n = 2 \colon \ \bar 3 \bar x_1 = \bar 6 \bar x_0 \bar x_2, & n = 3 \colon \ \bar 0 = \bar 0, 
\\
k=2 \colon & n = 1 \colon \ \bar 6 \bar x_2 =  \bar 6 \bar x_0 \bar x_1, &  n = 2 \colon \ \bar 3 \bar x_2 =   3 \bar x_0 \bar x_1, & n = 3 \colon \ \bar 0 = \bar 0.
\end{array}  
\] 
as characterizing conditions for $nx_kw_Z \in \ZZ \cdot q_k$,
which are equivalent to those of the preceding case after
switching $k=1,2$.
For the ID (1-9-8), the conditions
\[
\begin{array}{llll}
k=0 \colon & n = 1 \colon \ \bar 0 = \bar 0, &  &  
\\
k=1 \colon & n = 1 \colon \ \bar 0 = \bar 3 \bar x_0 \bar x_2, & n = 2 \colon \ \bar 0 = \bar 6 \bar x_0 \bar x_2, & n = 3 \colon \ \bar 0 = \bar 0, 
\\
k=2 \colon & n = 1 \colon \ \bar 0 = \bar 6 \bar x_0 \bar x_1, & n = 2 \colon \ \bar 0 = \bar 3 \bar x_0 \bar x_1, & n = 3 \colon \ \bar 0 = \bar 0
\end{array}  
\] 
characterize $nx_kw_Z \in \ZZ \cdot q_k$. We directly
conclude $\iota_0 = x_0$,  $\iota_1 = 3x_1$,  $\iota_2 = 3x_2$
and see that all $z(k)$ are $T$-singular.

\medskip
\noindent
\emph{IDs (1-8-$\ast$)}:
Here $\xi_0,\xi_1 = 1$, $\xi_2 = 2$, the local class
group orders are $8x_0^2,8x_1^2,16x_2^2$ and 
Lemma~\ref{lem:a_gl_8_w_i_mod_4} implies
$\bar x_k = \bar 1, \bar 3, \bar 5, \bar 7 \in \ZZ/8\ZZ$.
In particular, we always have $\bar 4 \bar x_k = \bar 4$.
For the ID (1-8-1), the conditions characterizing
$nx_kw_Z \in \ZZ \cdot q_k$ are 
\[
\begin{array}{llll}
k=0 \colon & n = 1 \colon \ \bar 2 \bar x_0 = \bar 0, &  n = 2 \colon \ \bar 4 \bar x_0 = \bar 0,  &  n = 3 \colon \ \bar 6 \bar x_0 = \bar 0,
\\
k=1 \colon & n = 1 \colon \ \bar 2 \bar x_1 = \bar 4 \bar x_0 \bar x_2, & n = 2 \colon \ \bar 4 \bar x_1 = \bar 0 & n = 3 \colon \ \bar 6 \bar x_1 = \bar 4 \bar x_0 \bar x_2, 
\\
k=2 \colon & n = 1 \colon \ \bar 2 \bar x_2 = \bar 2 \bar x_0 \bar x_1, & n = 2 \colon \ \bar 4 \bar x_2 = \bar 4 \bar x_0 \bar x_1, & n = 3 \colon \ \bar 6 \bar x_2 = \bar 6 \bar x_0 \bar x_1
\end{array}  
\] 
and $\bar 0 = \bar 0$ for $n=4$ and all $k$.
Using $\bar x_k \in (\ZZ/8\ZZ)^*$ and $\bar 4 \bar x_k = \bar 4$,
we directly see $(\iota_0,\iota_1) = (4x_0,4x_1)$ and further
obtain $\iota_2 = x_2,2x_2$.
Consequently, $z(0)$ and $z(1)$ are not $T$-singular,
whereas $z(2)$ is a $T$-singularity.
For the ID (1-8-3), we get
\[
\begin{array}{llll}
k=0 \colon & n = 1 \colon \ \bar 4 \bar x_0 = \bar 0, &  n = 2 \colon \ \bar 0 = \bar 0,  & 
\\
k=1 \colon & n = 1 \colon \ \bar 4 \bar x_1 = \bar 4 \bar x_0 \bar x_2, &  & 
\\
k=2 \colon & n = 1 \colon \ \bar 4 \bar x_2 = \bar 6 \bar x_0 \bar x_1, & n = 2 \colon \ \bar 0 = \bar 4 \bar x_0 \bar x_1, & n = 3 \colon \ \bar 4 \bar x_2 = \bar 2 \bar x_0 \bar x_1
\end{array}  
\] 
and $\bar 0 = \bar 0$ for $n=4$ and $k=2$.
Thus, using again $\bar x_k \in (\ZZ/8\ZZ)^*$ and $\bar 4 \bar x_k = \bar 4$,
we see $(\iota_0,\iota_1,\iota_2) = (2x_0,x_1,4x_1)$.
Hence all $z(k)$ are $T$-singular.
For the ID (1-8-5), the conditions characterizing
$nx_kw_Z \in \ZZ \cdot q_k$ are 
\[
\begin{array}{llll}
k=0 \colon & n = 1 \colon \ \bar 6 \bar x_0 = \bar 0, &  n = 2 \colon \ \bar 4 \bar x_0 = \bar 0,  &  n = 3 \colon \ \bar 2 \bar x_0 = \bar 0,
\\
k=1 \colon & n = 1 \colon \ \bar 6 \bar x_1 = \bar 4 \bar x_0 \bar x_2, & n = 2 \colon \ \bar 4 \bar x_1 = \bar 0 & n = 3 \colon \ \bar 2 \bar x_1 = \bar 4 \bar x_0 \bar x_2, 
\\
k=2 \colon & n = 1 \colon \ \bar 6 \bar x_2 = \bar 2 \bar x_0 \bar x_1, & n = 2 \colon \ \bar 4 \bar x_2 = \bar 4 \bar x_0 \bar x_1, & n = 3 \colon \ \bar 2 \bar x_2 = \bar 6 \bar x_0 \bar x_1
\end{array}  
\] 
and $\bar 0 = \bar 0$ for $n=4$ and all $k$. Multiplication by $\bar 3$
transforms the cases $k=0,1$ into the corresponding ones of 1-8-1,
hence $(\iota_0,\iota_1) = (4x_0,4x_1)$.
Moreover, using $\bar 4  \bar x_k = \bar 4$, we see $\iota_2 = x_2,2x_2$.
The assertions follow.
For the ID (1-8-7), we have
\[
\begin{array}{llll}
k=0 \colon & n = 1 \colon \ \bar 0 = \bar 0, &   &  
\\
k=1 \colon & n = 1 \colon \ \bar 0 = \bar 4 \bar x_0 \bar x_2, & n = 2 \colon \ \bar 0 = \bar 0, & 
\\
k=2 \colon & n = 1 \colon \ \bar 0 = \bar 6 \bar x_0 \bar x_1, & n = 2 \colon \ \bar 0 = \bar 4 \bar x_0 \bar x_1, & n = 3 \colon \ \bar 0 = \bar 2 \bar x_0 \bar x_1
\end{array}  
\] 
and $\bar 0 = \bar 0$ for $n=4$ and $k=2$.
As a direct consequence, we obtain $\iota_0=x_0$, $\iota_1=2x_1$ and  $\iota_2 = 4x_2$,
as claimed. Moreover, each of the points $z(k)$ is a $T$-singularity.

\medskip
\noindent
\emph{IDs (1-6-$\ast$)}:
Here $\xi_0=1$, $\xi_1=2$, $\xi_2=3$, we have $6x_0^2$, $12x_1^2$, $18x_2^2$
as local class group orders and Lemma~\ref{lem:a_gl_6_w_i_mod_6}
says $\bar x_0  = \bar 1, \bar 5$,
$\bar x_1  \ne \bar 3$ and $\bar x_2 \ne \bar 2, \bar 4$ for
the values in $\ZZ/6\ZZ$.
For the ID (1-6-1), the conditions characterizing 
$nx_kw_Z \in \ZZ \cdot q_k$ are
\[
\begin{array}{llll}
k=0 \colon & n = 1 \colon \ \bar 2 \bar x_0 = \bar 0, & n = 2 \colon \ \bar 4 \bar x_0 = \bar 0, & n = 3 \colon \ \bar 0 = \bar 0, 
\\
k=1 \colon & n = 1 \colon \ \bar 2 \bar x_1 = \bar 3 \bar x_0 \bar x_2 & n = 2 \colon \ \bar 4 \bar x_1 = \bar 0, & n = 3 \colon \ \bar 0 = \bar 3 \bar x_0 \bar x_2, 
\\
           & n = 4 \colon \ \bar 2 \bar x_1 = \bar 0, & n = 5 \colon \ \bar 4 \bar x_0 = \bar 3 \bar x_0 \bar x_2,& n= 6 \colon \ \bar 0 = \bar 0,                                          \\
k=2 \colon & n = 1 \colon \ \bar 2 \bar x_2 =  \bar 2 \bar x_0 \bar x_1, &  n = 2 \colon \ \bar 4 \bar x_2 =   4 \bar x_0 \bar x_1, & n = 3 \colon \ \bar 0 = \bar 0.
\end{array}  
\]
We directly see $\iota_0 = 3x_0$.
For $k=1$, the possible values $\bar x_k \in \ZZ/6\ZZ$
can only satisfy the last equation.
Hence $\iota_1 = 6x_1$.
Finally, $\iota_2 = x_2,3x_2$ by equivalence of the
first two equations in the case $k=2$.
For the ID (1-6-5), we get
\[
\begin{array}{llll}
k=0 \colon & n = 1 \colon \ \bar 0 = \bar 0, &  & 
\\
k=1 \colon & n = 1 \colon \ \bar 0 = \bar 3 \bar x_0 \bar x_2 & n = 2 \colon \ \bar 0 = \bar 0, & 
\\
k=2 \colon & n = 1 \colon \ \bar 0 =  \bar 4 \bar x_0 \bar x_1, &  n = 2 \colon \ \bar 0 =   2 \bar x_0 \bar x_1, & n = 3 \colon \ \bar 0 = \bar 0.
\end{array}  
\]
as conditions characterizing $x_kw_Z, 2x_kw_Z \in \ZZ \cdot q_k$.
Obviously, $\iota_0 = x_0$.
Moreover, according to the possible values $\bar x_k \in \ZZ/6\ZZ$,
we have $\iota_1 = 2x_1$ and $\iota_2 = 3x_2$, as claimed.
Consequently, all $z(k)$ are $T$-singularities.

\medskip
\noindent
\emph{IDs (1-5-$\ast$)}:
Lemma~\ref{lem:a_gl_5_w_i_mod_5} yields $\bar x_0, \bar x_1 \in (\ZZ/5\ZZ)^*$.
We treat (1-5-1), (1-5-2), (1-5-3) at once. For $k=0,1$,
the conditions in $\ZZ/5\ZZ$ 
characterizing $n x_kw_Z \in \ZZ \cdot q_k$ are 
\[
\bar n (\bar 1 + \bar s) \bar x_k \ = \ \bar 0,
\quad
n = 1,2,3,4,5, \ s = 1,2,3.
\]
These are only satisfied for $n=5$, showing
$\iota_0 = 5 x_0$ and $\iota_1 = 5x_1$. 
Moreover, we obtain $\iota_2 = x_2, 5x_2$
as the conditions characterizing 
$nx_2w_Z \in \ZZ \cdot q_2$ are 
\[
\bar n (\bar 1 + \bar s) \bar x_2  \ =  \ \bar n \bar s \bar x_0 \bar x_1,
\quad
n = 1,2,3,4,5, \ s = 1,2,3.
\]
For the ID (1-5-4), we directly see
$x_kw_Z \in \ZZ \cdot q_k$, hence $\iota_k = x_k$
for $k=0,1$.
Moreover, $nx_kw_Z \in \ZZ \cdot q_k$ is characterized
by $\bar 4 \bar n \bar x_0 \bar x_1 = \bar 0$.
Thus, $\iota_2 = 5x_2$.
\end{proof}


\section{$\KK^*$-surfaces of Picard number one and integral degree}
\label{sec:ratcstar}

We explicitly describe all quasismooth projective rational
$\KK^*$-surfaces of Picard number one and integral degree;
see Theorems~\ref{thm:adjpartner} and~\ref{thm:thm4}, covering
in particular Theorem~\ref{thm:introthm3} from the introduction.
An important step on the way is Proposition~\ref{prop:P1toP},
showing, roughly speaking, that a fake weighted projective
plane is an equivariant limit of a quasismooth $\KK^*$-surface
of the same Picard number and degree provided that at
least one of its toric fixed points is a $T$-singularity.
The second main ingredient is the study of the singularities
of the fake weighted projective planes performed in the
previous section.

We first recall the necessary background on $\KK^*$-surfaces
of Picard number one; see~\cite[Sec.~4]{HaHaHaSp} and \cite{HaHaSp}
for more details. 
These come embedded into certain fake weighted projective spaces.
More precisely, consider an integral $3 \times 4$ projective
generator matrix of the shape
\[
P
\ = \
\left[
\begin{array}{rrrr}
-1 & -1 & l_1 & 0
\\
-1 & -1 & 0 & l_2
\\
0 & d_0 & d_1 & d_2
\end{array}
\right],
\qquad
\begin{array}{c}
\scriptstyle
1 \le d_1 \le l_1 \le l_2, \ \ \gcd(l_i,d_i) = 1, 
\\[3pt]
\scriptstyle
d_0 + \frac{d_1}{l_1} + \frac{d_2}{l_2} \ < \ 0 \ < \ \frac{d_1}{l_1} + \frac{d_2}{l_2}.
\end{array}
\]
Let $Z = Z(P)$ be the associated fake weighted projective space.
Observe that the fake weight vector of $Z$ is given as
\[
w
\ = \   
w(P)  
\ = \
(-l_1l_2d_0 - l_2d_1 - l_1d_2, \, l_2d_1 + l_1d_2, \, - l_2d_0, \, -l_1d_0)
\ \in \
\ZZ_{>0}^4.
\]
Now, denote by $S_1,S_2,S_3$ the coordinates on the acting
torus $\TT^3 \subseteq Z$. Then we obtain a projective surface 
\[
X \ := \ X(P) \ := \  \overline{V(1+S_1+S_2)} \ \subseteq \ Z,
\]
taking the closure of $V(1+S_1+S_2) \subseteq \TT^3$ in $Z$.
The multiplicative group $\KK^*$ acts on $X$ as a subtorus of
the torus $\TT^3 \subseteq Z(P)$ via 
\[
t \cdot s \ = \ (s_1,s_2,ts_3).
\]
The $\KK^*$-surface $X$ is projective, rational, quasismooth,
del Pezzo and of Picard number one.
In homogeneous coordinates on $Z$, we have
\[
X
\ = \
V(T_1T_2 + T_3^{l_1} + T_4^{l_2}) 
\ \subseteq \
Z.
\]
Moreover, $X$ is non-toric if and only if $l_1,l_2  > 1$.
There are precisely three $\KK^*$-fixed points on $X$;
in homogeneous coordinates they are given as
\[
x_0 \ = \ [0,0,\zeta,1], \ \zeta^{l_1} = -1,
\qquad
x_1 \ = \ [0,1,0,0],
\qquad
x_2  \ = \ [1,0,0,0].
\]
The fixed point $x_0$ is hyperbolic and $x_1,x_2$ are both
elliptic.
There are exactly two non-trivial orbits $\KK^* \cdot z_1$
and $\KK^* \cdot z_2$ with non-trivial isotropy groups:
\[
z_1  = [-1,1,0,1], \quad  
\vert \KK^*_{z_1} \vert = l_1,
\qquad\qquad
z_2  = [-1,1,1,0], \quad
\vert \KK^*_{z_2} \vert = l_2.
\]
Moreover, for the local class group orders of the three
$\KK^*$-fixed points $x_0,x_1,x_2 \in X$, we obtain
\[
\cl(x_0)  =  -d_0, 
\qquad
\cl(x_1)  =  w_2, 
\qquad
\cl(x_2)  =  w_1. 
\]
Finally, the self intersection number of the canonical
divisor $\mathcal{K}_X$ on $X$ can be expressed as
follows:
\[
\mathcal{K}_X^2
\ = \
\Bigl( \frac{1}{w_1} + \frac{1}{w_2} \Bigr)
\Bigl(2 + \frac{l_1}{l_2} + \frac{l_2}{l_1} \Bigr)
\ = \
\frac{\cl(x_0)}{\cl(x_1)\cl(x_2)} (l_1+l_2)^2.
\]
As a consequence of \cite[Thm.~4.18, Props.~4.9, 4.15 and~5.1]{HaHaHaSp},
every non-toric, quasismooth, rational, projective $\KK^*$-surface
of Picard number one is isomorphic to an $X(P)$ as above.

\begin{construction}
\label{constr:degs}
See~\cite[Constr.~3.5]{HaKiWr}, compare also~\cites{AkKa,Il}.
Consider a $\KK^*$-surface $X=X(P)$ in $Z=Z(P)$ as above
and set
\[
\mathcal{X}_1 
\ := \
V(T_1T_2 + ST_3^{l_1} + T_4^{l_2})
\ \subseteq \ 
Z \times \KK,
\]
\[
\mathcal{X}_2
\ := \ 
V(T_1T_2 + T_3^{l_1} + ST_4^{l_2})
\ \subseteq \
Z \times \KK,
\]
where the $T_i$ are the homogeneous coordinates
on $Z$ and $S$ is the coordinate on $\KK$.
Then $\mathcal{X}_1$ and $\mathcal{X}_2$
are invariant under the respective $\KK^*$-actions
on $Z \times \KK$ given by 
\[
\vartheta \cdot ([z_1,z_2,z_3,z_4],s)
\ = \
([z_1,z_2,\vartheta^{-1}z_3,z_4],\vartheta s),
\]
\[
\vartheta \cdot ([z_1,z_2,z_3,z_4],s)
\ = \
([z_1,z_2,z_3,\vartheta^{-1}z_4],\vartheta s).
\]
Restricting the projection $Z \times \KK \to \KK$
yields flat families
$\psi_i \colon \mathcal{X}_i \to \KK$
being compatible with the above $\KK^*$-actions
and the scalar multiplication on $\KK$.
Set
\[
\tilde P_1
\ := \
\left[
\begin{array}{ccc}
l_1 & l_1 & -l_2 
\\ 
d_1 & d_1+l_1d_0 & d_2 
\end{array}
\right],
\qquad
\tilde P_2
\ := \
\left[
\begin{array}{ccc}
l_2 & l_2 &  -l_1 
\\
d_2 & d_2+l_2d_0 & d_1 
\end{array}
\right]
\]
and let $\tilde Z_1, \tilde Z_2$ denote 
the associated fake weighted projective planes.
Then the central fiber $\psi_i^{-1}(0)$
equals $\tilde Z_i$ and any other fiber
$\psi_i^{-1}(s)$ is isomorphic to $X$.
\end{construction}

\begin{proposition}
\label{prop:degprops}
Consider $X=X(P)$ in $Z=Z(P)$,
the fake weight vector $w(P) = (w_1,w_2,w_3,w_4)$
and $\mathcal{X}_i \to \KK$ as in
Construction~\ref{constr:degs}.
Then 
\[
w(\tilde P_1)
\ = \
(w_1, \, w_2, \, -d_0l_1^2),
\qquad
w(\tilde P_2)
\ = \
(w_1, \, w_2, \, -d_0l_2^2) 
\]
holds for the fake weight vectors of the central
fibers $\tilde Z_1$ and $\tilde Z_2$.
The canonical self intersection
numbers of $X, \tilde Z_1, \tilde Z_2$  satisfy
\[
\mathcal{K}_X^2
\ = \   
\mathcal{K}_{\tilde Z_1}^2
\ = \   
\mathcal{K}_{\tilde Z_2}^2.
\]
Furthermore, let $Z_i \subseteq Z$ be the affine toric
subvariety containing $x_i \in Z$ and define
$X_i := X \cap Z_i$. Then we have isomorphisms of
affine surfaces
\[
X_1
\ \cong \
\tilde Z_1 \setminus V(T_1)
\ \cong \
\tilde Z_2 \setminus V(T_1),
\qquad  
X_2
\ \cong \
\tilde Z_1 \setminus V(T_0)
\ \cong \
\tilde Z_2 \setminus V(T_0).
\]
Finally, let $Z^h \subseteq Z$ be the affine toric
subvariety given by $\cone(v_1,v_2)$ and set 
$X^h := X \cap Z^h$. Then $X^h$ is the index one
cover of $\tilde Z_i \setminus V(T_2)$, $i = 1,2$.
\end{proposition}

\begin{lemma}
\label{lem:locdegmat}
Let $(l_i,d_i)$, where $i=1,2$ be two pairs of coprime
integers, let $d_0 \in \ZZ$ such that we have affine
generator matrices
\[
B
\ = \
\left[
\begin{array}{cccc}
1 & l_1 & 0
\\
-1 & 0 & l_2
\\
d_0 & d_1 & d_2
\end{array}
\right],
\qquad
\tilde B
\ := \
\left[
\begin{array}{cc}
l_1 & -l_2 
\\ 
d_1+l_1d_0 & d_2 
\end{array}
\right].
\]
Then the associated affine toric varieties share the divisor class groups $K$
and class group order $w$ and these are explicitly given by
\begin{align*}
w & \ = \ \det(B) \ = \ \det(\tilde B) \ = \ -d_0l_1l_2-d_1l_2-d_2l_1,
\\[3pt]
K & \ =  \ \ZZ^3 / \mathrm{im}(B^*) \ = \ \ZZ^2 / \mathrm{im}(\tilde B^*) \ = \ \ZZ/w\ZZ.
\end{align*}
Moreover, let $\tilde C = [\bar c_1, \bar c_2]$ be a degree
matrix for $\tilde B$ in $K$.
Then $C = [\bar l_1 \bar c_1, \bar c_1, \bar c_2]$ is a degree
matrix for $B$ in $K$.
\end{lemma}

\begin{proof}
The statement on the determinants is obvious.
To proceed, we turn $B$ by means of unimodular row operations
into the matrix  
\[
B_0
\ = \
\left[
\begin{array}{cccc}
1 & l_1 & 0
\\
0 & l_1 & -l_2
\\
0 & d_1 + d_0l_1 & d_2
\end{array}
\right].
\]
Then
$K = \ZZ^3 / \mathrm{im}(B^*)  = \ZZ^3 / \mathrm{im}(B_0^*) = \ZZ^2 / \mathrm{im}(\tilde B^*)$.
As $\tilde C$ is a degree matrix, $K$ is generated by $\bar c_1$.
Finally, $B_0$ and $B$ share $C$ as a degree matrix in $K$.
\end{proof}

\begin{proof}[Proof of Proposition~\ref{prop:degprops}]
The first two statements hold by~\cite[Prop.~3.6]{HaKiWr}.
For the remaining statements, consider the following
submatrices
\[
P_1
\ = \
\left[
\begin{array}{rrrr}
-1 & l_1 & 0
\\
-1 & 0 & l_2
\\
0 & d_1 & d_2
\end{array}
\right],
\qquad\qquad
P_2
\ = \
\left[
\begin{array}{rrrr}
-1 & l_1 & 0
\\
-1 & 0 & l_2
\\
d_0 & d_1 & d_2
\end{array}
\right].
\]
The cones spanned by the columns of $P_i$ define
open toric subvarieties $Z_i \subseteq Z$.
Cutting down to $X$ provides us with 
affine $\KK^*$-invariant open neighbourhoods
\[
x_i
\ \in \
X_i
\ := \
X \cap Z_i
\ \cong \ 
X(P_i),
\]
where the affine $\KK^*$-surface $X(P_i)$ is constructed
in an analogous manner to the proceeding in the projective
case; see~\cite{HaWr}:
\[
X(P_i)
\ = \
\bar X_i / H_i,
\qquad
\bar X_i 
\ = \
V(T_1+T_2^{l_1}+T_3^{l_2})
\ \subseteq \ 
Z_i.
\]
Here, respective quasitori $H_{1,2} \subseteq \KK^*$
are the subgroups of $w_{2,1}$-th roots of unity and,
due to Lemma~\ref{lem:locdegmat}, the corresponding
degree matrices are given by
\[
Q_1 \ = \ [\bar l_1\bar b_1, \, \bar b_1, \,\bar c_1],
\qquad\qquad
Q_2 \ = \ [\bar l_1 \bar b_2,  \, \bar b_2, \, \bar c_2],
\]
with degree matrices $\tilde Q_1 = [\bar b_1, \bar c_1]$
for  $\tilde Z_1 \setminus V(T_1)$ and $\tilde Q_2 = [\bar b_2, \bar c_2]$
for $\tilde Z_1 \setminus V(T_0)$.
In particular, we have $H_i$-equivariant isomorphisms
\[
\varphi_i \colon \KK^2 \ \to \ \bar X_i,
\qquad
z \ \mapsto \ (-z_1^{l_1}-z_2^{l_2},z_1,z_2).
\]
Passing to the induced morphisms of the respective quotients
by $H_i$ gives the desired isomorphies of affine surfaces 
\[
\tilde Z_1 \setminus V(T_1) = \KK^2 / H_1 \cong \bar X_1 / H_1 = X_1, 
\quad
\tilde Z_1 \setminus V(T_0)  = \KK^2 / H_2 \cong \bar X_1 / H_1 = X_2.
\]
Finally, we infer from~\cite[Prop.~3.4.4.6]{ArDeHaLa}
that $x_h \in X^h$ is isomorphic to the toric fixed point
in the affine toric surface $Z_h$ by the generator matrix
\[
A
\ := \ 
\left[
\begin{array}{cc}
1 & 1
\\
0 & d_0 
\end{array}
\right].
\]
In terms of generator matrices, one directly checks that $Z_h$
naturally maps onto $\tilde Z_1 \setminus V(T_2)$ and
$\tilde Z_2 \setminus V(T_2)$
and that these maps are the index one covers: 
\[
S_i \cdot A
\ = \ 
\left[
\begin{array}{cc}
l_i & l_i
\\
d_i & d_i + d_0l_i 
\end{array}
\right],
\qquad
S_i
\ := \ 
\left[
\begin{array}{cc}
l_i & 0
\\
d_i & l_i 
\end{array}
\right].
\]
\end{proof}

\begin{proposition}
\label{prop:P1toP}
Let $\tilde Z_1$ be a fake weighted projective plane
with fake weight vector $w = (w_0,w_1,w_2)$ and
degree matrix~$\tilde Q_1$ such that
$z(2) \in \tilde Z_1$ is $T$-singular.
Then $\tilde Q_1$ admits precisely one corresponding
generator matrix $\tilde P_1$ of the form
\[
\tilde P_1  
\ = \
\left[
\begin{array}{ccc}
l_1 & l_1 & -l_2 
\\ 
d_1 & d_1+d_0l_1 & d_2 
\end{array}
\right],
\qquad
\begin{array}{l}
\scriptstyle
l_1 = \iota(z(2)), \ d_0 = -w_2/l_1^2, \ 0 \le d_1 < l_1, 
\\
\scriptstyle
l_2 = l_1(w_0+w_1)/w_2, 
\\
\scriptstyle
d_2 = -(d_1w_0 + d_1w_1 + d_0l_1w_1)/w_2,
\\
\scriptstyle
d_0l_1l_2 + d_1l_2 + l_1d_2 < 0 < d_1l_2 + l_1d_2.
\end{array}
\]
The entries of the matrix $\tilde P_1$ give rise to
generator matrices $P$ and $\tilde P_2$,
defining a $\KK^*$-surface~$X(P)$ and a
fake weighted projective plane $\tilde Z_2$,
namely

\[
P
\ = \
\left[
\begin{array}{rrrr}
-1 & -1 & l_1 & 0
\\
-1 & -1 & 0 & l_2
\\
0 & d_0 & d_1 & d_2
\end{array}
\right],
\qquad\qquad
\tilde P_2
\ = \
\left[
\begin{array}{ccc}
l_2 & l_2 & -l_1
\\ 
d_2 & d_2+d_0l_2 & d_1 
\end{array}
\right].
\]
The $\tilde Z_i$ are the central fibers of the
flat families $\mathcal{X}_i \to \KK$ from
Construction~\ref{constr:degs} applied to
$X(P)$, their fake weight vectors satisfy
\[
w(\tilde P_2) \ = \ \lambda(w(\tilde P_1)),
\]
and the point $z(2) \in \tilde Z_2$ is at most $T$-singular.
Moreover, the $\KK^*$-surface $X(P)$ is non-toric
if and only if $l_1$ and $l_1(w_0+w_1)/w_2$
both differ from one.
\end{proposition}

\begin{proof}
Let $P_1$ be any generator matrix corresponding to $\tilde Q_1$.
Lemma~\ref{lem:tsing} tells us that, after multiplying $P_1$
from the left with a suitable unimodular matrix, the
first two columns look as those of $\tilde P_1$.
Observe that the requirements $d_0 < 0$ and $0 \le d_1 < l_1$
uniquely determine $d_1$.
The third column of $P_1$ is determined by the fact that~$P_1$
annihilates the fake weight vector $w(\tilde Q_1)$.
Observe that the conditions on the entries of $\tilde P_1$
ensure the $P$ and $\tilde P_2$ are indeed projective generator
matrices. The remaining statements are clear by
Construction~\ref{constr:degs}, Proposition~\ref{prop:degprops}
and Lemma~\ref{lem:tsing}.
\end{proof}

\begin{corollary}
\label{cor:limchar}
Let $Z$ be a fake weighted projective plane with fake weight vector
$w = (w_0,w_1,w_2)$. Then the following two statements are equivalent.  
\begin{enumerate}
\item
$Z$ is the central fiber of an equivariant test configuration
of a non-toric, quasismooth, rational, projective $\KK^*$-surface
$X$ of Picard number one.
\item
At least one of the three toric fixed points of $Z$, say $z(i)$,
is $T$-singular and satisfies
\[
\qquad
\iota(z(i)) > 1, \quad w_0+w_1+w_2 >  \left(\frac{1}{\iota(z(i))}+1\right) w_i.
\]
\end{enumerate}
\end{corollary}

\begin{proof}
Combine Construction~\ref{constr:degs} and
Proposition~\ref{prop:P1toP} with~\cite[Prop.~5.4]{HaHaSu}.
\end{proof}

\begin{remark}
Consider the families $\psi_i \colon \mathcal{X}_i \to \KK$
from Construction~\ref{constr:degs}.
The fiberwise $\KK^*$-actions and the horizontal $\KK^*$-action
define an action of the two-torus~$\TT^2$ on the $\mathcal{X}_i$.
This allows us to track the three fixed points
$x_0,x_1,x_2$ in fibers $\psi_i^{-1}(1) = X(P)$
in the degeneration process. By properness of $\psi_i$, 
the horizontal $\KK^*$-orbit through~$x_k$ has a 
limit point in $\psi_k^{-1}(0) = \tilde Z_i$.
The hyperbolic fixed point $x_0$ tends to
the toric fixed point $z(2)$
and the two elliptic fixed points $x_1,x_2$ to
toric fixed points $z(0)$, $z(1)$.
\end{remark}

\begin{definition}  
\label{def:adjacent}
We call two generator matrices $P_1$, $P_2$ of fake
weighted projective planes $Z_1$, $Z_2$ \emph{adjacent}
if $P_1$, $P_2$ arise via Construction~\ref{constr:degs}
from a common matrix~$P$.
In this situation, we call any pair $Z_1'$, $Z_2'$ of 
fake weighted projective planes with $Z_1' \cong Z_1$
and $Z_2' \cong Z_2$ as well adjacent.
\end{definition}

\goodbreak

\begin{definition}
\label{def:Qmatadj}
Let $Q_1$ and $Q_2$ be degree matrices of fake weighted
projective planes $Z_1$ and $Z_2$, respectively.
We call $Q_1$, $Q_2$ \emph{adjacent} if the
following holds:
\begin{enumerate}
\item
there exist corresponding generator matrices $P_1$ for $Q_1$ and 
$P_2$ for $Q_2$ such that $P_1$, $P_2$ are adjacent,
\item
we have $w_0 \le w_1$ for the first two (common)
components of the fake weight vectors $w(Q_1)$
and $w(Q_2)$,
\item
up to permuting the columns, each of the $Q_1$, $Q_2$
is adjusted in the sense of Definition~\ref{def:adjusted}.
\end{enumerate}
\end{definition}

\begin{theorem}
\label{thm:adjpartner}
For every fake weighted projective plane $Z_1$ of integral
degree, there is a pair $(Q_1,Q_2)$ of adjacent degree
matrices with $Z_1 \cong Z(Q_1)$.
\end{theorem}

\begin{proof}
Propositions~\ref{prop:gisingdeg9865} 
to~\ref{prop:gisingdeg1} ensure that at least one of
the toric fixed points of the fake weighted projective
plane $Z_1$ is at most $T$-singular.
Thus Proposition~\ref{prop:P1toP} yields the assertion.
\end{proof}

\begin{definition}
\label{def:ordnontor}
Let $Q_1$, $Q_2$ be adjacent degree matrices.
Set $Z_i = Z(Q_i)$ and denote by $l_i$ the local
Gorenstein indices of $z(2) \in Z_i$.
We call $(Q_1,Q_2)$ \emph{ordered} if $l_1 \le l_2$
and we say that $(Q_1,Q_2)$ is \emph{non-toric} if
$l_1,l_2 \ge 2$.
\end{definition}

By the definition of adjacency, every non-toric,
ordered pair $(Q_1,Q_2)$ of adjacent degree matrices
gives rise to a $\KK^*$-surface $X$, degenerating
via Construction~\ref{constr:degs} to the fake
weighted projective planes associated with $Q_1$, $Q_2$.
Due to Proposition~\ref{prop:P1toP}, the surface~$X$
is non-toric and uniquely determined up to isomorphism
by $(Q_1,Q_2)$. We write $X = X(Q_1,Q_2)$.

\begin{theorem}
\label{thm:thm4}
Let $X$ be a non-toric, quasismooth, rational, projective
$\KK^*$-surface of Picard number one with $\mathcal{K}_X^2 \in \ZZ$. 
Then $X \cong X(Q_1,Q_2)$ with a non-toric, ordered pair
of adjacent degree matrices.
Moreover, distinct ordered pairs $(Q_1,Q_2)$ of adjacent
degree matrices yield non-isomorphic $\KK^*$-surfaces.
\end{theorem}

\begin{proof}
Let $X$ be a quasismooth, rational, projective
$\KK^*$-surface of Picard number one and integral
degree. 
Then $X \cong X(P)$ with a generator matrix~$P$ as
introduced in Section~\ref{sec:ratcstar};
see also~\cite[Constr.~3.1]{HaKiWr}.
According to Construction~\ref{constr:degs}, we have
$X(P) = X(Q_1,Q_2)$, where $Q_1,Q_2$ are adjacent
degree matrices corresponding to the generator 
matrices $P_1,P_2$ given there, and we may assume
that $(Q_1,Q_2)$ is ordered.
Moreover, Proposition~\ref{prop:P1toP} tells us that 
$(Q_1,Q_2)$ is non-toric.
Finally, the definition of a non-toric, ordered pair
of adjacent degree matrices ensures that distinct
pairs define non-isomorphic $\KK^*$-surfaces.
\end{proof}


\section{Examples and observations}

In a first example series, we check the
isomorphies for members inside one of
the series (1-9-$\ast$) or (1-8-$\ast$)
in the case of a small fake weight vector
and we discuss the first fake weight vector,
where any two different values of the
entry $\bar \eta$ in the adjusted degree
matrix lead to non-isomorphic fake weighted
projective planes. 
We begin by presenting the tacitly used tool
box.

\begin{lemma}
\label{lem:ZZmuZonto}
Let $\varphi \colon \ZZ\oplus\ZZ/\mu\ZZ \to \ZZ\oplus\ZZ/\mu\ZZ$
be a surjective group homomorphism. 
Then $\varphi$ is an isomorphism.
\end{lemma}

\begin{proof}
We first show that only torsion elements map to torsion elements.
Consider $x,y \in \ZZ$ with
$\varphi(x,\bar{y}) \in \{0\} \oplus \ZZ/\mu \ZZ$.
Then we have
\[
\varphi(\mu x,\bar{0})  
\ = \
\mu \varphi(x,\bar{y})
\ = \
(0, \bar{0}).
\]
Consequently, $\mu x \ZZ \oplus \{\bar 0\} \subseteq \ker(\varphi)$.
The homomorphism theorem yields a commutative diagram
\[
\xymatrix{
{\ZZ\oplus\ZZ/\mu\ZZ}
\ar[rr]^{\varphi}
\ar[dr]
&&
{\ZZ\oplus\ZZ/\mu\ZZ}
\\
&
{\ZZ/\mu x \ZZ \oplus \ZZ/\mu\ZZ}
\ar[ur]_{\bar{\varphi}}
&
}
\]
As $\varphi$ is surjective, also $\bar{\varphi}$ is so.
By group order reasons, $x=0$.
The claim is verified.
As a consequence, we obtain isomorphisms
\[
\xymatrix{
{ \{0\} \oplus \ZZ/\mu \ZZ }
\ar@{=}[r]
&
{ \varphi^{-1}( \{0\} \oplus \ZZ/\mu \ZZ ) }
\ar[r]^{\quad \cong}_{\quad \varphi}
&
{\{0\} \oplus \ZZ/\mu \ZZ}.
}
\]
In particular the kernel of $\varphi \colon \ZZ\oplus\ZZ/\mu\ZZ \to \ZZ\oplus\ZZ/\mu\ZZ$,
being contained in the torsion part, must be trivial.
\end{proof}

\begin{remark}
\label{rem:assocchek}
Consider a degree matrix $Q$ in $\ZZ \oplus \ZZ / \mu \ZZ$
and a $2 \times 3$ projective generator matrix $P$ sharing
the same fake weight vector.
Then Lemma~\ref{lem:ZZmuZonto} says that $Q$ and $P$
correspond to each other if and only if $Q$ annihilates
the rows of $P$.
\end{remark}

\begin{remark}
Let $v = (a,c)$ and $v' = (b,d)$ be primitive
vectors in $\ZZ^2$ generating a two-dimensional
cone in $\QQ^2$ and let $Z$ be the associated
affine toric variety.
Then, according to~\cite[Rem.~3.7]{HaHaSp}, the local
Gorenstein index of the toric fixed point $z \in Z$
is given by
\[
\iota(z)
\ =  \  
\frac{\vert ad -bc \vert}{\gcd(c-d,\, b-a)} .
\]
\end{remark}

\begin{example}
\label{ex:iso-19-111}
In the series (1-9-$\ast$), consider the degree matrices
sharing the fake weight vector $(9, 9, 9)$:
{\small
\[
\left[
\begin{array}{ccc}
1 & 1 & 1
\\[3pt]
\bar 0 & \bar 1 & \bar 2  
\end{array}
\right],
\qquad\quad
\left[
\begin{array}{ccc}
1 & 1 & 1
\\[3pt]
\bar 0 & \bar 1 & \bar 5  
\end{array}
\right],
\qquad\quad
\left[
\begin{array}{ccc}
1 & 1 & 1
\\[3pt]
\bar 0 & \bar 1 & \bar 8  
\end{array}
\right].
\]
}

\noindent
As corresponding generator matrices we obtain the following
ones, all defining isomorphic fake weighted projective planes:
{\small
\[
\left[
\begin{array}{ccc}
3 &   3 & -6
\\[3pt]
1  &  -2 & 1
\end{array}
\right],
\qquad
\left[
\begin{array}{ccc}
1 &    1 & -2
\\[3pt]
0  &  -9 &  9
\end{array}
\right],
\qquad                      
\left[
\begin{array}{ccc}
3 &   3 & -6
\\[3pt]
2  &  -1 &  -1 
\end{array}
\right].
\]
}

\noindent
For the toric fixed points $z(0)$, $z(1)$, $z(2)$, 
the local Gorenstein indices, $T$-singularity identifier
and numbers of exceptional curves of the minimal resolution
are
\[
\begin{array}{ccc}
(3, 1, 3), 
\qquad & \qquad
(3, 3,  1), 
\qquad & \qquad
(1, 3, 3),
\\[3pt]
(+, +, +), 
\qquad & \qquad
(+, +, +), 
\qquad & \qquad
(+, +, +),
\\[3pt]
(2, 8, 2), 
\qquad & \qquad
(2, 2, 8), 
\qquad & \qquad
(8, 2, 2).
\end{array}
\]
\end{example}

\begin{example}
\label{ex:iso-19-114}
In the series (1-9-$\ast$), consider the degree matrices
sharing the fake weight vector $(9, 9, 36)$:
{\small
\[
\left[
\begin{array}{ccc}
1 & 1 & 4
\\[3pt]
\bar 0 & \bar 1 & \bar 2  
\end{array}
\right],
\qquad\quad
\left[
\begin{array}{ccc}
1 & 1 & 4
\\[3pt]
\bar 0 & \bar 1 & \bar 5  
\end{array}
\right],
\qquad\quad
\left[
\begin{array}{ccc}
1 & 1 & 4
\\[3pt]
\bar 0 & \bar 1 & \bar 8  
\end{array}
\right].
\]
}

\noindent
As corresponding generator matrices we obtain the following
ones, where the last two define isomorphic fake weighted projective
planes:
{\small
\[
\left[
\begin{array}{ccc}
2 &   2 & -1
\\[3pt]
1  &  -17 &  4
\end{array}
\right],
\qquad
\left[
\begin{array}{ccc}
6 &    6 & -3
\\[3pt]
1  &  -5 &  1 
\end{array}
\right],
\qquad                      
\left[
\begin{array}{ccc}
6 &   6 & -3
\\[3pt]
5  &  -1 &  -1 
\end{array}
\right].
\]
}

\noindent
For the toric fixed points $z(0)$, $z(1)$, $z(2)$, 
the local Gorenstein indices, $T$-singularity identifier
and numbers of exceptional curves of the minimal resolution
are
\[
\begin{array}{ccc}
(3, 3, 2), 
\qquad & \qquad
(3, 1,  6), 
\qquad & \qquad
(1, 3, 6),
\\[3pt]
(+, +, +), 
\qquad & \qquad
(+, +, +), 
\qquad & \qquad
(+, +, +),
\\[3pt]
(2, 2, 9), 
\qquad & \qquad
(2, 8, 5), 
\qquad & \qquad
(8, 2, 5).
\end{array}
\]
\end{example}

\begin{example}
\label{ex:iso-19-1425}
In the series (1-9-$\ast$), consider the degree matrices
sharing the fake weight vector $(9, 36, 225)$:
{\small
\[
\left[
\begin{array}{ccc}
1 & 4 & 25
\\[3pt]
\bar 0 & \bar 1 & \bar 2  
\end{array}
\right],
\qquad\quad
\left[
\begin{array}{ccc}
1 & 4 & 25
\\[3pt]
\bar 0 & \bar 1 & \bar 5  
\end{array}
\right],
\qquad\quad
\left[
\begin{array}{ccc}
1 & 4 & 25
\\[3pt]
\bar 0 & \bar 1 & \bar 8  
\end{array}
\right].
\]
}

\noindent
As corresponding generator matrices we obtain the following
ones, defining pairwise non-isomorphic fake weighted projective
planes:
{\small
\[
\left[
\begin{array}{ccc}
15 &   15 & -3
\\[3pt]
2  &  -13 &  2 
\end{array}
\right],
\qquad
\left[
\begin{array}{ccc}
 5 &    5 & -1
\\[3pt]
1  &  -44 &  7 
\end{array}
\right],
\qquad                      
\left[
\begin{array}{ccc}
15 &   15 & -3
\\[3pt]
7  &  - 8 &  1 
\end{array}
\right].
\]
}

\noindent
For the toric fixed points $z(0)$, $z(1)$, $z(2)$, 
the local Gorenstein indices, $T$-singularity identifier
and numbers of exceptional curves of the minimal resolution
are
\[
\begin{array}{ccc}
(3, 2, 15), 
\qquad & \qquad
(3, 6,  5), 
\qquad & \qquad
(1, 6, 15),
\\[3pt]
(+, +, +), 
\qquad & \qquad
(+, +, +), 
\qquad & \qquad
(+, +, +),
\\[3pt]
(2, 9, 8), 
\qquad & \qquad
(2, 5, 12), 
\qquad & \qquad
(8, 5, 8).
\end{array}
\]
\end{example}

\begin{example}
\label{ex:iso-18-112}
In the series (1-8-$\ast$), consider the degree matrices
$Q_1, \ldots, Q_4$, sharing the fake weight vector $(8, 8, 16)$:
{\small
\[
\left[
\begin{array}{ccc}
1 & 1 & 2
\\[3pt]
\bar 0 & \bar 1 & \bar 1  
\end{array}
\right],
\qquad\quad
\left[
\begin{array}{ccc}
1 & 1 & 2
\\[3pt]
\bar 0 & \bar 1 & \bar 3  
\end{array}
\right],
\qquad\quad
\left[
\begin{array}{ccc}
1 & 1 & 2
\\[3pt]
\bar 0 & \bar 1 & \bar 5  
\end{array}
\right],
\qquad\quad
\left[
\begin{array}{ccc}
1 & 1 & 2
\\[3pt]
\bar 0 & \bar 1 & \bar 7 
\end{array}
\right].
\]
}

\noindent
As corresponding generator matrices $P_1, \ldots, P_4$,
we obtain the following, showing in particular
$Z(P_2) \cong Z(P_4)$:
{\small
\[
\left[
\begin{array}{ccc}
1 &   1 & -1
\\[3pt]
0  &  -16 & 8
\end{array}
\right],
\quad
\left[
\begin{array}{ccc}
4 &    4 & -4
\\[3pt]
1  &  -3 &  1
\end{array}
\right],
\quad                      
\left[
\begin{array}{ccc}
2 &   2 & -2
\\[3pt]
1  &  -7 & 3
\end{array}
\right],
\quad
\left[
\begin{array}{ccc}
4 &   4 & -4
\\[3pt]
3  &  -1 &  -1 
\end{array}
\right].
\]
}

\noindent
For the toric fixed points $z(0)$, $z(1)$, $z(2)$, 
the local Gorenstein indices, $T$-singularity identifier
and numbers of exceptional curves of the minimal resolution
are
\[
\begin{array}{cccc}
(4, 4, 1), 
\qquad & \qquad
(2, 1,  4), 
\qquad & \qquad
(4, 4,  2), 
\qquad & \qquad
(1, 2, 4),
\\[3pt]
(-, -, +), 
\qquad & \qquad
(+, +, +), 
\qquad & \qquad
(-, -,  +), 
\qquad & \qquad
(+, +, +),
\\[3pt]
(1, 1, 15), 
\qquad & \qquad
(2, 7, 3), 
\qquad & \qquad
(3, 3, 4),
\qquad & \qquad
(7, 2, 3).
\end{array}
\]
\end{example}

\begin{example}
\label{ex:iso-18-192}
In the series (1-8-$\ast$), consider the degree matrices
$Q_1, \ldots, Q_4$, sharing the fake weight vector $(8, 16, 72)$:
{\small
\[
\left[
\begin{array}{ccc}
1 & 9 & 2
\\[3pt]
\bar 0 & \bar 1 & \bar 1  
\end{array}
\right],
\qquad\quad
\left[
\begin{array}{ccc}
1 & 9 & 2
\\[3pt]
\bar 0 & \bar 1 & \bar 3  
\end{array}
\right],
\qquad\quad
\left[
\begin{array}{ccc}
1 & 9 & 2
\\[3pt]
\bar 0 & \bar 1 & \bar 5  
\end{array}
\right],
\qquad\quad
\left[
\begin{array}{ccc}
1 & 9 & 2
\\[3pt]
\bar 0 & \bar 1 & \bar 7 
\end{array}
\right].
\]
}

\noindent
As corresponding generator matrices $P_1, \ldots, P_4$,
we obtain the following, no two of them defining
isomorphic fake weighted projective planes:
{\small
\[
\left[
\begin{array}{ccc}
2 &   2 & -10
\\[3pt]
1  &  -7 & 31
\end{array}
\right],
\quad
\left[
\begin{array}{ccc}
4 &    4 & -20
\\[3pt]
3  &  -1 &  3
\end{array}
\right],
\quad                      
\left[
\begin{array}{ccc}
1 &   1 & -5
\\[3pt]
0  &  -16 & 72
\end{array}
\right],
\quad
\left[
\begin{array}{ccc}
4 &   4 & -20
\\[3pt]
1  &  -3 &  13 
\end{array}
\right].
\]
}

\goodbreak

\noindent
For the toric fixed points $z(0)$, $z(1)$, $z(2)$, 
the local Gorenstein indices, $T$-singularity identifier
and numbers of exceptional curves of the minimal resolution
are
\[
\begin{array}{cccc}
(4, 12, 2), 
\qquad & \qquad
(2, 3, 4), 
\qquad & \qquad
(4, 12, 1), 
\qquad & \qquad
(1, 6, 4),
\\[3pt]
(-, -, +), 
\qquad & \qquad
(+, +, +), 
\qquad & \qquad
(-, -,  +), 
\qquad & \qquad
(+, +, +),
\\[3pt]
(1, 11, 4), 
\qquad & \qquad
(2, 9, 3), 
\qquad & \qquad
(3, 3, 15),
\qquad & \qquad
(7, 6, 3).
\end{array}
\]
\end{example}

We take a closer look at the adjacency relation
on the class of fake weighted projective planes,
as introduced in Definition~\ref{def:adjacent}.
The aim is to indicate, how adjacency can be read
off from adjusted degree matrices.
In the case of at most $T$-singular fake weighted
projective planes, the situation is completely
described by the trees $T(a)$ of the Markov
type equations.

\begin{construction}
\label{constr:adj-graphs}
Let $T(a,\mu)$ denote the set of isomorphy
classes of all at most $T$-singular fake
weighted projective planes of degree
$a \in \ZZ_{>0}$ and multiplicity $\mu$.
We join two distinct classes
$\mathcal{Z}_1, \mathcal{Z}_2 \in T(a,\mu)$
by an edge, if there are representatives
$Z_i \in \mathcal{Z}_i$ such that $Z_1$ and $Z_2$
are adjacent.
\end{construction}

\begin{proposition}
For $a=9,8,6,5,4,3$, each of the graphs $T(a,\mu)$ is connected.
Moreover, we have canonical isomophisms of graphs:
\[
T(9) \cong T(9,1) \cong T(3,3),
\qquad  
T(8) \cong T(8,1) \cong T(4,2),
\]
\[
T(6) \cong T(6,1) \cong T(3,2) \cong T(2,3) \cong T(1,6),
\qquad
T(5) \cong T(5,1) \cong T(1,5),
\]
where $T(a)$ denotes the tree of ascendingly ordered
solution triples of the squared Markov equation from
Remark~\ref{rem:markgraf}.
\end{proposition}

We discuss the remaining cases.
For the subsequent examples, we explicitly
computed the data of Proposition~\ref{prop:P1toP}
for the first adjusted degree matrices in the
series provided by our classification results.
We omit the computation and just present the
resulting graphs, where the vertices, that means
the isomorphy classes of fake weighted projective planes,
are specified by $(u_0,u_1,u_2; \bar \eta)$ with
the first row $u = (u_0,u_1,u_2)$ of the corresponding
adjusted degree matrix $Q$ and the last entry~$\bar \eta$
of the second row of~$Q$.

\begin{remark}
The graph $T(2,4)$ has the series (2-4-1) and (2-4-3) as
its connected components, where each of them is 
isomorphic to $S(8)$ as a graph.
\end{remark}

\begin{remark}
The series (2-3-1) is not connected to $T(2,3)$ by adjacency.
Inside (2-3-1), we observe a simple pattern for the adjacencies:

\bigskip

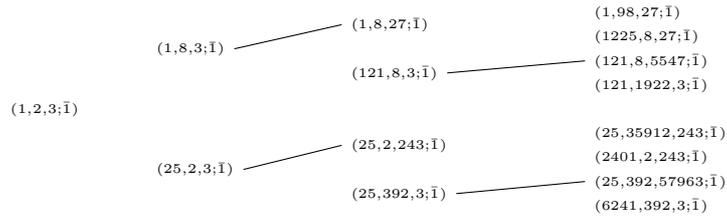
\begin{center}
  
\begin{tikzpicture}[scale=0.64]
\sffamily


\node[right] (1-2-3-b1) at (0,0) {$\scriptscriptstyle (1,2,3; \bar 1)$};


\node[right] (1-8-3-b1) at (3,1.25) {$\scriptscriptstyle (1,8,3; \bar 1)$};
\node[right] (25-2-3-b1) at (3,-1.25) {$\scriptscriptstyle (25,2,3; \bar 1)$};


\node[right] (1-8-27-b1) at (7,1.75) {$\scriptscriptstyle (1,8,27; \bar 1)$};
\node[right] (121-8-3-b1) at (7,0.75) {$\scriptscriptstyle (121,8,3; \bar 1)$};

\node[right] (25-2-243-b1) at (7,-0.75) {$\scriptscriptstyle (25,2,243; \bar 1)$};
\node[right] (25-392-3-b1) at (7,-1.75) {$\scriptscriptstyle (25,392,3; \bar 1)$};


\node[right] (1-98-27-b1) at (12,2) {$\scriptscriptstyle (1,98,27; \bar 1)$};
\node[right] (1225-8-27-b1) at (12,1.5) {$\scriptscriptstyle (1225,8,27; \bar 1)$};

\node[right] (121-8-5547-b1) at (12,1) {$\scriptscriptstyle (121,8,5547; \bar 1)$};
\node[right] (121-1922-3-b1) at (12,0.5) {$\scriptscriptstyle (121,1922,3; \bar 1)$};

\node[right] (25-35912-243-b1) at (12,-0.5) {$\scriptscriptstyle (25,35912,243; \bar 1)$};
\node[right] (196-4489-5-b1) at (12,-1) {$\scriptscriptstyle (2401,2,243; \bar 1)$};

\node[right] (25-392-57963-b1) at (12,-1.5) {$\scriptscriptstyle (25,392,57963; \bar 1)$};
\node[right] (6241-392-3-b1) at (12,-2) {$\scriptscriptstyle (6241,392,3; \bar 1)$};d


\draw[] (1-8-3-b1.east) edge (1-8-27-b1.west);

\draw[] (25-2-3-b1.east) edge (25-2-243-b1.west);

\draw[] (121-8-3-b1.east) edge (121-8-5547-b1.west);

\draw[] (25-392-3-b1.east) edge (25-392-57963-b1.west);

\end{tikzpicture}
\end{center}

\end{remark}

\goodbreak

\begin{remark}
We take a look at the connected graph $T(1,9)$,
where the vertices stem from the three
series (1-9-2), (1-9-5) and (1-9-8):

\bigskip

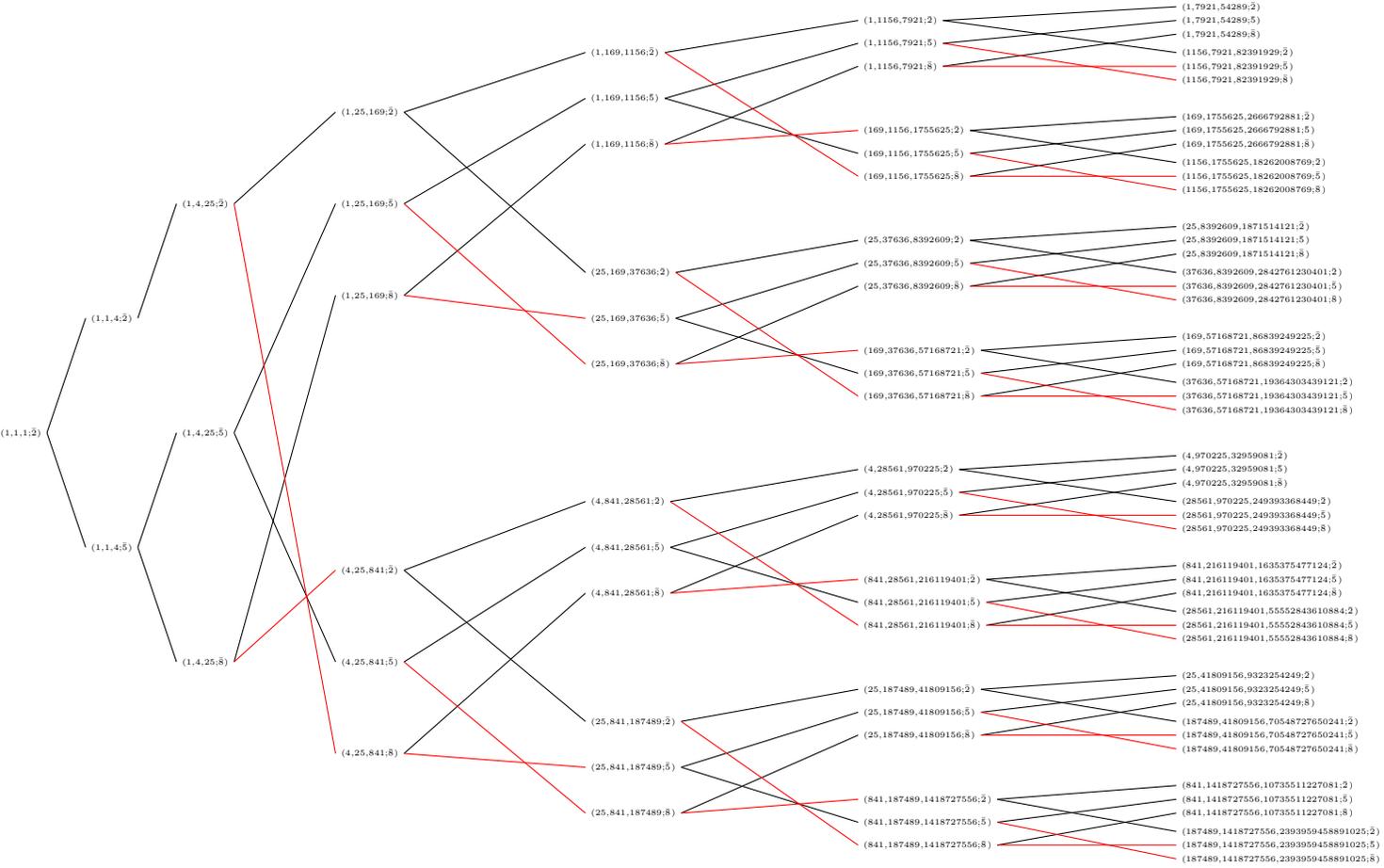
\begin{center}  
\begin{tikzpicture}[scale=0.64,rotate=270,transform shape]
\sffamily


\node[right] (1-1-1) at (0,0) {$\scriptscriptstyle (1,1,1; \bar 2)$};


\node[right] (1-1-4-b2) at (2,2.5) {$\scriptscriptstyle (1,1,4; \bar 2)$};
\node[right] (1-1-4-b5) at (2,-2.5) {$\scriptscriptstyle (1,1,4; \bar 5)$};


\node[right] (1-4-25-b2) at (4,5) {$\scriptscriptstyle (1,4,25; \bar 2)$};
\node[right] (1-4-25-b5) at (4,0) {$\scriptscriptstyle (1,4,25; \bar 5)$};
\node[right] (1-4-25-b8) at (4,-5) {$\scriptscriptstyle (1,4,25; \bar 8)$};


\node[right] (1-25-169-b2) at (7.5,7) {$\scriptscriptstyle (1,25,169; \bar 2)$};
\node[right] (1-25-169-b5) at (7.5,5) {$\scriptscriptstyle (1,25,169; \bar 5)$};
\node[right] (1-25-169-b8) at (7.5,3) {$\scriptscriptstyle (1,25,169; \bar 8)$};

\node[right] (4-25-841-b2) at (7.5,-3) {$\scriptscriptstyle (4,25,841; \bar 2)$};
\node[right] (4-25-841-b5) at (7.5,-5) {$\scriptscriptstyle (4,25,841; \bar 5)$};
\node[right] (4-25-841-b8) at (7.5,-7) {$\scriptscriptstyle (4,25,841; \bar 8)$};


\node[right] (1-169-1156-b2) at (13,8.3) {$\scriptscriptstyle (1,169,1156; \bar 2)$};
\node[right] (1-169-1156-b5) at (13,7.3) {$\scriptscriptstyle (1,169,1156; \bar 5)$};
\node[right] (1-169-1156-b8) at (13,6.3) {$\scriptscriptstyle (1,169,1156; \bar 8)$};

\node[right] (25-169-37636-b2) at (13,3.5) {$\scriptscriptstyle (25,169,37636; \bar 2)$};
\node[right] (25-169-37636-b5) at (13,2.5) {$\scriptscriptstyle (25,169,37636; \bar 5)$};
\node[right] (25-169-37636-b8) at (13,1.5) {$\scriptscriptstyle (25,169,37636; \bar 8)$};

\node[right] (4-841-28561-b2) at (13,-1.5) {$\scriptscriptstyle (4,841,28561; \bar 2)$};
\node[right] (4-841-28561-b5) at (13,-2.5) {$\scriptscriptstyle (4,841,28561; \bar 5)$};
\node[right] (4-841-28561-b8) at (13,-3.5) {$\scriptscriptstyle (4,841,28561; \bar 8)$};

\node[right] (25-841-187489-b2) at (13,-6.3) {$\scriptscriptstyle (25,841,187489; \bar 2)$};
\node[right] (25-841-187489-b5) at (13,-7.3) {$\scriptscriptstyle (25,841,187489; \bar 5)$};
\node[right] (25-841-187489-b8) at (13,-8.3) {$\scriptscriptstyle (25,841,187489; \bar 8)$};


\node[right] (1-1156-7921-b2) at (19,9) {$\scriptscriptstyle (1,1156,7921; \bar 2)$};
\node[right] (1-1156-7921-b5) at (19,8.5) {$\scriptscriptstyle (1,1156,7921; \bar 5)$};
\node[right] (1-1156-7921-b8) at (19,8) {$\scriptscriptstyle (1,1156,7921; \bar 8)$};
\node[right] (169-1156-1755625-b2) at (19,6.6) {$\scriptscriptstyle (169,1156,1755625; \bar 2)$};
\node[right] (169-1156-1755625-b5) at (19,6.1) {$\scriptscriptstyle (169,1156,1755625; \bar 5)$};
\node[right] (169-1156-1755625-b8) at (19,5.6) {$\scriptscriptstyle (169,1156,1755625; \bar 8)$};

\node[right] (25-37636-8392609-b2) at (19,4.2) {$\scriptscriptstyle (25,37636,8392609; \bar 2)$};
\node[right] (25-37636-8392609-b5) at (19,3.7) {$\scriptscriptstyle (25,37636,8392609; \bar 5)$};
\node[right] (25-37636-8392609-b8) at (19,3.2) {$\scriptscriptstyle (25,37636,8392609; \bar 8)$};
\node[right] (169-37636-57168721-b2) at (19,1.8) {$\scriptscriptstyle (169,37636,57168721; \bar 2)$};
\node[right] (169-37636-57168721-b5) at (19,1.3) {$\scriptscriptstyle (169,37636,57168721; \bar 5)$};
\node[right] (169-37636-57168721-b8) at (19,0.8) {$\scriptscriptstyle (169,37636,57168721; \bar 8)$};

\node[right] (4-28561-970225-b2) at (19,-0.8) {$\scriptscriptstyle (4,28561,970225; \bar 2)$};
\node[right] (4-28561-970225-b5) at (19,-1.3) {$\scriptscriptstyle (4,28561,970225; \bar 5)$};
\node[right] (4-28561-970225-b8) at (19,-1.8) {$\scriptscriptstyle (4,28561,970225; \bar 8)$};

\node[right] (841-28561-216119401-b2) at (19,-3.2) {$\scriptscriptstyle (841,28561,216119401; \bar 2)$};
\node[right] (841-28561-216119401-b5) at (19,-3.7) {$\scriptscriptstyle (841,28561,216119401; \bar 5)$};
\node[right] (841-28561-216119401-b8) at (19,-4.2) {$\scriptscriptstyle (841,28561,216119401; \bar 8)$};

\node[right] (25-187489-41809156-b2) at (19,-5.6) {$\scriptscriptstyle (25,187489,41809156; \bar 2)$};
\node[right] (25-187489-41809156-b5) at (19,-6.1) {$\scriptscriptstyle (25,187489,41809156; \bar 5)$};
\node[right] (25-187489-41809156-b8) at (19,-6.6) {$\scriptscriptstyle (25,187489,41809156; \bar 8)$};

\node[right] (841-187489-1418727556-b2) at (19,-8) {$\scriptscriptstyle (841,187489,1418727556; \bar 2)$};
\node[right] (841-187489-1418727556-b5) at (19,-8.5) {$\scriptscriptstyle (841,187489,1418727556; \bar 5)$};
\node[right] (841-187489-1418727556-b8) at (19,-9) {$\scriptscriptstyle (841,187489,1418727556; \bar 8)$};


\node[right] (1-7921-54289-b2) at (26,9.3) {$\scriptscriptstyle (1,7921,54289; \bar 2)$};
\node[right] (1-7921-54289-b5) at (26,9) {$\scriptscriptstyle (1,7921,54289; \bar 5)$};
\node[right] (1-7921-54289-b8) at (26,8.7) {$\scriptscriptstyle (1,7921,54289; \bar 8)$};
\node[right] (1156-7921-82391929-b2) at (26,8.3) {$\scriptscriptstyle (1156,7921,82391929; \bar 2)$};
\node[right] (1156-7921-82391929-b5) at (26,8) {$\scriptscriptstyle (1156,7921,82391929; \bar 5)$};
\node[right] (1156-7921-82391929-b8) at (26,7.7) {$\scriptscriptstyle (1156,7921,82391929; \bar 8)$};

\node[right] (169-1755625-2666792881-b2) at (26,6.9) {$\scriptscriptstyle (169,1755625,2666792881; \bar 2)$};
\node[right] (169-1755625-2666792881-b5) at (26,6.6) {$\scriptscriptstyle (169,1755625,2666792881; \bar 5)$};
\node[right] (169-1755625-2666792881-b8) at (26,6.3) {$\scriptscriptstyle (169,1755625,2666792881; \bar 8)$};
\node[right] (1156-1755625-18262008769-b2) at (26,5.9) {$\scriptscriptstyle (1156,1755625,18262008769; \bar 2)$};
\node[right] (1156-1755625-18262008769-b5) at (26,5.6) {$\scriptscriptstyle (1156,1755625,18262008769; \bar 5)$};
\node[right] (1156-1755625-18262008769-b8) at (26,5.3) {$\scriptscriptstyle (1156,1755625,18262008769; \bar 8)$};

\node[right] (25-8392609-1871514121-b2) at (26,4.5) {$\scriptscriptstyle (25,8392609,1871514121; \bar 2)$};
\node[right] (25-8392609-1871514121-b5) at (26,4.2) {$\scriptscriptstyle (25,8392609,1871514121; \bar 5)$};
\node[right] (25-8392609-1871514121-b8) at (26,3.9) {$\scriptscriptstyle (25,8392609,1871514121; \bar 8)$};
\node[right] (37636-8392609-2842761230401-b2) at (26,3.5) {$\scriptscriptstyle (37636,8392609,2842761230401; \bar 2)$};
\node[right] (37636-8392609-2842761230401-b5) at (26,3.2) {$\scriptscriptstyle (37636,8392609,2842761230401; \bar 5)$};
\node[right] (37636-8392609-2842761230401-b8) at (26,2.9) {$\scriptscriptstyle (37636,8392609,2842761230401; \bar 8)$};

\node[right] (169-57168721-86839249225-b2) at (26,2.1) {$\scriptscriptstyle (169,57168721,86839249225; \bar 2)$};
\node[right] (169-57168721-86839249225-b5) at (26,1.8) {$\scriptscriptstyle (169,57168721,86839249225; \bar 5)$};
\node[right] (169-57168721-86839249225-b8) at (26,1.5) {$\scriptscriptstyle (169,57168721,86839249225; \bar 8)$};
\node[right] (37636-57168721-19364303439121-b2) at (26,1.1) {$\scriptscriptstyle (37636,57168721,19364303439121; \bar 2)$};
\node[right] (37636-57168721-19364303439121-b5) at (26,0.8) {$\scriptscriptstyle (37636,57168721,19364303439121; \bar 5)$};
\node[right] (37636-57168721-19364303439121-b8) at (26,0.5) {$\scriptscriptstyle (37636,57168721,19364303439121; \bar 8)$};

\node[right] (4-970225-32959081-b2) at (26,-0.5) {$\scriptscriptstyle (4,970225,32959081; \bar 2)$};
\node[right] (4-970225-32959081-b5) at (26,-0.8) {$\scriptscriptstyle (4,970225,32959081; \bar 5)$};
\node[right] (4-970225-32959081-b8) at (26,-1.1) {$\scriptscriptstyle (4,970225,32959081; \bar 8)$};
\node[right] (28561-970225-249393368449-b2) at (26,-1.5) {$\scriptscriptstyle (28561,970225,249393368449; \bar 2)$};
\node[right] (28561-970225-249393368449-b5) at (26,-1.8) {$\scriptscriptstyle (28561,970225,249393368449; \bar 5)$};
\node[right] (28561-970225-249393368449-b8) at (26,-2.1) {$\scriptscriptstyle (28561,970225,249393368449; \bar 8)$};

\node[right] (841-216119401-1635375477124-b2) at (26,-2.9) {$\scriptscriptstyle (841,216119401,1635375477124; \bar 2)$};
\node[right] (841-216119401-1635375477124-b5) at (26,-3.2) {$\scriptscriptstyle (841,216119401,1635375477124; \bar 5)$};
\node[right] (841-216119401-1635375477124-b8) at (26,-3.5) {$\scriptscriptstyle (841,216119401,1635375477124; \bar 8)$};
\node[right] (28561-216119401-55552843610884-b2) at (26,-3.9) {$\scriptscriptstyle (28561,216119401,55552843610884; \bar 2)$};
\node[right] (28561-216119401-55552843610884-b5) at (26,-4.2) {$\scriptscriptstyle (28561,216119401,55552843610884; \bar 5)$};
\node[right] (28561-216119401-55552843610884-b8) at (26,-4.5) {$\scriptscriptstyle (28561,216119401,55552843610884; \bar 8)$};

\node[right] (25-41809156-9323254249-b2) at (26,-5.3) {$\scriptscriptstyle (25,41809156,9323254249; \bar 2)$};
\node[right] (25-41809156-9323254249-b5) at (26,-5.6) {$\scriptscriptstyle (25,41809156,9323254249; \bar 5)$};
\node[right] (25-41809156-9323254249-b8) at (26,-5.9) {$\scriptscriptstyle (25,41809156,9323254249; \bar 8)$};
\node[right] (187489-41809156-70548727650241-b2) at (26,-6.3) {$\scriptscriptstyle (187489,41809156,70548727650241; \bar 2)$};
\node[right] (187489-41809156-70548727650241-b5) at (26,-6.6) {$\scriptscriptstyle (187489,41809156,70548727650241; \bar 5)$};
\node[right] (187489-41809156-70548727650241-b8) at (26,-6.9) {$\scriptscriptstyle (187489,41809156,70548727650241; \bar 8)$};

\node[right] (841-1418727556-10735511227081-b2) at (26,-7.7) {$\scriptscriptstyle (841,1418727556,10735511227081; \bar 2)$};
\node[right] (841-1418727556-10735511227081-b5) at (26,-8) {$\scriptscriptstyle (841,1418727556,10735511227081; \bar 5)$};
\node[right] (841-1418727556-10735511227081-b8) at (26,-8.3) {$\scriptscriptstyle (841,1418727556,10735511227081; \bar 8)$};
\node[right] (187489-1418727556-2393959458891025-b2) at (26,-8.7) {$\scriptscriptstyle (187489,1418727556,2393959458891025; \bar 2)$};
\node[right] (187489-1418727556-2393959458891025-b5) at (26,-9) {$\scriptscriptstyle (187489,1418727556,2393959458891025; \bar 5)$};
\node[right] (187489-1418727556-2393959458891025-b8) at (26,-9.3) {$\scriptscriptstyle (187489,1418727556,2393959458891025; \bar 8)$};


\draw[] (1-1-1.east) edge (1-1-4-b2.west);
\draw[] (1-1-1.east) edge (1-1-4-b5.west);

\draw[] (1-1-4-b2.east) edge (1-4-25-b2.west);
\draw[] (1-1-4-b5.east) edge (1-4-25-b5.west);
\draw[] (1-1-4-b5.east) edge (1-4-25-b8.west);

\draw[] (1-4-25-b2.east) edge (1-25-169-b2.west);
\draw[] (1-4-25-b5.east) edge (1-25-169-b5.west);
\draw[] (1-4-25-b5.east) edge (4-25-841-b5.west);
\draw[] (1-4-25-b8.east) edge (1-25-169-b8.west);
\draw[color=red!100] (1-4-25-b2.east) edge (4-25-841-b8.west);
\draw[color=red!100] (1-4-25-b8.east) edge (4-25-841-b2.west);

\draw[] (1-25-169-b2.east) edge (1-169-1156-b2.west);
\draw[] (1-25-169-b5.east) edge (1-169-1156-b5.west);
\draw[] (1-25-169-b8.east) edge (1-169-1156-b8.west);
\draw[] (1-25-169-b2.east) edge (25-169-37636-b2.west);
\draw[color=red!100] (1-25-169-b5.east) edge (25-169-37636-b8.west);
\draw[color=red!100] (1-25-169-b8.east) edge (25-169-37636-b5.west);

\draw[] (4-25-841-b2.east) edge (25-841-187489-b2.west);
\draw[] (4-25-841-b2.east) edge (4-841-28561-b2.west);
\draw[] (4-25-841-b5.east) edge (4-841-28561-b5.west);
\draw[] (4-25-841-b8.east) edge (4-841-28561-b8.west);
\draw[color=red!100] (4-25-841-b5.east) edge (25-841-187489-b8.west);
\draw[color=red!100] (4-25-841-b8.east) edge (25-841-187489-b5.west);

\draw[] (1-169-1156-b2.east) edge (1-1156-7921-b2.west);
\draw[] (1-169-1156-b5.east) edge (1-1156-7921-b5.west);
\draw[] (1-169-1156-b5.east) edge (169-1156-1755625-b5.west);
\draw[] (1-169-1156-b8.east) edge (1-1156-7921-b8.west);
\draw[color=red!100] (1-169-1156-b2.east) edge (169-1156-1755625-b8.west);
\draw[color=red!100] (1-169-1156-b8.east) edge (169-1156-1755625-b2.west);

\draw[] (25-169-37636-b2.east) edge (25-37636-8392609-b2.west);
\draw[] (25-169-37636-b5.east) edge (25-37636-8392609-b5.west);
\draw[] (25-169-37636-b5.east) edge (169-37636-57168721-b5.west);
\draw[] (25-169-37636-b8.east) edge (25-37636-8392609-b8.west);
\draw[color=red!100] (25-169-37636-b8.east) edge (169-37636-57168721-b2.west);
\draw[color=red!100] (25-169-37636-b2.east) edge (169-37636-57168721-b8.west);

\draw[] (4-841-28561-b2.east) edge (4-28561-970225-b2.west);
\draw[] (4-841-28561-b5.east) edge (841-28561-216119401-b5.west);
\draw[] (4-841-28561-b5.east) edge (4-28561-970225-b5.west);
\draw[] (4-841-28561-b8.east) edge (4-28561-970225-b8.west);
\draw[color=red!100] (4-841-28561-b2.east) edge (841-28561-216119401-b8.west);
\draw[color=red!100] (4-841-28561-b8.east) edge (841-28561-216119401-b2.west);

\draw[] (25-841-187489-b2.east) edge (25-187489-41809156-b2.west);
\draw[] (25-841-187489-b5.east) edge (25-187489-41809156-b5.west);
\draw[] (25-841-187489-b5.east) edge (841-187489-1418727556-b5.west);
\draw[] (25-841-187489-b8.east) edge (25-187489-41809156-b8.west);
\draw[color=red!100] (25-841-187489-b2.east) edge (841-187489-1418727556-b8.west);
\draw[color=red!100] (25-841-187489-b8.east) edge (841-187489-1418727556-b2.west);

\draw[] (1-1156-7921-b2.east) edge (1-7921-54289-b2.west);
\draw[] (1-1156-7921-b5.east) edge (1-7921-54289-b5.west);
\draw[] (1-1156-7921-b8.east) edge (1-7921-54289-b8.west);
\draw[] (1-1156-7921-b2.east) edge (1156-7921-82391929-b2.west);
\draw[color=red!100] (1-1156-7921-b5.east) edge (1156-7921-82391929-b8.west);
\draw[color=red!100] (1-1156-7921-b8.east) edge (1156-7921-82391929-b5.west);

\draw[] (169-1156-1755625-b2.east) edge (169-1755625-2666792881-b2.west);
\draw[] (169-1156-1755625-b5.east) edge (169-1755625-2666792881-b5.west);
\draw[] (169-1156-1755625-b8.east) edge (169-1755625-2666792881-b8.west);
\draw[] (169-1156-1755625-b2.east) edge (1156-1755625-18262008769-b2.west);
\draw[color=red!100] (169-1156-1755625-b5.east) edge (1156-1755625-18262008769-b8.west);
\draw[color=red!100] (169-1156-1755625-b8.east) edge (1156-1755625-18262008769-b5.west);

\draw[] (25-37636-8392609-b2.east) edge (25-8392609-1871514121-b2.west);
\draw[] (25-37636-8392609-b5.east) edge (25-8392609-1871514121-b5.west);
\draw[] (25-37636-8392609-b8.east) edge (25-8392609-1871514121-b8.west);
\draw[] (25-37636-8392609-b2.east) edge(37636-8392609-2842761230401-b2.west);
\draw[color=red!100] (25-37636-8392609-b5.east) edge(37636-8392609-2842761230401-b8.west);
\draw[color=red!100] (25-37636-8392609-b8.east) edge(37636-8392609-2842761230401-b5.west);

\draw[] (169-37636-57168721-b2.east) edge (169-57168721-86839249225-b2.west);
\draw[] (169-37636-57168721-b5.east) edge (169-57168721-86839249225-b5.west);
\draw[] (169-37636-57168721-b8.east) edge (169-57168721-86839249225-b8.west);
\draw[] (169-37636-57168721-b2.east) edge (37636-57168721-19364303439121-b2.west);
\draw[color=red!100] (169-37636-57168721-b5.east) edge (37636-57168721-19364303439121-b8.west);
\draw[color=red!100] (169-37636-57168721-b8.east) edge (37636-57168721-19364303439121-b5.west);

\draw[] (4-28561-970225-b2.east) edge (4-970225-32959081-b2.west);
\draw[] (4-28561-970225-b5.east) edge (4-970225-32959081-b5.west);
\draw[] (4-28561-970225-b8.east) edge (4-970225-32959081-b8.west);
\draw[] (4-28561-970225-b2.east) edge (28561-970225-249393368449-b2.west);
\draw[color=red!100] (4-28561-970225-b5.east) edge (28561-970225-249393368449-b8.west);
\draw[color=red!100] (4-28561-970225-b8.east) edge (28561-970225-249393368449-b5.west);

\draw[] (841-28561-216119401-b2.east) edge (841-216119401-1635375477124-b2.west);
\draw[] (841-28561-216119401-b5.east) edge (841-216119401-1635375477124-b5.west);
\draw[] (841-28561-216119401-b8.east) edge (841-216119401-1635375477124-b8.west);
\draw[] (841-28561-216119401-b2.east) edge (28561-216119401-55552843610884-b2.west);
\draw[color=red!100] (841-28561-216119401-b5.east) edge (28561-216119401-55552843610884-b8.west);
\draw[color=red!100] (841-28561-216119401-b8.east) edge (28561-216119401-55552843610884-b5.west);

\draw[] (25-187489-41809156-b2.east) edge (25-41809156-9323254249-b2.west);
\draw[] (25-187489-41809156-b5.east) edge (25-41809156-9323254249-b5.west);
\draw[] (25-187489-41809156-b8.east) edge (25-41809156-9323254249-b8.west);
\draw[] (25-187489-41809156-b2.east) edge (187489-41809156-70548727650241-b2.west);
\draw[color=red!100] (25-187489-41809156-b5.east) edge (187489-41809156-70548727650241-b8.west);
\draw[color=red!100] (25-187489-41809156-b8.east) edge (187489-41809156-70548727650241-b5.west);

\draw[] (841-187489-1418727556-b2.east) edge (841-1418727556-10735511227081-b2.west);
\draw[] (841-187489-1418727556-b5.east) edge (841-1418727556-10735511227081-b5.west);
\draw[] (841-187489-1418727556-b8.east) edge (841-1418727556-10735511227081-b8.west);
\draw[] (841-187489-1418727556-b2.east) edge (187489-1418727556-2393959458891025-b2.west);
\draw[color=red!100] (841-187489-1418727556-b5.east) edge (187489-1418727556-2393959458891025-b8.west);
\draw[color=red!100] (841-187489-1418727556-b8.east) edge (187489-1418727556-2393959458891025-b5.west);

\end{tikzpicture}
\end{center}

\bigskip

\noindent
Here, the red edges indicate the jumps between the series
(1-9-2), (1-9-5), (1-9-8) for the first seven levels.
We expect the pattern to continue.
\end{remark}

\goodbreak

\begin{remark}
The connected graph $T(1,8)$ takes its vertices from
the two series (1-8-3) and (1-8-7) and we have the
following adjacencies:

\bigskip

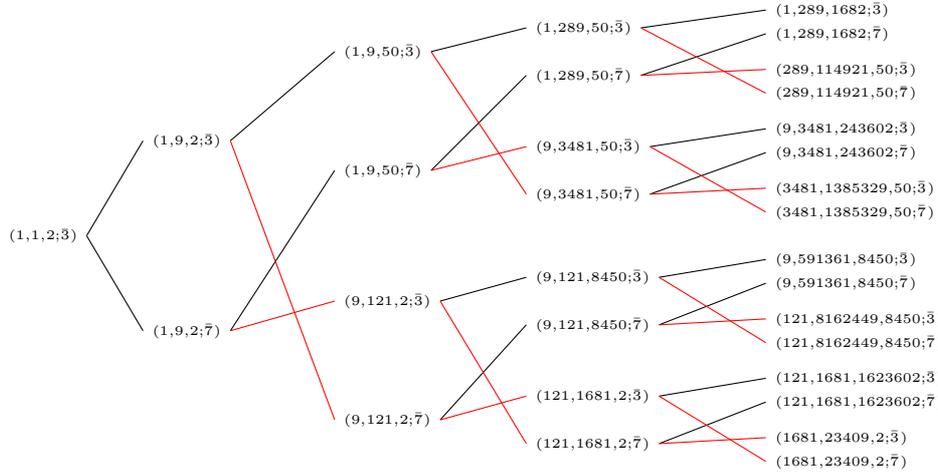
\begin{center}
  
\begin{tikzpicture}[scale=0.63]
\sffamily


\node[right] (1-1-2-b3) at (-1,0) {$\scriptscriptstyle (1,1,2; \bar 3)$};


\node[right] (1-9-2-b3) at (2,2) {$\scriptscriptstyle (1,9,2; \bar 3)$};
\node[right] (1-9-2-b7) at (2,-2) {$\scriptscriptstyle (1,9,2; \bar 7)$};


\node[right] (1-9-50-b3) at (6,3.875) {$\scriptscriptstyle (1,9,50; \bar 3)$};
\node[right] (1-9-50-b7) at (6,1.375) {$\scriptscriptstyle (1,9,50; \bar 7)$};

\node[right] (9-121-2-b3) at (6,-1.375) {$\scriptscriptstyle (9,121,2; \bar 3)$};
\node[right] (9-121-2-b7) at (6,-3.875) {$\scriptscriptstyle (9,121,2; \bar 7)$};


\node[right] (1-289-50-b3) at (10,4.375) {$\scriptscriptstyle (1,289,50; \bar 3)$};
\node[right] (1-289-50-b7) at (10,3.375) {$\scriptscriptstyle (1,289,50; \bar 7)$};

\node[right] (9-3481-50-b3) at (10,1.875) {$\scriptscriptstyle (9,3481,50; \bar 3)$};
\node[right] (9-3481-50-b7) at (10,0.875) {$\scriptscriptstyle (9,3481,50; \bar 7)$};

\node[right] (9-121-8450-b3) at (10,-0.875) {$\scriptscriptstyle (9,121,8450; \bar 3)$};
\node[right] (9-121-8450-b7) at (10,-1.875) {$\scriptscriptstyle (9,121,8450; \bar 7)$};

\node[right] (121-1681-2-b3) at (10,-3.375) {$\scriptscriptstyle (121,1681,2; \bar 3)$};
\node[right] (121-1681-2-b7) at (10,-4.375) {$\scriptscriptstyle (121,1681,2; \bar 7)$};


\node[right] (1-289-1682-b3) at (15,4.75) {$\scriptscriptstyle (1,289,1682; \bar 3)$};
\node[right] (1-289-1682-b7) at (15,4.25) {$\scriptscriptstyle (1,289,1682; \bar 7)$};

\node[right] (289-114921-50-b3) at (15,3.5) {$\scriptscriptstyle (289,114921,50; \bar 3)$};
\node[right] (289-114921-50-b7) at (15,3) {$\scriptscriptstyle (289,114921,50; \bar 7)$};

\node[right] (9-3481-243602-b3) at (15,2.25) {$\scriptscriptstyle (9,3481,243602; \bar 3)$};
\node[right] (9-3481-243602-b7) at (15,1.75) {$\scriptscriptstyle (9,3481,243602; \bar 7)$};

\node[right] (3481-1385329-50-b3) at (15,1) {$\scriptscriptstyle (3481,1385329,50; \bar 3)$};
\node[right] (3481-1385329-50-b7) at (15,0.5) {$\scriptscriptstyle (3481,1385329,50; \bar 7)$};

\node[right] (9-591361-8450-b3) at (15,-0.5) {$\scriptscriptstyle (9,591361,8450; \bar 3)$};
\node[right] (9-591361-8450-b7) at (15,-1) {$\scriptscriptstyle (9,591361,8450; \bar 7)$};

\node[right] (121-8162449-8450-b3) at (15,-1.75) {$\scriptscriptstyle (121,8162449,8450; \bar 3)$};
\node[right] (121-8162449-8450-b7) at (15,-2.25) {$\scriptscriptstyle (121,8162449,8450; \bar 7)$};

\node[right] (121-1681-1623602-b3) at (15,-3) {$\scriptscriptstyle (121,1681,1623602; \bar 3)$};
\node[right] (121-1681-1623602-b7) at (15,-3.5) {$\scriptscriptstyle (121,1681,1623602; \bar 7)$};

\node[right] (1681-23409-2-b3) at (15,-4.25) {$\scriptscriptstyle (1681,23409,2; \bar 3)$};
\node[right] (1681-23409-2-b7) at (15,-4.75) {$\scriptscriptstyle (1681,23409,2; \bar 7)$};


\draw[] (1-1-2-b3.east) edge (1-9-2-b3.west);
\draw[] (1-1-2-b3.east) edge (1-9-2-b7.west);

\draw[] (1-9-2-b3.east) edge (1-9-50-b3.west);
\draw[] (1-9-2-b7.east) edge (1-9-50-b7.west);
\draw[color=red] (1-9-2-b3.east) edge (9-121-2-b7.west);
\draw[color=red] (1-9-2-b7.east) edge (9-121-2-b3.west);

\draw[] (1-9-50-b3.east) edge (1-289-50-b3.west);
\draw[] (1-9-50-b7.east) edge (1-289-50-b7.west);
\draw[color=red] (1-9-50-b3.east) edge (9-3481-50-b7.west);
\draw[color=red] (1-9-50-b7.east) edge (9-3481-50-b3.west);

\draw[] (9-121-2-b3.east) edge (9-121-8450-b3.west);
\draw[] (9-121-2-b7.east) edge (9-121-8450-b7.west);
\draw[color=red] (9-121-2-b3.east) edge (121-1681-2-b7.west);
\draw[color=red] (9-121-2-b7.east) edge (121-1681-2-b3.west);

\draw[] (1-289-50-b3.east) edge (1-289-1682-b3.west);
\draw[] (1-289-50-b7.east) edge (1-289-1682-b7.west);
\draw[color=red] (1-289-50-b3.east) edge (289-114921-50-b7.west);
\draw[color=red] (1-289-50-b7.east) edge (289-114921-50-b3.west);

\draw[] (9-3481-50-b3.east) edge (9-3481-243602-b3.west);
\draw[] (9-3481-50-b7.east) edge (9-3481-243602-b7.west);
\draw[color=red] (9-3481-50-b3.east) edge (3481-1385329-50-b7.west);
\draw[color=red] (9-3481-50-b7.east) edge (3481-1385329-50-b3.west);

\draw[] (9-121-8450-b3.east) edge (9-591361-8450-b3.west);
\draw[] (9-121-8450-b7.east) edge (9-591361-8450-b7.west);
\draw[color=red] (9-121-8450-b3.east) edge (121-8162449-8450-b7.west);
\draw[color=red] (9-121-8450-b7.east) edge (121-8162449-8450-b3.west);

\draw[] (121-1681-2-b3.east) edge (121-1681-1623602-b3.west);
\draw[] (121-1681-2-b7.east) edge (121-1681-1623602-b7.west);
\draw[color=red]  (121-1681-2-b3.east) edge (1681-23409-2-b7.west);
\draw[color=red]  (121-1681-2-b7.east) edge (1681-23409-2-b3.west);

\end{tikzpicture}
\end{center}

\bigskip

\noindent
Similarly as before, we indicate the jumps between the involved
series by the red edges.
We expect the pattern to continue.
\end{remark}

\begin{remark}
The series (1-8-1) and (1-8-5) are neither connected to $T(1,8)$
via adjacencies nor to each other.
We obtain the following pattern of adjacencies.

\bigskip

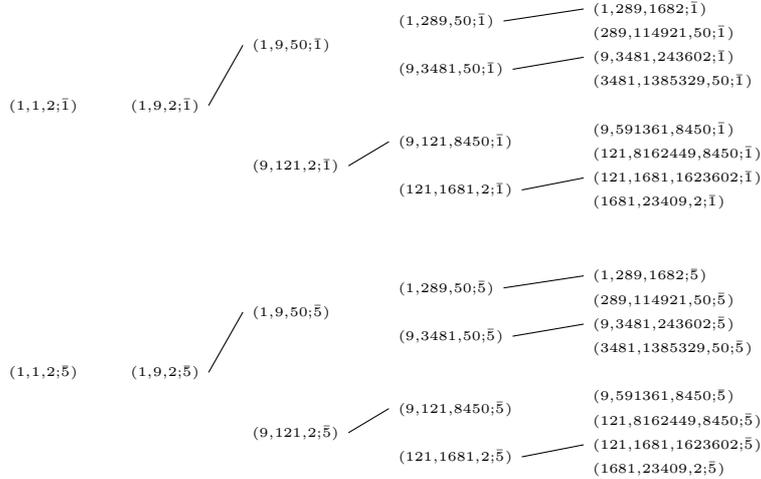
\begin{center}
\begin{tikzpicture}[scale=0.64]
\sffamily


\node[right] (1-1-2-b1) at (-1,0) {$\scriptscriptstyle (1,1,2; \bar 1)$};


\node[right] (1-9-2-b1) at (1.5,0) {$\scriptscriptstyle (1,9,2; \bar 1)$};


\node[right] (1-9-50-b1) at (4,1.25) {$\scriptscriptstyle (1,9,50; \bar 1)$};
\node[right] (9-121-2-b1) at (4,-1.25) {$\scriptscriptstyle (9,121,2; \bar 1)$};


\node[right] (1-289-50-b1) at (7,1.75) {$\scriptscriptstyle (1,289,50; \bar 1)$};
\node[right] (9-3481-50-b1) at (7,0.75) {$\scriptscriptstyle (9,3481,50; \bar 1)$};
\node[right] (9-121-8450-b1) at (7,-0.75) {$\scriptscriptstyle (9,121,8450; \bar 1)$};
\node[right] (121-1681-2-b1) at (7,-1.75) {$\scriptscriptstyle (121,1681,2; \bar 1)$};


\node[right] (1-289-1682-b1) at (11,2) {$\scriptscriptstyle (1,289,1682; \bar 1)$};
\node[right] (289-114921-50-b1) at (11,1.5) {$\scriptscriptstyle (289,114921,50; \bar 1)$};
\node[right] (9-3481-243602-b1) at (11,1) {$\scriptscriptstyle (9,3481,243602; \bar 1)$};
\node[right] (3481-1385329-50-b1) at (11,0.5) {$\scriptscriptstyle (3481,1385329,50; \bar 1)$};

\node[right] (9-591361-8450-b1) at (11,-0.5) {$\scriptscriptstyle (9,591361,8450; \bar 1)$};
\node[right] (121-8162449-8450-b1) at (11,-1) {$\scriptscriptstyle (121,8162449,8450; \bar 1)$};
\node[right] (121-1681-1623602-b1) at (11,-1.5) {$\scriptscriptstyle (121,1681,1623602; \bar 1)$};
\node[right] (1681-23409-2-b1) at (11,-2) {$\scriptscriptstyle (1681,23409,2; \bar 1)$};


\draw[]  (1-9-2-b1.east) edge (1-9-50-b1.west);

\draw[]  (9-121-2-b1.east) edge (9-121-8450-b1.west);

\draw[] (1-289-50-b1.east) edge (1-289-1682-b1.west);

\draw[] (9-3481-50-b1.east) edge (9-3481-243602-b1.west);

\draw[]  (121-1681-2-b1.east) edge (121-1681-1623602-b1.west);

\end{tikzpicture}
\end{center}

\bigskip

\begin{center}
\begin{tikzpicture}[scale=0.64]
\sffamily


\node[right] (1-1-2-b5) at (-1,0) {$\scriptscriptstyle (1,1,2; \bar 5)$};


\node[right] (1-9-2-b5) at (1.5,0) {$\scriptscriptstyle (1,9,2; \bar 5)$};


\node[right] (1-9-50-b5) at (4,1.25) {$\scriptscriptstyle (1,9,50; \bar 5)$};
\node[right] (9-121-2-b5) at (4,-1.25) {$\scriptscriptstyle (9,121,2; \bar 5)$};


\node[right] (1-289-50-b5) at (7,1.75) {$\scriptscriptstyle (1,289,50; \bar 5)$};
\node[right] (9-3481-50-b5) at (7,0.75) {$\scriptscriptstyle (9,3481,50; \bar 5)$};
\node[right] (9-121-8450-b5) at (7,-0.75) {$\scriptscriptstyle (9,121,8450; \bar 5)$};
\node[right] (121-1681-2-b5) at (7,-1.75) {$\scriptscriptstyle (121,1681,2; \bar 5)$};


\node[right] (1-289-1682-b5) at (11,2) {$\scriptscriptstyle (1,289,1682; \bar 5)$};
\node[right] (289-114921-50-b5) at (11,1.5) {$\scriptscriptstyle (289,114921,50; \bar 5)$};
\node[right] (9-3481-243602-b5) at (11,1) {$\scriptscriptstyle (9,3481,243602; \bar 5)$};
\node[right] (3481-1385329-50-b5) at (11,0.5) {$\scriptscriptstyle (3481,1385329,50; \bar 5)$};

\node[right] (9-591361-8450-b5) at (11,-0.5) {$\scriptscriptstyle (9,591361,8450; \bar 5)$};
\node[right] (121-8162449-8450-b5) at (11,-1) {$\scriptscriptstyle (121,8162449,8450; \bar 5)$};
\node[right] (121-1681-1623602-b5) at (11,-1.5) {$\scriptscriptstyle (121,1681,1623602; \bar 5)$};
\node[right] (1681-23409-2-b5) at (11,-2) {$\scriptscriptstyle (1681,23409,2; \bar 5)$};


\draw[]  (1-9-2-b5.east) edge (1-9-50-b5.west);

\draw[]  (9-121-2-b5.east) edge (9-121-8450-b5.west);

\draw[] (1-289-50-b5.east) edge (1-289-1682-b5.west);

\draw[] (9-3481-50-b5.east) edge (9-3481-243602-b5.west);

\draw[]  (121-1681-2-b5.east) edge (121-1681-1623602-b5.west);

\end{tikzpicture}
\end{center}

\end{remark}

\goodbreak

\begin{remark}
The series (1-6-1) is not connected to $T(1,6)$ by adjacency.
Inside (1-6-1), the pattern of adjacencies equals that of (2,3,1):
\bigskip

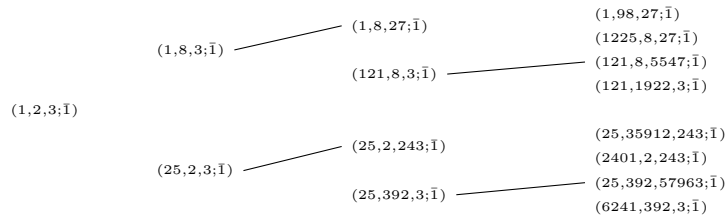
\begin{center}
  
\begin{tikzpicture}[scale=0.64]
\sffamily


\node[right] (1-2-3-b1) at (0,0) {$\scriptscriptstyle (1,2,3; \bar 1)$};


\node[right] (1-8-3-b1) at (3,1.25) {$\scriptscriptstyle (1,8,3; \bar 1)$};
\node[right] (25-2-3-b1) at (3,-1.25) {$\scriptscriptstyle (25,2,3; \bar 1)$};


\node[right] (1-8-27-b1) at (7,1.75) {$\scriptscriptstyle (1,8,27; \bar 1)$};
\node[right] (121-8-3-b1) at (7,0.75) {$\scriptscriptstyle (121,8,3; \bar 1)$};

\node[right] (25-2-243-b1) at (7,-0.75) {$\scriptscriptstyle (25,2,243; \bar 1)$};
\node[right] (25-392-3-b1) at (7,-1.75) {$\scriptscriptstyle (25,392,3; \bar 1)$};


\node[right] (1-98-27-b1) at (12,2) {$\scriptscriptstyle (1,98,27; \bar 1)$};
\node[right] (1225-8-27-b1) at (12,1.5) {$\scriptscriptstyle (1225,8,27; \bar 1)$};

\node[right] (121-8-5547-b1) at (12,1) {$\scriptscriptstyle (121,8,5547; \bar 1)$};
\node[right] (121-1922-3-b1) at (12,0.5) {$\scriptscriptstyle (121,1922,3; \bar 1)$};

\node[right] (25-35912-243-b1) at (12,-0.5) {$\scriptscriptstyle (25,35912,243; \bar 1)$};
\node[right] (196-4489-5-b1) at (12,-1) {$\scriptscriptstyle (2401,2,243; \bar 1)$};

\node[right] (25-392-57963-b1) at (12,-1.5) {$\scriptscriptstyle (25,392,57963; \bar 1)$};
\node[right] (6241-392-3-b1) at (12,-2) {$\scriptscriptstyle (6241,392,3; \bar 1)$};


\draw[] (1-8-3-b1.east) edge (1-8-27-b1.west);

\draw[] (25-2-3-b1.east) edge (25-2-243-b1.west);

\draw[] (121-8-3-b1.east) edge (121-8-5547-b1.west);

\draw[] (25-392-3-b1.east) edge (25-392-57963-b1.west);

\end{tikzpicture}
\end{center}

\end{remark}

\goodbreak

\begin{remark}
The series (1-5-1) is not connected to any other series (1-5-$*$)
via adjacencies. Inside (1-5-1), we have the following pattern
of adjacencies:
\bigskip

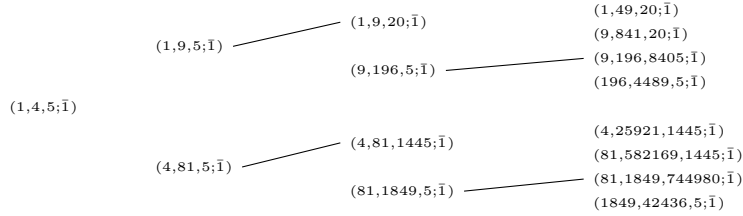
\begin{center}
  
\begin{tikzpicture}[scale=0.64]
\sffamily


\node[right] (1-4-5-b1) at (0,0) {$\scriptscriptstyle (1,4,5; \bar 1)$};


\node[right] (1-9-5-b1) at (3,1.25) {$\scriptscriptstyle (1,9,5; \bar 1)$};
\node[right] (4-81-5-b1) at (3,-1.25) {$\scriptscriptstyle (4,81,5; \bar 1)$};


\node[right] (1-9-20-b1) at (7,1.75) {$\scriptscriptstyle (1,9,20; \bar 1)$};
\node[right] (9-196-5-b1) at (7,0.75) {$\scriptscriptstyle (9,196,5; \bar 1)$};
\node[right] (4-81-1445-b1) at (7,-0.75) {$\scriptscriptstyle (4,81,1445; \bar 1)$};
\node[right] (81-1849-5-b1) at (7,-1.75) {$\scriptscriptstyle (81,1849,5; \bar 1)$};


\node[right] (1-49-20-b1) at (12,2) {$\scriptscriptstyle (1,49,20; \bar 1)$};

\node[right] (9-841-20-b1) at (12,1.5) {$\scriptscriptstyle (9,841,20; \bar 1)$};

\node[right] (9-196-8405-b1) at (12,1) {$\scriptscriptstyle (9,196,8405; \bar 1)$};

\node[right] (196-4489-5-b1) at (12,0.5) {$\scriptscriptstyle (196,4489,5; \bar 1)$};

\node[right] (4-1445-25921-b1) at (12,-0.5) {$\scriptscriptstyle (4,25921,1445; \bar 1)$};

\node[right] (81-582169-1445-b1) at (12,-1) {$\scriptscriptstyle (81,582169,1445; \bar 1)$};

\node[right] (81-1849-744980-b1) at (12,-1.5) {$\scriptscriptstyle (81,1849,744980; \bar 1)$};

\node[right] (1849-42436-5-b1) at (12,-2) {$\scriptscriptstyle (1849,42436,5; \bar 1)$};


\draw[] (1-9-5-b1.east) edge (1-9-20-b1.west);

\draw[] (4-81-5-b1.east) edge (4-81-1445-b1.west);

\draw[] (9-196-5-b1.east) edge (9-196-8405-b1.west);

\draw[] (81-1849-5-b1.east) edge (81-1849-744980-b1.west);

\end{tikzpicture}
\end{center}

\bigskip

\noindent
The series (1-5-2), (1-5-3) are not connected
to $T(1,5)$ by adjancency.
Inside these series, the adjacency pattern
resembles the previous one, but only jumps occur:

\bigskip

\begin{center}
  
\begin{tikzpicture}[scale=0.64]
\sffamily


\node[right] (1-4-5-b2) at (0,2) {$\scriptscriptstyle (1,4,5; \bar 2)$};
\node[right] (1-4-5-b3) at (0,-2) {$\scriptscriptstyle (1,4,5; \bar 3)$};


\node[right] (1-9-5-b2) at (3,3.875) {$\scriptscriptstyle (1,9,5; \bar 2)$};
\node[right] (1-9-5-b3) at (3,1.375) {$\scriptscriptstyle (1,9,5; \bar 3)$};

\node[right] (4-81-5-b2) at (3,-1.375) {$\scriptscriptstyle (4,81,5; \bar 2)$};
\node[right] (4-81-5-b3) at (3,-3.875) {$\scriptscriptstyle (4,81,5; \bar 3)$};


\node[right] (1-9-20-b2) at (7,4.375) {$\scriptscriptstyle (1,9,20; \bar 2)$};
\node[right] (1-9-20-b3) at (7,3.375) {$\scriptscriptstyle (1,9,20; \bar 3)$};

\node[right] (9-196-5-b2) at (7,1.875) {$\scriptscriptstyle (9,196,5; \bar 2)$};
\node[right] (9-196-5-b3) at (7,0.875) {$\scriptscriptstyle (9,196,5; \bar 3)$};

\node[right] (4-81-1445-b2) at (7,-0.875) {$\scriptscriptstyle (4,81,1445; \bar 2)$};
\node[right] (4-81-1445-b3) at (7,-1.875) {$\scriptscriptstyle (4,81,1445; \bar 3)$};

\node[right] (81-1849-5-b2) at (7,-3.375) {$\scriptscriptstyle (81,1849,5; \bar 2)$};
\node[right] (81-1849-5-b3) at (7,-4.375) {$\scriptscriptstyle (81,1849,5; \bar 3)$};


\node[right] (1-49-20-b2) at (12,4.75) {$\scriptscriptstyle (1,49,20; \bar 2)$};
\node[right] (1-49-20-b3) at (12,4.25) {$\scriptscriptstyle (1,49,20; \bar 3)$};

\node[right] (9-841-20-b2) at (12,3.5) {$\scriptscriptstyle (9,841,20; \bar 2)$};
\node[right] (9-841-20-b3) at (12,3) {$\scriptscriptstyle (9,841,20; \bar 3)$};

\node[right] (9-196-8405-b2) at (12,2.25) {$\scriptscriptstyle (9,196,8405; \bar 2)$};
\node[right] (9-196-8405-b3) at (12,1.75) {$\scriptscriptstyle (9,196,8405; \bar 3)$};

\node[right] (196-4489-5-b2) at (12,1) {$\scriptscriptstyle (196,4489,5; \bar 2)$};
\node[right] (196-4489-5-b3) at (12,0.5) {$\scriptscriptstyle (196,4489,5; \bar 3)$};

\node[right] (4-1445-25921-b2) at (12,-0.5) {$\scriptscriptstyle (4,25921,1445; \bar 2)$};
\node[right] (4-1445-25921-b3) at (12,-1) {$\scriptscriptstyle (4,25921,1445; \bar 3)$};

\node[right] (81-1445-582169-b2) at (12,-1.75) {$\scriptscriptstyle (81,582169,1445; \bar 2)$};
\node[right] (81-1445-582169-b3) at (12,-2.25) {$\scriptscriptstyle (81,582169,1445; \bar 3)$};

\node[right] (81-1849-744980-b2) at (12,-3) {$\scriptscriptstyle (81,1849,744980; \bar 2)$};
\node[right] (81-1849-744980-b3) at (12,-3.5) {$\scriptscriptstyle (81,1849,744980; \bar 3)$};

\node[right] (1849-42436-5-b2) at (12,-4.25) {$\scriptscriptstyle (1849,42436,5; \bar 2)$};
\node[right] (1849-42436-5-b3) at (12,-4.75) {$\scriptscriptstyle (1849,42436,5; \bar 3)$};


\draw[color=red] (1-4-5-b2) edge (1-4-5-b3);

\draw[color=red] (1-9-5-b2.east) edge (1-9-20-b3.west);
\draw[color=red] (1-9-5-b3.east) edge (1-9-20-b2.west);

\draw[color=red] (4-81-5-b2.east) edge (4-81-1445-b3.west);
\draw[color=red] (4-81-5-b3.east) edge (4-81-1445-b2.west);

\draw[color=red] (9-196-5-b2.east) edge (9-196-8405-b3.west);
\draw[color=red] (9-196-5-b3.east) edge (9-196-8405-b2.west);

\draw[color=red] (81-1849-5-b2.east) edge (81-1849-744980-b3.west);
\draw[color=red] (81-1849-5-b3.east) edge (81-1849-744980-b2.west);

\end{tikzpicture}
\end{center}

\end{remark}

\begin{remark}
Note that Construction~\ref{constr:adj-graphs} does not
take ``self-adjacency'' into account, that means fake
weighted projective planes $Z$ of integral degree that
are adjacent to themselves.
This happens precisely in the following 16 cases:
\[
\begin{array}{ll}
u = (1,1,2): & \text{(8-1-0), (4-2-1), (2-4-1), (2-4-3), (1-8-1), (1-8-3), (1-8-5),}
\\[3pt]
u = (1,2,3): & \text{(6-1-0), (3-2-1), (2-3-1), (2-3-2), (1-6-1), (1-6-5),}
\\[3pt]
u = (1,2,3): & \text{(5-1-0), (1-5-1), (1-5-4),}
\end{array}
\]
where $u$ denotes the first row of the degree matrix.
Among these cases, we have $Z \reflectbox{$\leadsto$} X \hbox{$\leadsto$} Z$
with a non-toric $X$ from Construction~\ref{constr:degs}
in precisely 6 cases, namely
\[
\text{(2-4-3), \ (1-8-3), \ (1-8-5), \ (1-6-5), \ (1-5-1), \ (1-5-4).}
\]
\end{remark}

\begin{thanx}
We would like to thank Martin Bohnert for many inspiring and
enjoyable discussions.
\end{thanx}



\begin{bibdiv}
\begin{biblist}

  

  
\bib{AkKa}{article}{
   author={Akhtar, Mohammad E.},
   author={Kasprzyk, Alexander M.},
   title={Mutations of fake weighted projective planes},
   journal={Proc. Edinb. Math. Soc. (2)},
   volume={59},
   date={2016},
   number={2},
   pages={271--285},
   issn={0013-0915},
 }

 
\bib{ArDeHaLa}{book}{
   author={Arzhantsev, Ivan},
   author={Derenthal, Ulrich},
   author={Hausen, J\"urgen},
   author={Laface, Antonio},
   title={Cox rings},
   series={Cambridge Studies in Advanced Mathematics},
   volume={144},
   publisher={Cambridge University Press, Cambridge},
   date={2015},
   pages={viii+530},
}

 





 \bib{GyMa}{article}{
   author={Gyoda, Yasuaki},
   author={Matsushita, Kodai},
   title={Generalization of Markov Diophantine equation via generalized
   cluster algebra},
   journal={Electron. J. Combin.},
   volume={30},
   date={2023},
   number={4},
   pages={Paper No. 4.10, 20},
} 


\bib{HaPro}{article}{
   author={Hacking, Paul},
   author={Prokhorov, Yuri},
   title={Smoothable del Pezzo surfaces with quotient singularities},
   journal={Compos. Math.},
   volume={146},
   date={2010},
   number={1},
   pages={169--192},
   issn={0010-437X},
}


\bib{HaHaHaSp}{article}{
   author = {H\"{a}ttig, Daniel},
   author = {Hafner, Beatrix},
   author = {Hausen, J\"{u}rgen},
   author = {Springer, Justus},
   title = {Del Pezzo surfaces of Picard number one admitting a torus action},
   year = {2022},
   eprint={arXiv:2207.14790},
 }


\bib{HaHaSp}{article}{
  author = {H\"{a}ttig, Daniel},
  author = {Hausen, J\"{u}rgen},
  author = {Springer, Justus},
  title = {Classifying log del Pezzo surfaces with torus action},
  year = {2023},
  eprint={arXiv:2302.03095},
}



\bib{HaHaSu}{article}{
  author = {H\"{a}ttig, Daniel},
  author = {Hausen, J\"{u}rgen},
  author = {S\"{u}\ss, Hendrik},
  title={Log del Pezzo $\CC^*$-surfaces, K\"{a}hler-Einstein metrics,
     K\"{a}hler-Ricci solitons and Sasaki-Einstein metrics},
  journal={To appear in Michigan Math. J.},
  eprint={arXiv:2306.03796},
  }

\bib{HaKiWr}{article}{
  author = {Hausen, J\"{u}rgen},
  author = {Kir\'{a}ly, Katharina},
  author = {Wrobel, Milena},
  title={Markov numbers and rational $\CC^*$-surfaces},
  eprint={arXiv:2405.04862},
  }

\bib{HaWr}{article}{
   author={Hausen, J\"urgen},
   author={Wrobel, Milena},
   title={Non-complete rational $T$-varieties of complexity one},
   journal={Math. Nachr.},
   volume={290},
   date={2017},
   number={5-6},
   pages={815--826},
   issn={0025-584X},
} 

\bib{Hu}{article}{
  author = {Hurwitz, Adolf},
  title={\"Uber eine Aufgabe der unbestimmten Analysis},
  journal={In: Mathematische Werke, Band~II: Zahlentheorie, Algebra und Geometrie},
  pages={410--421},
  publisher={Springer, Basel},
  date={1963},
  }

\bib{Il}{article}{
   author={Ilten, Nathan Owen},
   title={Mutations of Laurent polynomials and flat families with toric
   fibers},
   journal={SIGMA Symmetry Integrability Geom. Methods Appl.},
   volume={8},
   date={2012},
   pages={Paper 047, 7},
}

\bib{KaNo}{article}{
   author={Karpov, B. V.},
   author={Nogin, D. Yu.},
   title={Three-block exceptional sets on del Pezzo surfaces},
   language={Russian, with Russian summary},
   journal={Izv. Ross. Akad. Nauk Ser. Mat.},
   volume={62},
   date={1998},
   number={3},
   pages={3--38},
   issn={1607-0046},
   translation={
      journal={Izv. Math.},
      volume={62},
      date={1998},
      number={3},
      pages={429--463},
      issn={1064-5632},
   },
}







 


\end{biblist}
\end{bibdiv}
\end{document}